\documentclass[12pt,reqno]{amsart}

\usepackage[margin=0.8in]{geometry}
\usepackage{amsmath}
\usepackage{amssymb}
\usepackage{amsthm}
\usepackage{dsfont}
\usepackage{graphicx} 
\usepackage{microtype}

\usepackage[dvipsnames]{xcolor}
\usepackage{hyperref}
\hypersetup{
    colorlinks=true,
    linkcolor=cyan!80!black,
    citecolor=MidnightBlue,
    urlcolor=magenta,
}

\usepackage{tikz}

\usepackage{apptools}
\AtAppendix{\counterwithin{lemma}{section}}

\newcommand{\bbE}{\mathbb{E}}
\newcommand{\bbC}{\mathbb{C}}
\newcommand{\bbN}{\mathbb{N}}
\newcommand{\bbP}{\mathbb{P}}
\newcommand{\bbR}{\mathbb{R}}
\newcommand{\bbS}{\mathbb{S}}
\newcommand{\bbZ}{\mathbb{Z}}

\newcommand{\bX}{\mathbf{X}}
\newcommand{\bY}{\mathbf{Y}}
\newcommand{\bZ}{\mathbf{Z}}

\newcommand{\bV}{\mathbf{V}}

\newcommand{\ba}{\mathbf{a}}
\newcommand{\bg}{\mathbf{g}}
\newcommand{\bx}{\mathbf{x}}

\newcommand{\bv}{\mathbf{v}}
\newcommand{\bw}{\mathbf{w}}
\newcommand{\bu}{\mathbf{u}}

\newcommand{\bone}{\mathbf{1}}

\newcommand{\cA}{\mathcal{A}}
\newcommand{\cH}{\mathcal{H}}

\newcommand{\cP}{\mathcal{P}}
\newcommand{\cF}{\mathcal{F}}
\newcommand{\cO}{\mathcal{O}}
\newcommand{\cS}{\mathcal{S}}

\newcommand{\cM}{\mathcal{M}}

\newcommand{\cY}{\mathcal{Y}}
\newcommand{\cG}{\mathcal{G}}

\newcommand{\Ber}{\mathsf{Ber}}
\newcommand{\Hyper}{\mathsf{Hypergeometric}}
\newcommand{\BL}{\mathrm{BL}}
\newcommand{\KS}{\mathrm{KS}}
\newcommand{\levy}{\text{L\'evy}}
\newcommand{\PI}{\mathrm{PI}}
\newcommand{\LSI}{\mathrm{LSI}}

\newcommand{\ind}{\mathds{1}}
\newcommand{\Mat}{\mathrm{Mat}}
\newcommand{\Lip}{\mathrm{Lip}}
\renewcommand{\sc}{\mathsf{sc}}
\newcommand{\spec}{\mathrm{spec}}
\newcommand{\Gauss}{\mathsf{G}}
\newcommand{\avg}{\mathrm{avg}}
\newcommand{\op}{\mathrm{op}}
\newcommand{\mubar}{\bar{\mu}}
\newcommand{\diag}{\mathrm{diag}}

\DeclareMathOperator{\rank}{rank}
\DeclareMathOperator{\Var}{Var}
\DeclareMathOperator{\Tr}{Tr}

\DeclareMathOperator{\Cov}{Cov}

\newcommand{\convas}{\xrightarrow{\text{a.s.}}}
\newcommand{\convp}{\xrightarrow{p}}
\newcommand{\convd}{\xrightarrow{d}}

\theoremstyle{plain}
\newtheorem{theorem}{Theorem}[section]
\newtheorem{corollary}{Corollary}[section]
\newtheorem{proposition}{Proposition}[section]
\newtheorem{lemma}{Lemma}[section]

\theoremstyle{definition}
\newtheorem{assumption}{Assumption}[section]
\newtheorem{definition}{Definition}[section]
\newtheorem{example}{Example}[section]

\theoremstyle{remark}
\newtheorem{remark}{Remark}[section]

\let\hat\widehat
\let\tilde\widetilde

\title[Spectra of Erd\H{o}s-R\'{e}nyi hypergraphs]{Spectra of adjacency and Laplacian matrices of Erd\H{o}s-R\'{e}nyi hypergraphs}
\author[S. S. Mukherjee, D. Pal and H. Talukdar]{Soumendu Sundar Mukherjee, Dipranjan Pal and Himasish Talukdar}
\address{
    Statistics and Mathematics Unit \\
    Indian Statistical Institute \\
    203 B.T. Road, Kolkata 700108 \\
    West Bengal, India.
}
\email{ssmukherjee@isical.ac.in, dipranjanpal064@gmail.com, talukdar.himasish@gmail.com}

\begin{document}

\begin{abstract}
We study adjacency and Laplacian matrices of Erd\H{o}s-R\'{e}nyi $r$-uniform hypergraphs on $n$ vertices with hyperedge inclusion probability $p$, in the setting where $r$ can vary with $n$ such that $r / n \to c \in [0, 1)$. Adjacency matrices of hypergraphs are contractions of adjacency tensors and their entries exhibit long range correlations. We show that under the  Erd\H{o}s-R\'{e}nyi model, the expected empirical spectral distribution of an appropriately normalised hypergraph adjacency matrix converges weakly to the semi-circle law with variance $(1 - c)^2$ as long as $\frac{d_{\avg}}{r^7} \to \infty$, where $d_{\avg} = \binom{n-1}{r-1} p$. In contrast with the Erd\H{o}s-R\'{e}nyi random graph ($r = 2$), two eigenvalues stick out of the bulk of the spectrum. When $r$ is fixed and $d_{\avg} \gg n^{r - 2} \log^4 n$, we uncover an interesting Baik-Ben Arous-P\'{e}ch\'{e} (BBP) phase transition at the value $r = 3$. For $r \in \{2, 3\}$, an appropriately scaled largest (resp. smallest) eigenvalue converges in probability to $2$ (resp. $-2$), the right (resp. left) end point of the support of the standard semi-circle law, and when $r \ge 4$, it converges to $\sqrt{r - 2} + \frac{1}{\sqrt{r - 2}}$ (resp. $-\sqrt{r - 2} - \frac{1}{\sqrt{r - 2}}$). Further, in a Gaussian version of the model we show that an appropriately scaled largest (resp. smallest) eigenvalue converges in distribution to $\frac{c}{2} \zeta + \big[\frac{c^2}{4}\zeta^2 + c(1 - c)\big]^{1/2}$ (resp. $\frac{c}{2} \zeta - \big[\frac{c^2}{4}\zeta^2 + c(1 - c)\big]^{1/2}$), where $\zeta$ is a standard Gaussian. We also establish analogous results for the bulk and edge eigenvalues of the associated Laplacian matrices.
\end{abstract}

\maketitle

\section{Introduction}
\subsection{Hypergraphs and associated matrices and tensors}
A hypergraph $\cH = (V, E)$ consists of a vertex set $V$ and a set $E \subseteq 2^V$ of \emph{ hyperedges}. If $|e| = r$ for all $e \in E$, then $\cH$ is called \emph{$r$-uniform}. If hyperedges of different sizes exist, then it is called \emph{non-uniform}. Usually $V$ is taken to be the set $[n] := \{1, \ldots, n \}$. Hypergraphs are generalisations of graphs (indeed, a simple undirected graph is just a $2$-uniform hypergraph) and are very useful for modelling higher-order interactions in various types of complex networks. In particular, hypergraphs have been used for community detection in networks \cite{ghoshdastidar2014consistency, ghoshdastidar2017consistency, pal2021community, dumitriu2021partial, dumitriu2023exact}, in biology \cite{tian2009hypergraph, michoel2012alignment}, for modelling chemical reactions \cite{skvortsova2014hypergraph, flamm2015generalized, mann2023ai}, in modelling citation networks \cite{ji2016coauthorship}, in recommendation systems \cite{bu2010music} and for processing image data \cite{govindu2005tensor}, among other areas.

Adjacency matrices of graphs naturally generalise to \emph{adjacency tensors} for hypergraphs. Let $\cH$ be an $r$-uniform hypergraph, where $r \ge 3$ is an integer. The adjacency tensor $\mathcal{A} \in \{0,1\}^{n^r}$ associated with $\cH$ is defined as
\begin{align*}
   \cA_{i_1, i_2, \cdots, i_r} &= \begin{cases}
       1  & \text{if} \ \ \{i_1, i_2, \ldots, i_r\} \in E, \\
       0  & \text{otherwise},
   \end{cases} 
\end{align*}
and $\cA_{i_1, i_2, \cdots, i_r} = \cA_{\sigma(i_1), \sigma(i_2), \cdots, \sigma(i_r)}$ for any $\sigma \in \cS_n$, where $\cS_n$ is the set of all permutations of $[n]$. This is a tensor of order $r$ and dimension $n$. In this article, we consider an adjacency matrix $A$ associated to $\cH$, which is a contraction of the adjacency tensor $\cA$ and is defined as follows:
\begin{equation}\label{eq:adj_mx_defn_via_contraction}
    A_{ij} = \frac{1}{(r - 2)!}\sum_{i_3, i_4, \cdots, i_r} \cA_{i, j, i_3, i_4 \cdots, i_r}.
\end{equation}
Observe that
\begin{equation}\label{eq:adj_mx_defn}
    A_{ij} = \begin{cases}
        \sum_{e \in E} \ind(i, j \in e) & \text{if } i \ne j, \\
        0 & \text{otherwise}.
    \end{cases}
\end{equation}
We note that different aspects of the matrix $A$ have been studied in the literature \cite{lee2020robust, banerjee2021spectrum}.

Let $M = \binom{n}{r}$ and $\{e_1, \ldots, e_M\}$ be an enumeration of the hyperedges of the complete $r$-uniform hypergraph on $n$ vertices. Consider the following matrices
\[
    Q_\ell := \ba_{\ell} \ba_{\ell}^\top - \diag(\ba_\ell), \,\, 1 \le \ell \le M,
\]
where
\[
    (\ba_{\ell})_i = \begin{cases}
    1 & \text{if} \ \ i \in e_\ell, \\ 
    0 & \text{otherwise},
\end{cases}
\]
and for a vector $\bv$, $\diag(\bv)$ denotes the diagonal matrix whose diagonal entries are given by $\bv$. Note then that $A$ admits the following representation:
\begin{equation}\label{eq:GHAM_repr_lin_comb}
    A = \sum_{\ell = 1}^M h_{\ell} \, Q_\ell,
\end{equation}
where $h_{\ell}= \ind(e_\ell \in E)$ indicates if hyperedge $e_{\ell}$ is present in the hypergraph $\cH$. 

We also consider the associated Laplacian matrices. Recall that for $r = 2$, i.e. when $A$ is the adjacency matrix of a simple undirected graph, the associated (combinatorial) graph Laplacian is defined as
\[
    L_A = \diag(A\bone) - A.
\]
Where $\bone$ is an $n \times 1$ vector with all entries $1$. More generally, for any matrix $X$, define the associated Laplacian matrix as
\begin{equation}\label{eq:lap_orig}
    L_X := \diag(X \bone) - X.
\end{equation}
We also consider the variant
\begin{equation}\label{eq:lap_modified}
    \tilde{L}_X := \frac{1}{r - 1}\diag(X \bone) - X.
\end{equation}
Note that for $r = 2$, $L_X = \tilde{L}_X$, as such both of these matrices generalise the graph Laplacian.

\subsection{Preliminaries on random matrices} Let $A_n$ be an $n \times n$ Hermitian matrix with ordered eigenvalues $\lambda_1 \ge \cdots \ge \lambda_n$. The probability measure 
\[
    \mu_{A_n} := \frac{1}{n} \sum_{i = 1}^n \delta_{\lambda_i}
\]
is called the \emph{Empirical Spectral Distribution} (ESD) of $A_n$. If entries of $A_n$ are random variables defined on a common probability space $(\Omega, \cA, \bbP)$ then $\mu_{A_n}$ is a random probability measure. In that case, there is another probability measure associated to the eigenvalues, namely the \textit{Expected Empirical Spectral Distribution} (EESD) of $A_n$, which is defined via its action on bounded measurable test functions $f$ as follows:
\[
    \int f \, d\mubar_{A_n} = \bbE \int f \, d\mu_{A_n},
\]
where $\bbE$ denotes expectation with respect to $\bbP$. In random matrix theory, one is typically interested in an ensemble $(A_n)_{n \ge 1}$ of such matrices of growing dimension $n$. If the weak limit, say $\mu_{\infty}$, of the sequence $(\mubar_{A_n})_{n \ge 1}$, exists, then it is referred to as the Limiting Spectral Distribution (LSD). Often one is able to show that the random measure $\mu_{A_n}$ also converges weakly (in probability or in almost sure sense) to $\mu_{\infty}$. For a comprehensive introductory account of the theory of random matrices, we refer the reader to \cite{anderson2010introduction}.

The preeminent model of random matrices is perhaps the Wigner matrix. For us a Wigner matrix $W_n$ will be a Hermitian random matrix whose upper triangular entries $W_{n, i, j}$ are i.i.d. zero mean unit variance random variables and the diagonal entries $W_{n, i, i}$ are i.i.d. zero mean random variables with finite variance. Moreover, the diagonal and the off-diagonal entries are mutually independent. If the entries are jointly Gaussian with the diagonal entries having variance $2$, then resulting ensemble of matrices is called the Gaussian Orthogonal Ensemble (GOE). In this article, we will denote the centered Gaussian distribution with variance $\sigma^2$ by $\nu_{\Gauss, \sigma^2}$.

We also define the \emph{semi-circle distribution with variance $\sigma^2$}, henceforth denoted by $\nu_{\sc, \sigma^2}$, as the probability distribution on $\bbR$ with density
\[
    f(x) := \begin{cases}
    \frac{1}{2\pi \sigma^2} \sqrt{4 \sigma^2 - x^2} & \text{if} \, |x| \le 2 \sigma, \\
    0  & \text{otherwise}. 
\end{cases}
\]
E. Wigner proved in his famous paper \cite{wigner1958distribution} that the EESD of $n^{-1/2} W_n$ converges weakly to the standard semi-circle distribution $\nu_{\sc, 1}$. It is well known that $2k$-th moment of the standard semi-circle distribution is the $k$-th \emph{Catalan number}, defined as
\[
    C_k := \frac{1}{k + 1} \binom{2k}{k}.
\]
Catalan numbers have many interesting combinatorial interpretations. Most notably, $C_k$ counts the number of \emph{Dyck paths} of length $2k$, i.e. the number of non-negative Bernoulli walks of length $2k$ that both start and end at the origin.

\subsection{Note on asymptotic notation}
Before proceeding further, we define here various asymptotic notations used throughout the paper. For functions $f, g : \bbN \to \bbR$, we write (i) $f(n) = O(g(n))$, if there exist positive constants $n_0$ and $C$ such that $|f(n)| \le C |g(n)|$ for all $n \ge n_0$; (ii) $f(n) = o(g(n))$ or $f(n) \ll g(n)$ if $\lim_{n \to \infty} \frac{f(n)}{g(n)} = 0$ (we also write $f(n) \gg g(n)$ if $g(n) \ll f(n)$); (iii) $f(n) = \Theta(g(n))$ or $f(n) \asymp g(n)$ if $f(n) = O(g(n))$ and $g(n) = O(f(n))$; (iv) $f(n) \sim g(n)$ if $\lim_{n \to \infty} \frac{f(n)}{g(n)} = 1$; (v) $f(n) \gg g(n)$ if $\lim_{n \to \infty} \frac{g(n)}{f(n)} = 0$.

For a sequence of random variables $\{X_n\}_{n \ge 1}$, we write $X_n = O_P(1)$ if for any $\epsilon > 0$, there exists $M > 0$ such that $\sup_n \bbP(|X_n| > M) \le \epsilon$. For two sequence of random variables $\{X_n\}_{n \ge 1}$ and $\{Y_n\}_{n \ge 1}$ we write $X_n = O(Y_n)$ to mean $X_n = Z_n Y_n $ with $Z_n = O_P(1)$. 

\subsection{Our random matrix model} In this article, our main objective is to study the ESDs of adjacency matrices of \emph{Erd\H{o}s-R\'{e}nyi random $r$-uniform hypergraphs}, where each hyperedge is included independently with probability $p$, i.e. $h_{\ell} \overset{\text{i.i.d.}}{\sim} \Ber(p)$. Towards that end, consider the centered (and normalised) matrix
\[
    \tilde{A} = \frac{A - \bbE[A]}{\sqrt{p(1 - p)}}= \sum_{\ell=1}^{M} \frac{(h_{\ell} - \bbE[h_{\ell}])}{\sqrt{\Var(h_{\ell})}} \, Q_\ell = \sum_{\ell = 1}^{M} Y_{\ell} \, Q_{\ell}.
\]
Note that $Y_{\ell}$'s are i.i.d. zero mean unit variance random variables. Thus
\[
    \tilde{A}_{ij} = \begin{cases}
    \sum_{\ell = 1}^{M} Y_{\ell} \cdot \ind(i, j \in e_{\ell}) & \text{if } i \ne j, \\
    0 & \text{otherwise.}
    \end{cases}
\]
Note that
\[
    \Var( \tilde{A}_{ij} ) = \Var\bigg( \sum_{\ell = 1}^{M} Y_{\ell} \cdot \ind(i, j \in e_{\ell})\bigg) = \sum_{\ell = 1}^{M} \Var(Y_{\ell}) \ind(i, j \in e_{\ell}) = \binom{n-2}{r-2}.
\]
Therefore, we shall consider the following normalised matrix (viewed as a matrix-valued function of the vector $\bY = (Y_{\ell})_{1\le \ell \le M}$)
\begin{equation}\label{eq:gham_defn}
    H_n \equiv H_n(\bY) := \frac{\tilde{A}}{\sqrt{N}} = \frac{1}{\sqrt{N}}\sum_{\ell = 1}^{M} Y_{\ell} \, Q_{\ell},
\end{equation}
where $N = \binom{n - 2}{r - 2}$, so that the all off-diagonal entries of $H_n$ have zero mean and unit variance. This normalisation ensures that the variance of the EESD of $H_n$ is of order $n^2$ (comparable to that of a usual Wigner matrix).

\begin{definition}
Let $\bY = (Y_{\ell})_{1 \le \ell \le M}$ be independent zero mean unit variance random variables. The matrix $H_n = H_n(\bY)$ defined in \eqref{eq:gham_defn} will be referred to as a \emph{Generalised Hypergraph Adjacency Matrix} (GHAM). 
\end{definition}
A principal feature of this ensemble is that the entries show long-range correlation. Indeed,
\begin{align*}
\Cov(\tilde{A}_{ij}, \tilde{A}_{i'j'}) 
 &= \sum_{\ell = 1}^{M} 
 \sum_{\ell' = 1}^{M} \Cov(Y_{\ell},  Y_{\ell'}) \, \ind(i, j \in e_{\ell}) \ind(i', j' \in e_{\ell'}) \\
 &=  \sum_{\ell = 1}^{M} \Var(Y_{\ell}) \, \ind(i, j \in e_{\ell}) \ind(i', j' \in e_{\ell}) \\
 &= \begin{cases}
     \binom{n-4}{r-4} & \text{if} \ \ \big|\{i, j\} \cap \{i', j' \}\big| = 0, \\
     \binom{n-3}{r-3} & \text{if} \ \ \big|\{i, j\} \cap \{i', j' \}\big| = 1.
 \end{cases}
\end{align*}
Therefore
\begin{align*}
   \Cov(H_{n, ij}, H_{n, i^{'}j^{'}}) &= \begin{cases}
       \frac{ \binom{n-4}{r-4} }{\binom{n-2}{r-2}} = \frac{(r-2)(r-3)}{(n-2)(n-3)} =: \rho_n & \text{if} \ \ \big|\{i, j\} \cap \{i^{'}, j^{'} \}\big| = 0, \\
       \frac{ \binom{n-3}{r-3} }{\binom{n-2}{r-2}} = \frac{r-2}{n-2} =: \gamma_n & \text{if} \ \ \big|\{i, j\} \cap \{i^{'}, j^{'} \}\big| = 1.
   \end{cases}
\end{align*}

In this article, we are interested in the spectrum of $n^{-1/2} H_n$ in the regime where $r$ grows with $n$ in such a way that $\frac{r}{n} \to c \in [0, 1)$. Our main result regarding the bulk spectrum of $n^{-1/2} H_n$ is the following.
\begin{theorem}[LSD of the adjacency matrix]\label{thm:main}
    Suppose the entries $(Y_{\ell})_{1 \le \ell \le M}$ satisfy the Pastur-type condition given in Assumption~\ref{ass:tail_of_entries}. Then the EESD of $n^{-1/2} H_n$ converges weakly to $\nu_{\sc, (1 - c)^2}$.
\end{theorem}
For two probability measures $\mu$ an $\nu$, let $\mu \boxplus \nu$ denote their free additive convolution (defined in Section~\ref{sec:lsd_gaussian_GHAM}).
\begin{theorem}[LSDs for Laplacians]\label{thm:main_laplacian}
~ 
Suppose the entries $(Y_{\ell})_{1 \le \ell \le M}$ satisfy the Pastur-type condition given in Assumption~\ref{ass:tail_of_entries}.
   \begin{enumerate}
       \item [(i)] Suppose that $r$ is fixed. Then the EESD of $n^{-1/2}L_{H_n}$ converges weakly to $\nu_{\Gauss, r - 1} \boxplus \nu_{\sc, 1}$ and the EESD of $n^{-1/2}\tilde{L}_{H_n}$ converges weakly to $\nu_{\Gauss,\frac{1}{r - 1}} \boxplus \nu_{\sc, 1}$. 
       \item [(ii)] Suppose $r \to \infty$ such that $r/n \to c \in [0, 1)$.
       Then the EESD of $(nr)^{-1/2} L_{H_n}$ converges weakly to $\nu_{\Gauss, 1} \boxplus \nu_{\Gauss, c}$ and the EESD of $n^{-1/2}\tilde{L}_{H_n}$ converges weakly to $\nu_{\sc, (1 - c)^2}$.
\end{enumerate}
In fact, for $L_{H_n}$ we only need the weaker Assumption~\ref{ass:tail_of_entries_variation} on the entries.
\end{theorem}

Theorems~\ref{thm:main} and \ref{thm:main_laplacian} follow from a universality result (see Theorem~\ref{thm:univ}) and an analysis of the Gaussian case, where $(Y_{\ell})_{1 \le \ell \le M}$ are i.i.d. standard Gaussians (see Theorems~\ref{thm:eesd_gaussian} and \ref{thm:eesd_gaussian_laplacian}).
In the regime $r/n \to 0$, we also prove effective concentration inequalities for the ESD (see Corollary~\ref{cor:conc_esd})  assuming further conditions on the entries $Y_{\ell}$, which enable us to show in-probability weak convergence of the ESDs (see Corollary~\ref{cor:conv_esd}). Further, if $r \sqrt{\log n} / n \to 0$, we have almost sure weak convergence of the ESDs. The question of almost sure weak convergence of the ESDs in the full regime $r / n \to c \in [0, 1)$ will be considered in a future work. In Figure~\ref{fig:esd_gaussian_model}, we show the ESDs of $300 \times 300$ GHAMs with Gaussian entries for various choices of $r$. 

\begin{figure}[!t]
    \centering
    \begin{tabular}{ccc}
        \includegraphics[width = 0.3\textwidth]{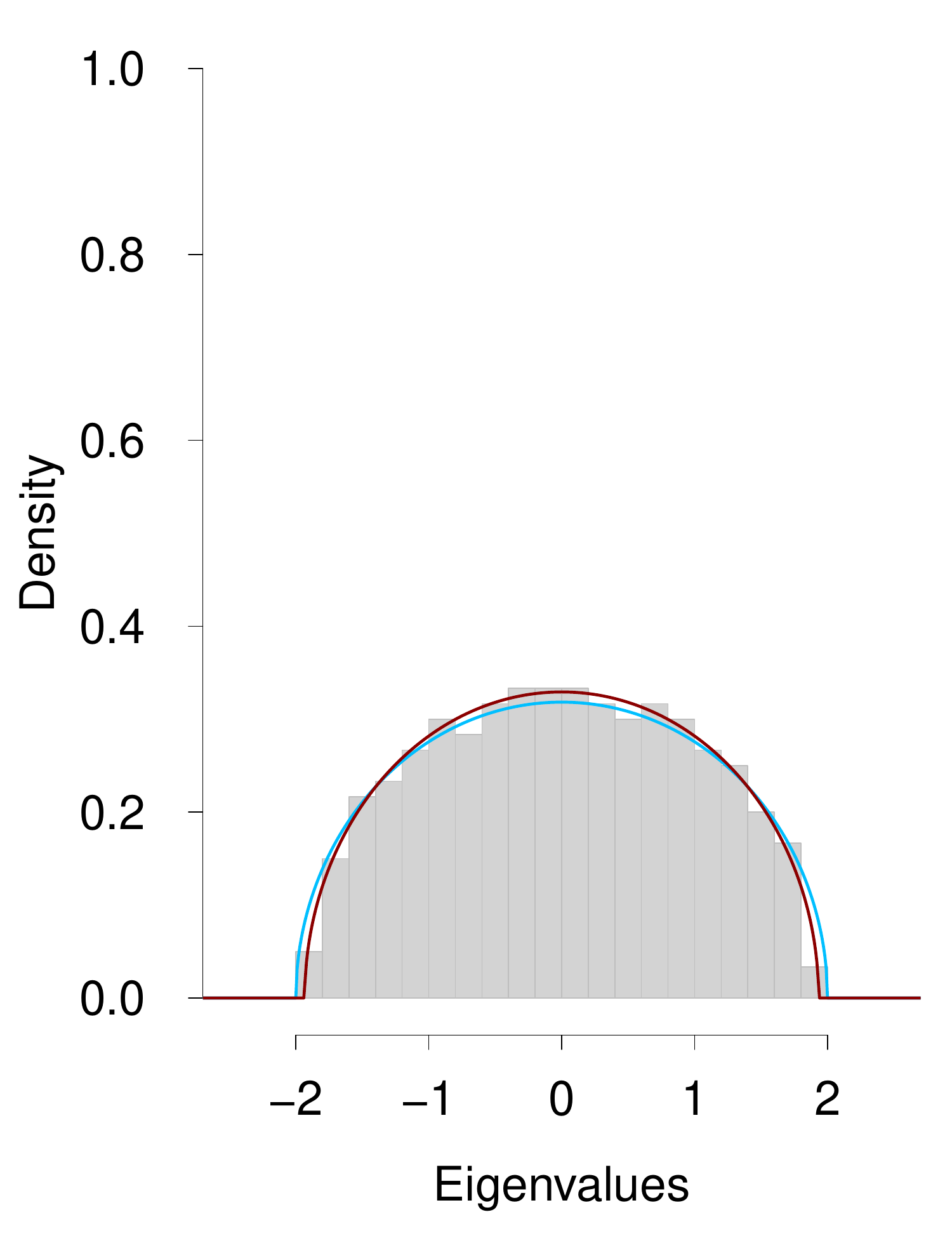} & \includegraphics[width = 0.3\textwidth]{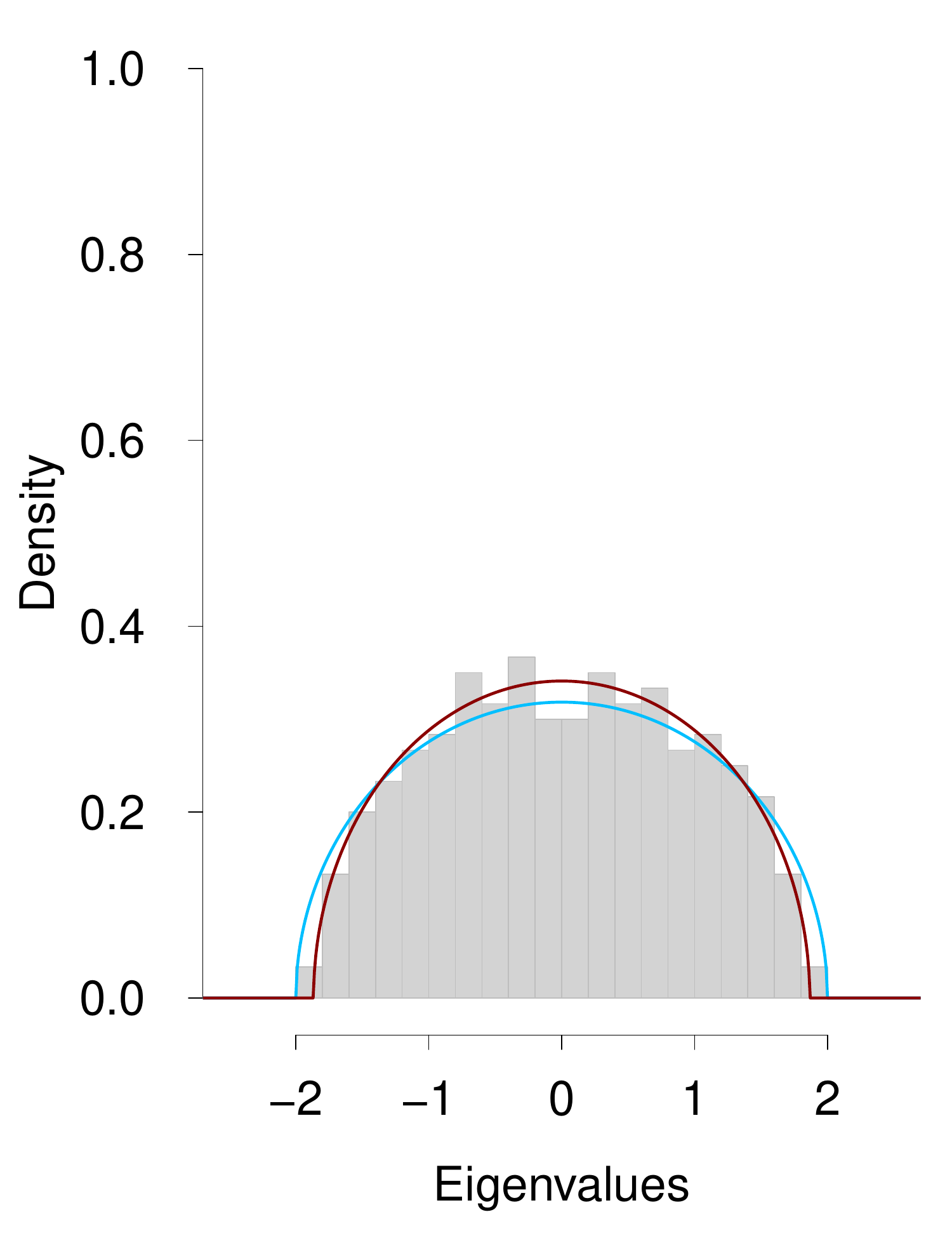} & \includegraphics[width = 0.3\textwidth]{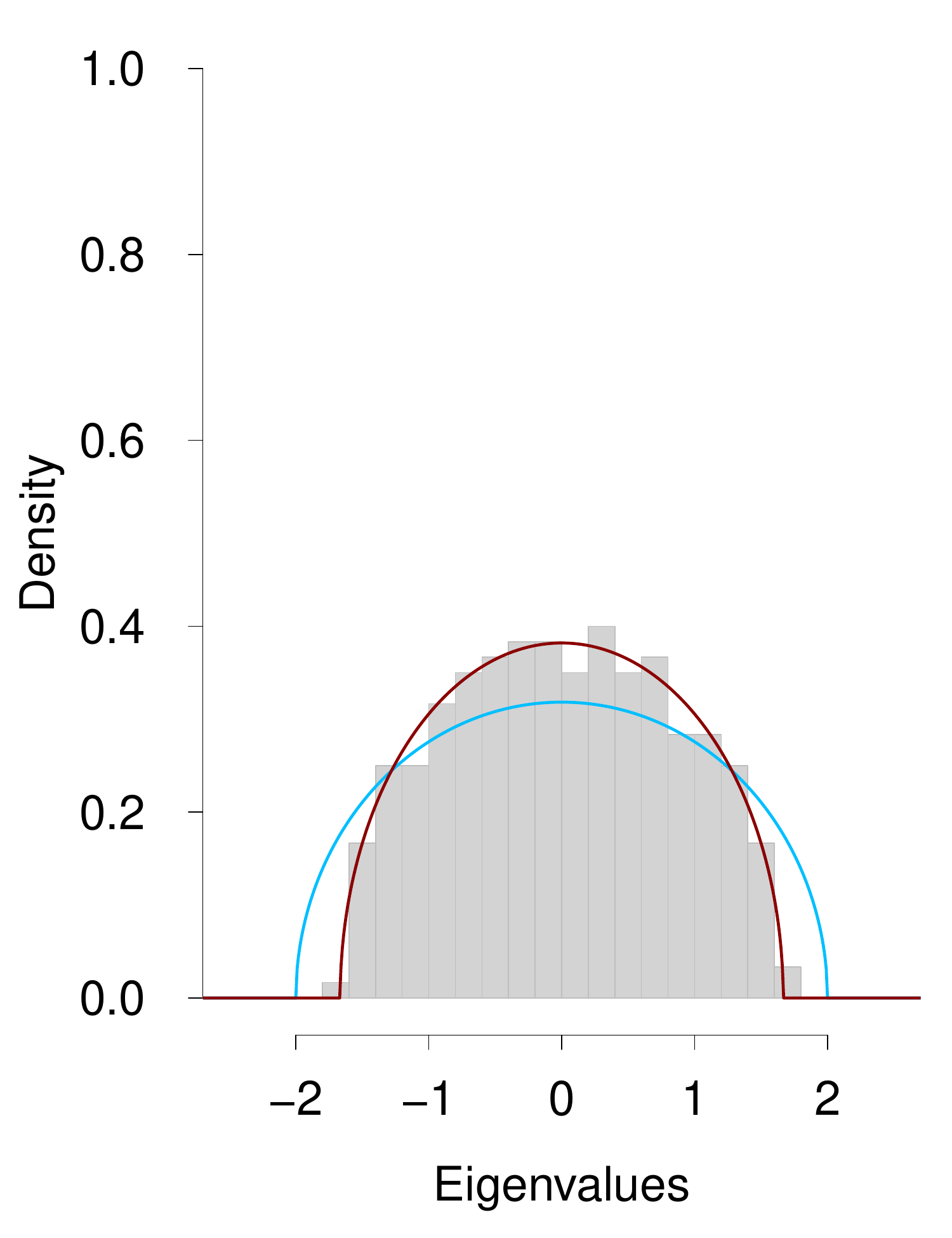} \\
         (a) $r = 10$ & (b) $r = 20$ & (c) $r = 50$ \\
        \includegraphics[width = 0.30\textwidth]{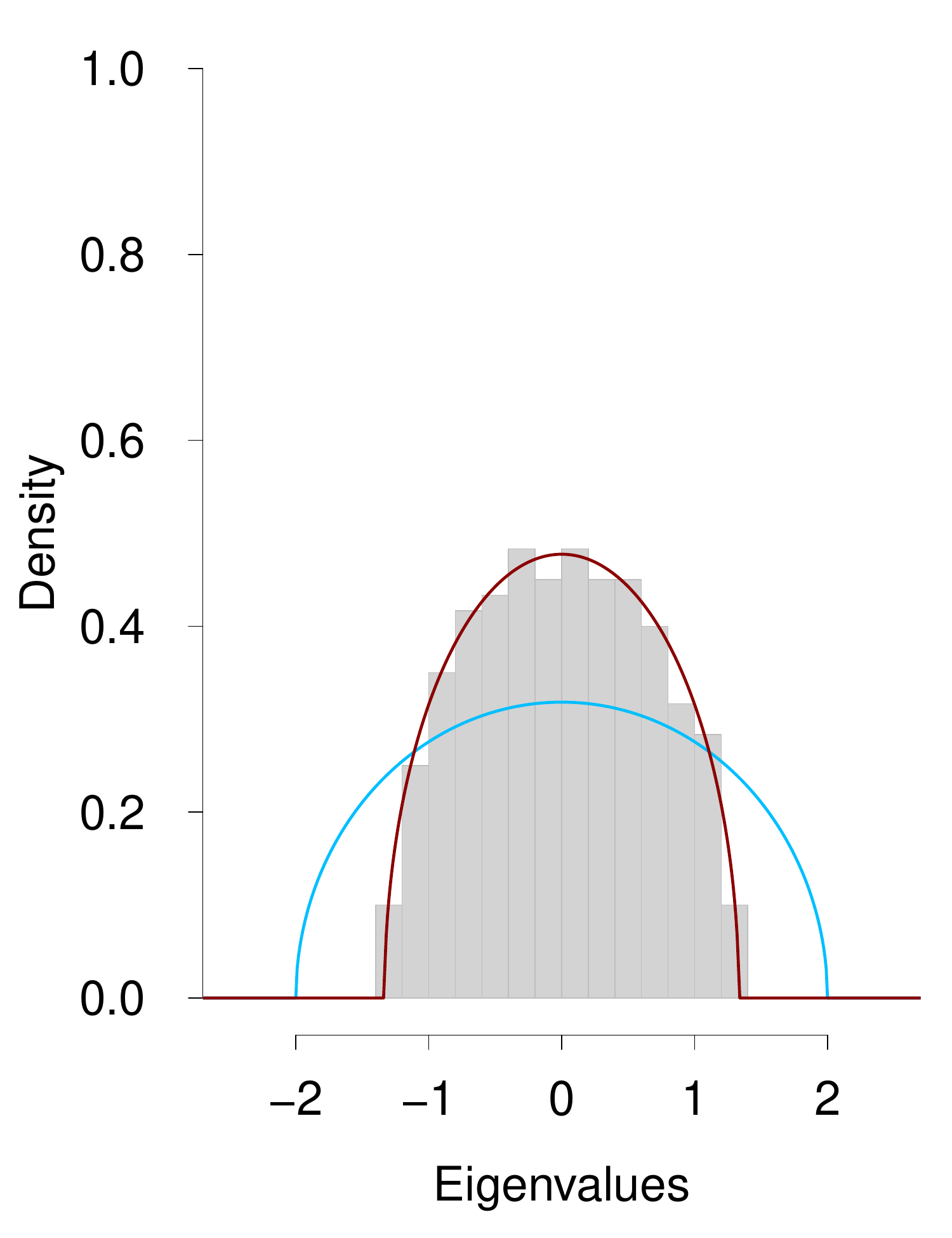} & \includegraphics[width = 0.30\textwidth]{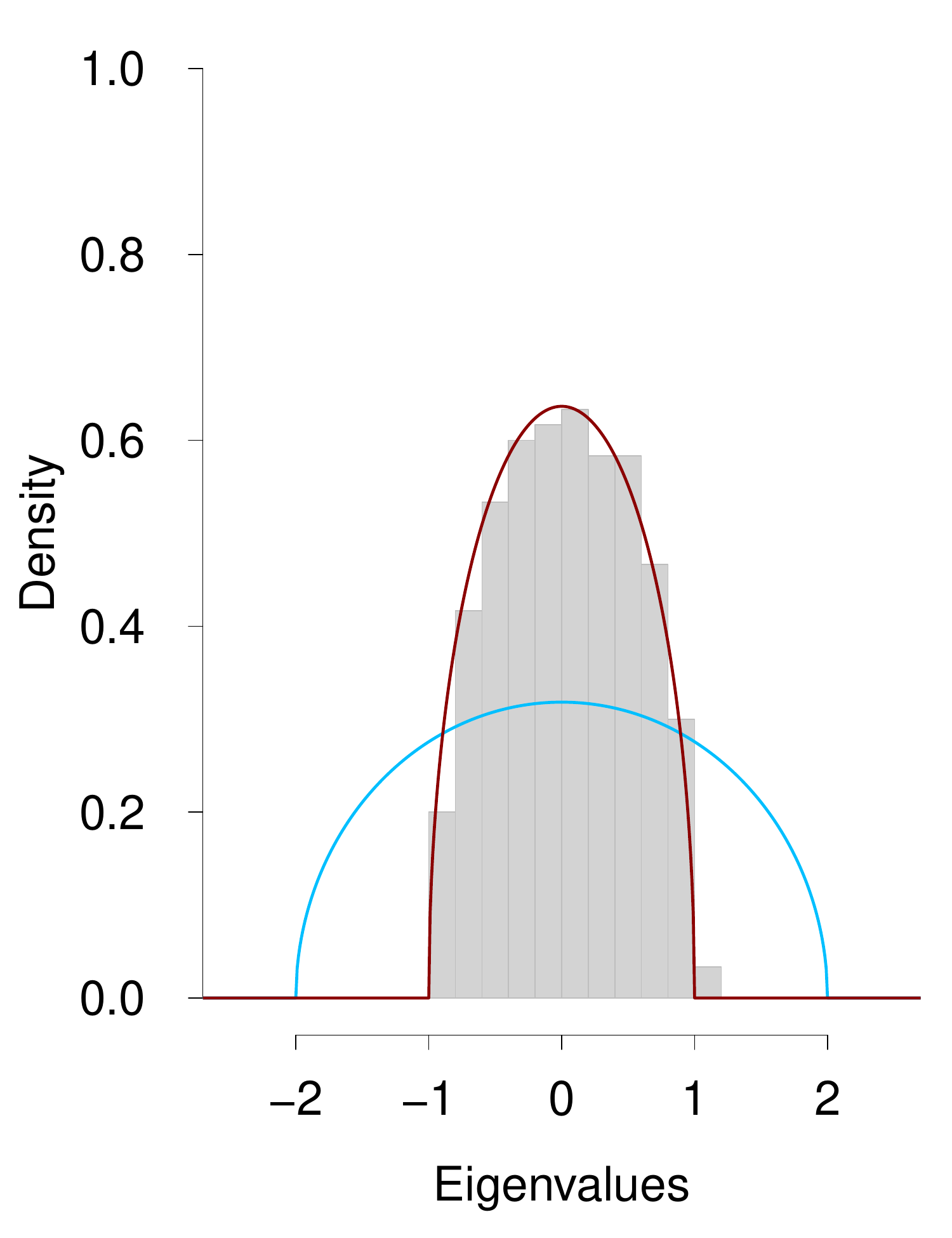} & \includegraphics[width = 0.30\textwidth]{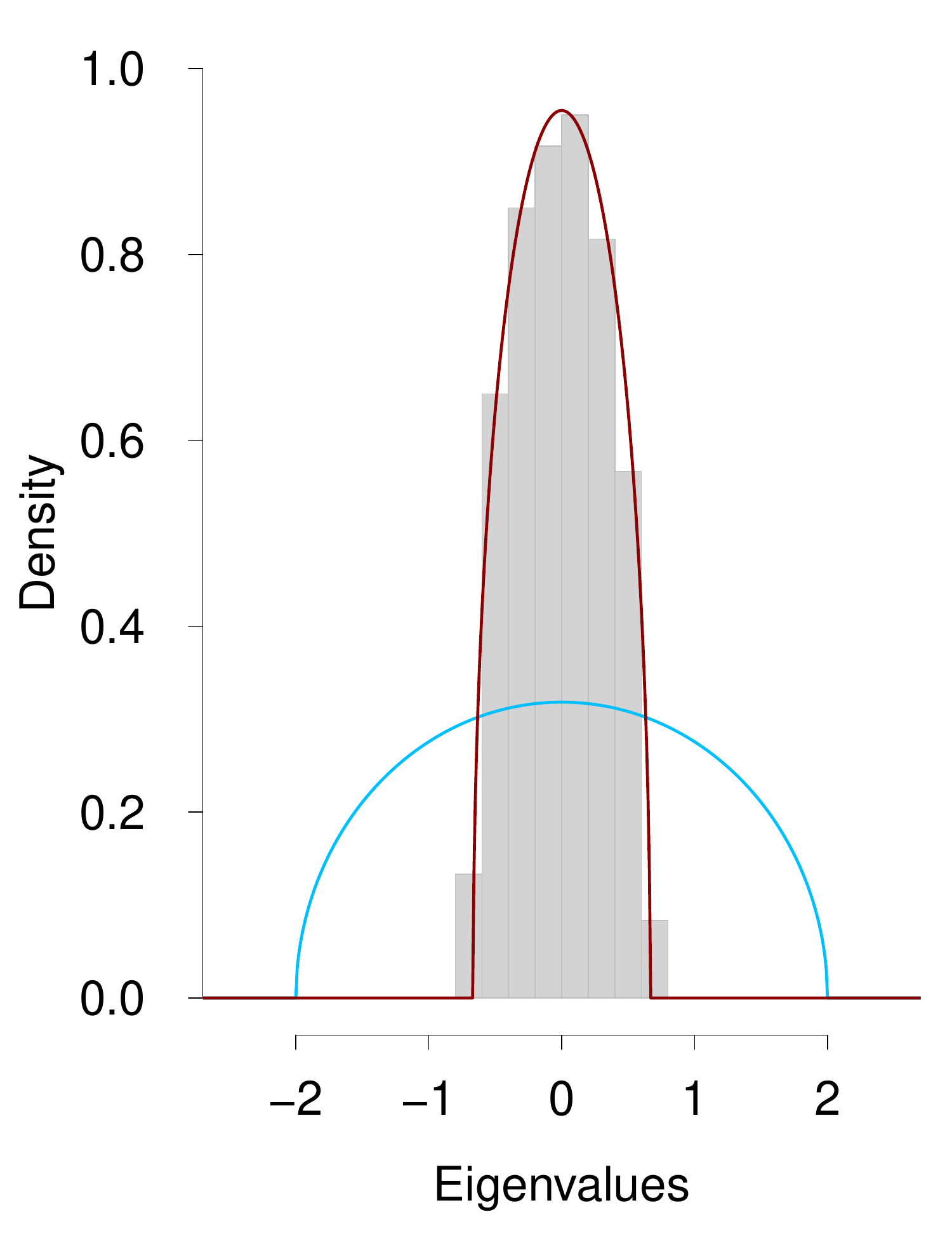} \\
         (d) $r = 100$ & (e) $r = 150$ & (f) $r = 200$
    \end{tabular}
    \caption{Histograms of the ESDs of $300 \times 300$ generalised hypergraph adjacency matrices for different values of $r$ (one realisation each). Each histogram is based on a single realisation. The blue curves depict the density of a standard semi-circle law. The red curves are the densities of semi-circle laws with variance $(1 - r/n)^2$.}
    \label{fig:esd_gaussian_model}
\end{figure}

Let us now briefly comment on the implications of the above result in the context of hypergraphs. Let the hyperedge inclusion probability $p$ be potentially dependent on both $n$ and $r$. Let us also assume for simplicity that $p$ is bounded away from $1$. In analogy with random graphs, the quantity $d_{\avg} = \binom{n - 1}{r - 1} p$ may be thought of as the ``average degree'' of a vertex. In fact, it is the expected number of hyperedges containing the said vertex. For $r = 2$, i.e. random graphs, it is well known that the limit of the EESD is the semi-circle law as long as $d_{\avg} = (n - 1) p \to \infty$, see \cite{ding2010spectral} for more details. We show in Example~\ref{exm:sparse_hypg} that our assumptions on the entries are satisfied as long as $d_{\avg} / r^7 \to \infty$ (in fact, for the Laplacian $L_{H_n}$ we only need that $d_{\avg}/r^4 \to \infty$). In the setting of fixed $r$, we thus have a semi-circle limit as long as $d_{\avg} \to \infty$. However the the factor $r^7$ is likely sub-optimal and an artefact of the Lindeberg exchange argument \cite{chatterjee2005simple} that we employ. Consider also the setting where $r / n \to c \in (0, 1)$. Let $I(x) = - x \log x - (1 - x) \log (1 - x)$ denote the entropy of a $\Ber(p)$ random variable. Then note that
\[
    \frac{d_{\avg}}{r^7} \sim \frac{\exp(I(c) n)}{n^{15/2} c^7 \sqrt{c(1 - c)}} p.
\]
Thus for obtaining the limit $\nu_{\sc, (1 - c)^2}$, we need
\[
    p \gg n^{15/2} \exp(- I(c) n).
\]
It would be interesting to understand what happens at complementary regimes of sparsity. We leave this for future work.

Unlike the random graph setting ($r = 2$), two eigenvalues stick out of the bulk of the spectrum. We can derive limits of appropriately scaled largest and smallest eigenvalues for the Gaussian GHAM. The following is an abridged version of Theorem~\ref{thm:extreme_eigen}.
\begin{theorem}
Let $H_n$ be a Gaussian GHAM and suppose that $\zeta$ is a standard Gaussian. In the regime $r/n \to c \in (0,1)$, we have 
  \begin{align*}
      \frac{\lambda_1(H_n)}{n} \xrightarrow{d} \frac{c}{2}\zeta + \sqrt{\frac{c^2}{4}\zeta^2 +c(1-c)} \quad \text{and} \quad \frac{\lambda_n(H_n)}{n} \xrightarrow{d} \frac{c}{2}\zeta - \sqrt{\frac{c^2}{4}\zeta^2 + c(1-c)}.
  \end{align*}
  When $r \to \infty$ and $r/n \to 0$, we need a rescaling:
 \begin{align*}
    \frac{\lambda_1(H_n)}{\sqrt{nr}} \xrightarrow{p} 1 \quad \text{and} \quad \frac{\lambda_n(H_n)}{\sqrt{nr}} \xrightarrow{p} -1.
 \end{align*}
 Finally, when $r$ is fixed, we have
     \begin{align*}
        \frac{\lambda_1(H_n)}{\sqrt{n}} &\convp
        \begin{cases}
            \sqrt{r-2} + \frac{1}{\sqrt {r-2}} & \text{ if } r \geq 4, \\
            2 &\text{ if } r \leq 3,
        \end{cases} \\
        \frac{\lambda_n(H_n)}{\sqrt{n}}&\convp
        \begin{cases}
            -\sqrt{r-2} - \frac{1}{\sqrt {r-2}} & \text{ if } r \geq 4, \\
            -2 & \text{ if } r \leq 3.
        \end{cases}
    \end{align*}
\end{theorem}
Noteworthy here is the phase transition at $r = 3$. The underlying reason for this is the so-called Baik-Ben Arous-P\'{e}ch\'{e} (BBP) transition for deformed Wigner matrices. This phase transition result remains valid for the original hypergraph adjacency matrix as well due to the strong universality results proved in \cite{brailovskaya2024universality}, provided that $p \gg \frac{\log^4 n}{n}$ which in terms of $d_{\avg}$ reads $d_{\avg} \gg n^{r - 2} \log^4 n$ (see the discussion in Section~\ref{sec:dep_wig}).

The behaviour of the extreme eigenvalues of the two Laplacian matrices is given in Theorem~\ref{thm:extreme_eigen_laplacian}.

\subsection{Related works on hypergraph and tensor spectra}
Spectra of hypergraphs had drawn the attention of researchers at least twenty years ago. Indeed, \cite{feng1996spectra}
studied the spectrum of $(d,r)$-regular hypergraphs and identified the LSD, which generalizes the work of \cite{mckay1981expected} for random regular graphs. \cite{lu2012loose} studied the spectra of the so-called \emph{loose Laplacians} of an Erd\H{o}s-R\'{e}nyi hypergraph. \cite{cooper2012spectra} studied the adjacency tensor of (deterministic) uniform hypergraphs. \cite{cooper2020adjacency} studied various aspects of the adjacency tensor for random $k$-uniform hypergraphs.

Friedman \cite{friedman1995second} was the first to define the ``second eigenvalue" of a multilinear form, analogous to the bilinear forms associated with the adjacency matrix of a graph, for both deterministic and random $d$-regular $r$-uniform hypergraphs, and provided explicit expressions for it. In a follow-up work, \cite{friedman1991spectra} introduced the notion of regular infinite hypertrees and established that the ``first eigenvalue'' of an infinite hypertree agrees with the ``second eigenvalue" of a random regular hypergrpah of the same degree within a logarithmic factor. On the other hand, \cite{bollalaplacian1993} attempted to define an analogue of the Laplacian for hypergraphs, but encountered various obstacles that made the generalisation difficult. It was Chung \cite{chung1993laplacian}, who paved the way by defining the Laplacian for hypergraphs in full generality using homological considerations. \cite{chang1999laplacian} further enriched the study by introducing an adjacency graph of hypergraphs, examining the spectra of the Laplacian for $d$-regular hypergraphs, and improving the lower bound for the ``spectral value'' of the Laplacian defined in \cite{chung1993laplacian}. \cite{hu2012algebraic} proposed another definition of a Laplacian tensor associated with even and uniform hypergraphs and analysed its connection with edge and vertex connectivity. \cite{cooper2012spectra} examined the eigenvalues of adjacency tensor also known as 'hypermatrix' of uniform hypergraphs and established several natural analogues of basic results in spectral graph theory. \cite{li2013z} introduced another definition for the Laplacian tensor of an even uniform hypergraph, proved a variational formula for its second smallest Z-eigenvalue, and used it to provide lower bounds for the bipartition width of the hypergraph. \cite{xie2013z} proposed a definition for the signless Laplacian tensor of an even uniform hypergraph, studied its largest and smallest H-eigenvalues and Z-eigenvalues, as well as its applications in the edge cut and the edge connectivity of the hypergraph. In \cite{xie2013zadj} they also investigated the largest and the smallest Z-eigenvalues of the adjacency tensor of a uniform hypergraph. Pearson and Zhang \cite{pearson2014spectral} studied the H-eigenvalues and the Z-eigenvalues of the adjacency tensor of a uniform hypergraph. \cite{hu2015laplacian} studied $H^{+}$-eigenvalues of normalized Laplacian tensor of uniform hypergraphs and showed that the second smallest $H^{+}$-eigenvalue is positive if and only if the hypergraph is connected. Until now, the literature has primarily focused on uniform hypergraphs. \cite{banerjee2017spectra} was the first to introduce adjacency hypermatrix, Laplacian hypermatrix and normalized Laplacian hypermatrix for general hypergraphs, and extensively studied various spectral properties of these tensors. In addition to adjacency and Laplacian tensors, researchers have also suggested different types of adjacency and Laplacian matrices associated with hypergraphs. 

In a series of papers \cite{rodri2002laplacian, rodriguez2003laplacian, rodriguez2009laplacian}, Rodr\'iguez introduced an $ n \times n$ Laplacian matrix similar to ours and studied how the spectrum of this Laplacian relates to various structural properties of hypergraphs such as the diameter, the mean distance, excess, etc., as well as its connection to the hypergraph partition problem.  \cite{banerjee2021spectrum} also studied adjacency and Laplacian matrices, which are the same as the matrices we consider up to minor scaling factors and showed how the various structural notions of connectivity, vertex chromatic number and diameter are related to the eigenvalues of these matrices. Building on the work of \cite{banerjee2021spectrum}, 
\cite{sahalapspectra2022} examined the spectral properties of the Laplacian matrix in greater detail. Specifically, they obtained bounds on the spectral radius of uniform hypergraphs in terms of some invariants of hypergraphs such as the maximum degree, the minimum degree, etc. A more general treatment is provided in \cite{Bangenoperators2023}, where the authors generalized the concepts of adjacency and Laplacian matrices for hypergraphs by introducing general adjacency operators and general Laplacian operators.

The study of Laplacian matrix for Wigner matrices was initiated in \cite{bryc2005spectral}. Later, \cite{sparselap2012} considered the Laplacian matrix for sparse Erd\H{o}s-R\'eyni random graphs. The normalized Laplacian in the non-sparse setting was studied in \cite{chimarkovmatrix2016}. For Wigner matrices with a variance profile, the spectrum  of the Laplacian matrix was analysed in \cite{LapgenerealWigner2022}.

More generally, in the random tensor front, \cite{gurau2020generalization} considered a tensor version of the Gaussian Orthogonal Ensemble (GOE). He showed that the expectation of an appropriate generalisation of the resolvent can be described by a generalised Wigner law, whose even moments are the Fuss-Catalan numbers. There has been a flurry of activity surrounding random tensors in the past few years \cite{goulart2022random, au2023spectral, bonnin2024universality,seddik2024random}. Among these, the most relevant to our setting are \cite{goulart2022random, au2023spectral, bonnin2024universality}. The papers \cite{goulart2022random, au2023spectral} considered contractions of random tensors to form random matrices. \cite{goulart2022random} considered the above-mentioned tensor version of the GOE and showed that for any contraction along a direction $\bv \in \bbS^{n - 1}$, one gets a semi-circle law as the LSD. This result also follows from the work of \cite{bonnin2024universality} who studied more general contractions. \cite{au2023spectral} obtained the same result for more general entries. They also established joint convergence of a family of contracted matrices in the sense of free probability.

The only difference between our adjacency tensor $\cA$ and a Wigner random tensor is that for us
\[
    \cA_{i_1, \ldots, i_r} = 0 \text{ if } |\{i_1, \ldots, i_r\}| < r.
\]
Modulo this difference, it is clear that one obtains our hypergraph adjacency matrix $A$ by contracting the adjacency tensor $\cA$ along the direction $\bone / \sqrt{n}$. Using the Hoffmann-Wielandt inequality it is not difficult to show that if $r$ is fixed, then zeroing out the elements $\cY_{i_1, \ldots, i_r}$ of a Wigner random tensor $\cY$ for any $(i_1, \ldots, i_r)$ such that $|\{i_1, \ldots, i_r\}| < r $ does not affect the bulk spectra of its contractions. Therefore for fixed $r$, our Theorem~\ref{thm:main} follows from the Theorem~1.5 of \cite{au2023spectral}
 (or Theorem~2 of \cite{goulart2022random} or Theorem~4 of \cite{bonnin2024universality} in the tensor GOE case).

We emphasize that unlike our setting where the order $r$ of the tensor grows with the dimension $n$, the above-mentioned existing works on random tensors or contractions thereof consider the setting of fixed $r$. To the best of our knowledge only two lines of work consider settings where $r$, the common size of the hyperedges, grow with $n$, namely in quantum spin glasses \cite{erdHos2014phase} and in the Sachdev-Ye-Kitaev (SYK) model for black holes \cite{feng2019spectrum}. Both of these models consider matrices representable as random linear combinations of deterministic matrices, with the i.i.d. random weights being indexed by the hyperedges of a complete $r$-uniform hypergraph (exactly like the representation \eqref{eq:gham_defn} of a GHAM). In both of these works, a phase transition phenomenon is observed at the phase boundary $r \asymp \sqrt{n}$: When $r \ll \sqrt{n}$, the LSD is Gaussian; when $r \gg \sqrt{n}$, the LSD is the standard semi-circle law; and when $r/\sqrt{n} \to c > 0$, a different LSD emerges.

\subsection{Related works on Hermitian matrices with dependent entries}\label{sec:dep_wig}

We will refer to Hermitian random matrices with dependent entries as \emph{dependent/correlated Wigner matrices}. Various different types of correlated Wigner matrices has been studied in the literature in great detail. Here we will recall some of these existing works and compare them to our results.

Chapter 17 of \cite{pastur2011eigenvalue} considered a class of random $n \times n$ real symmetric matrices of the form
\[
    H = H^{(0)} + n^{-\frac{1}{2}} W,
\]
where $H^{(0)}$ is a given real symmetric matrix and $W = \{W_{jk}\}_{j,k = -m}^m$, $n = 2m + 1$, is a real symmetric random matrix, the $n \times n$ central block of the double-infinite matrix $W_{\infty} = \{W_{jk}\}_{j,k = -\infty}^{\infty}$ whose entries are Gaussian random variables with
\[
    \bbE{W_{jk}} = 0, \ \ \bbE{W_{j_1 k_1} W_{j_2 k_2}} = C_{j_1k_1 ; j_2 k_2}.
\]
The authors assume that
\[
    C := \sup_{j_1,k_1 \in \bbZ} \sum_{j_2, k_2 \in \bbZ} \big|C_{j_1k_1 ; j_2 k_2}\big| < \infty.
\]
Approximate versions of conditions like these are clearly not satisfied by our model where  the sum of the pairwise absolute correlations between a given entry and all the other entries grows like $r^2$.

\cite{chakrabarty2013limiting} and \cite{chakrabarty2016random} considered correlated Gaussian Wigner models where the entries come from stationary Gaussian fields. However, this specific structural assumption makes their models incompatible with ours. (Further, they also require summability of the pairwise correlations.)

Another dependent model was considered by \cite{gotze2015limit} where the entries $(X_{ij})_{1 \le i \le j \le n}$ form a dependent random field. One of their main assumptions however is that $\bbE[X_{ij} |\cF_{ij} ] = 0$, where $\cF_{ij}$ is the $\sigma$-algebra generated by the entries other than $X_{ij}$. This implies that
\[
    \bbE [X_{ij} X_{i'j'}] = \bbE [\bbE [X_{ij} X_{i'j'} | \cF_{ij}] ] = \bbE [X_{i'j'} \bbE [X_{ij} \mid \cF_{ij}]] = 0,
\]
i.e. the matrix entries are uncorrelared. This rules out a direct application of their method to our set-up. Notably, their proof technique for showing universality, which uses an interpolation between their random matrix model and a Wigner matrix, can be modified to work in a situation where $\bbE[X_{ij} \mid \cF_{ij}] \ne 0$ (we note here that the same result can also be proved by the Lindeberg exchange argument of \cite{chatterjee2006generalization}). However, the resulting universality result turns out to be too weak to be applicable to our setting (in our notation, it requires $r$ to go to $0$).

To the best of our knowledge, bulk universality results for the most general correlated model is obtained in \cite{erdHos2019random} with appropriate decay conditions on the joint cumulants of the entries. For a Gaussian model, they require the operator norm of the covariance matrix of the entries to be $O(N^{\epsilon})$ for any $\epsilon > 0$. In our setting the operator norm is of order $r^2$ and hence their results do not apply to the setting of polynomially growing $r$.

It is also worth mentioning the remarkably general work of
\cite{brailovskaya2024universality}, who consider matrices of the form $X = Z_0 + \sum_{i= 1}^n Z_i$, where $Z_0$ is a deterministic matrix and $Z_i$'s are $d\times d$ independent self-adjoint random matrices. Under appropriate conditions on these matrices the authors of \cite{brailovskaya2024universality} prove strong universality results. Owing to the representation \eqref{eq:GHAM_repr_lin_comb}, our adjacency matrix clearly falls under this rather general model. Unfortunately, the universality results proved in \cite{brailovskaya2024universality} do not work beyond the regime $r \ll n^{1/4}$. This is still powerful enough to give us universality of the edge eigenvalues for fixed $r$ and hence of the BBP-type phase transition at $r = 3$ described earlier. Below we quickly work out this universality result, assuming that the $Y_{\ell}$'s are uniformly bounded.

Let $d_H(A, B)$ denote the Hausdorff distance between two subsets $A, B \subset \bbR$. For a symmetric matrix $A$, let $\spec(A)$ denote its spectrum. Theorem~2.6 of \cite{brailovskaya2024universality} shows that if the matrices $Z_i$ are uniformly bounded, then
\[
    \bbP(d_H(\spec(X), \spec(G)) > C \varepsilon(t)) \le d e^{-t},
\]
where $G$ is a Gaussian random matrix with the same expectation and covariance structure as $X$ and 
\[
    \varepsilon(t) = \sigma_*(X) t^{1/2} + R(X)^{1/3} \sigma(X)^{2/3} t^{2/3} + R(X) t,
\]
with
\begin{align*}
    v(X) &:= \|\Cov(X)\|_{\op}^{1/2}, \\
    \sigma(X) &:= \|\bbE[(X - \bbE X)^2]\|_{\op}^{1/2}, \\
    \sigma_*(X) &= \sup_{\|v\| = \|w\| = 1} \bbE[|\langle v, (X - \bbE X) w\rangle|]^{1/2}, \\
    R(X) &:= \bigg\| \max_{1 \le i \le n} \|Z_i\|_{\op}\bigg\|_{\infty}.
\end{align*}
Let us now calculate these parameters for $X = n^{-1/2} H_n$. First note that 
\begin{align*}
    v(X) &= \|\Cov(X)\|_{\op}^{1/2} = O([\text{maximum row sum of } \Cov(X)]^{1/2}) \\
    &= \frac{1}{\sqrt{n}} \cdot O\bigg(\bigg[\Theta(n) \cdot \frac{r}{n} + \Theta(n^2) \cdot \frac{r^2}{n^2}\bigg]^{1/2}\bigg) \\
    &= O\bigg(\frac{r}{\sqrt{n}}\bigg).
\end{align*}
Since $\bbE (H_n^2)_{ij} = \sum_{k \ne i, j} \bbE H_{n, ik} H_{n, kj} = (r - 2)$ and $\bbE (H_n^2)_{ii} = \sum_{k \ne i} \bbE H_{n, ik}^2 = (n - 1)$, we have
\[
    \bbE H_n^2 = (n - 1) I + (r - 2) (J - I) = (n - r + 1) I + (r - 2) J,
\]
where $J = \bone \bone^\top$. It follows that
\[
    \sigma(X):= \|\bbE[(X - \bbE X)^2]\|^{1/2} = \frac{1}{\sqrt{n}} \cdot \|\bbE H_n^2 \|_{\op}^{1/2} = \frac{\sqrt{n + r - 1 + n (r - 2)}}{\sqrt{n}} = \Theta(\sqrt{r}).
\]
We also have
\[
    \sigma_*(X) \le v(X) = O\bigg(\frac{r}{\sqrt{n}}\bigg).
\]
Finally, assuming that the $Y_\ell$'s are uniformly bounded, by say $K$, we have
\[
    R(X) := \frac{1}{\sqrt{n}} \cdot \bigg\| \max_{1 \le \ell \le M} \|Y_{\ell} Q_{\ell}\|_{\op}\bigg\|_{\infty} \le \frac{K r}{\sqrt{n}}.
\]
Thus
\[
    \varepsilon(t) = O\bigg( \frac{r t^{1/2}}{\sqrt{n}} + \bigg(\frac{Kr}{\sqrt{n}}\bigg)^{1/3} \cdot r^{1/3} t^{2/3} + \frac{K r t}{\sqrt{n}}\bigg).
\]
Choosing $t = 2 \log n$, yields that with probability at least $1 - \frac{1}{n}$, we have
\[
    d_H(\spec(n^{-1/2}H_n(\bY)), \spec(n^{-1/2} H_n(\bZ)) = O\bigg( \frac{r \sqrt{\log n}}{\sqrt{n}} + K^{1/3} \cdot \bigg(\frac{r^2 \log^2 n}{\sqrt{n}}\bigg)^{1/3} + \frac{K r \log n}{\sqrt{n}}\bigg).
\]
It is evident that the above upper bound is small if $r \ll \frac{n^{1/4}}{\sqrt{K} \log n}$, assuming that $K$ is constant, which is, for instance, the case for the hypergraph, where one can take $K = \max\big\{\sqrt{\frac{p}{1 - p}}, \sqrt{\frac{1 - p}{p}}\big\}$. Assuming that $p < 1/2$, the above condition for universality simplifies to $r \ll \frac{(np)^{1/4}}{\log n}$.

\subsection{Organisation of the paper.}
The rest of the paper is organised as follows: In Section~\ref{sec:main_results}, we formally describe the main results of this paper. Section~\ref{sec:proofs} contains the proofs of these results. Finally, in Appendix~\ref{sec:aux}, we collect some useful results from matrix analysis and concentration of measure which are used throughout the paper, and in Appendix~\ref{sec:more_proofs} we provide proofs of some technical results.

\section{Main results}\label{sec:main_results}
In the first three subsections, we study the bulk of the spectrum. In the last subsection, we study the edge eigenvalues.
\subsection{Replacing general entries with Gaussians}
We first employ a Lindeberg swapping argument \cite{chatterjee2005simple} to replace the independent zero mean unit variance random variables $Y_\ell$ in the definition of $H_n$ by i.i.d. standard Gaussians. Recall that $M = \binom{n}{r}$ and $N = \binom{n - 2}{r - 2}$. We will need a Pastur-type condition on the entries $(Y_{\ell})_{1\le \ell \le M}$.

\begin{assumption}[A Pastur-type condition]
\label{ass:tail_of_entries}
Suppose $(Y_{\ell})_{1 \le \ell \le M}$ are independent zero mean unit variance random variables satisfying the following condition: For every $\varepsilon > 0$,
\[
    \frac{r^4}{n^2 N}  \sum_{\ell = 1}^M \bbE[Y^2_{\ell} \ind(|Y_{\ell}| > \varepsilon K_n)] \to 0 \ \ \text{as} \ \  n \to \infty,
\]
where $K_n = \frac{\sqrt{n N}}{r^4}$.
\end{assumption}
For dealing with the Laplacian $L_{H_n}$ we need a weaker assumption.
\begin{assumption}[A (weaker) Pastur-type condition]
\label{ass:tail_of_entries_variation}
Suppose $(Y_{\ell})_{1 \le \ell \le M}$ are independent zero mean unit variance random variables satisfying the following condition: For every $\varepsilon > 0$,
\[
    \frac{r^{3/2}}{n^2 N}  \sum_{\ell = 1}^M \bbE[Y^2_{\ell} \ind(|Y_{\ell}| > \varepsilon K_n')] \to 0 \ \ \text{as} \ \  n \to \infty,
\]
where $K_n' = \frac{\sqrt{n N}}{r^{5/2}}$.
\end{assumption}
\begin{remark}
    For $r$ fixed, Assumptions ~\ref{ass:tail_of_entries} and \ref{ass:tail_of_entries_variation} are clearly equivalent. For $r = 2$, they become Pastur's condition \cite{pastur1972}. In general, Assumption~\ref{ass:tail_of_entries} is stronger than Assumption~\ref{ass:tail_of_entries_variation}. Indeed, $K_n' > K_n$ and hence
\begin{align*}
    \frac{r^{3/2}}{n^2 N} \sum_{\ell = 1}^M \bbE[|Y_{\ell}|^2 \ind(|Y_{\ell}| > \varepsilon K_n')] &\le \frac{r^{3/2}}{n^2 N} \sum_{\ell = 1}^M \bbE[|Y_{\ell}|^2 \ind(|Y_{\ell}| > \varepsilon K_n)] \\
    &\le \frac{1}{r^{5/2}} \frac{r^4}{n^2 N} \sum_{\ell = 1}^M \bbE[|Y_{\ell}|^2 \ind(|Y_{\ell}| > \varepsilon K_n)].
\end{align*}
    So if $(Y_{\ell})_{1 \le \ell \le M}$ satisfy Assumption~\ref{ass:tail_of_entries}, then they also trivially satisfy \eqref{ass:tail_of_entries_variation}.
\end{remark}
\begin{remark}
    For i.i.d. random variables $(Y_{\ell})_{1 \le \ell \le M}$ with mean $0$ and variance $1$, Assumption~\ref{ass:tail_of_entries} becomes the following tail condition:
    \begin{equation}\label{eq:tail_cond_iid}
        \bbE[Y_1^2 \ind(|Y_1| > \varepsilon K_n)] = o(r^{-2})
    \end{equation}
    for any $\varepsilon > 0$.
\end{remark}

Our first result shows the universality of the EESD in the limit, assuming that the entries satisfy Assumption~\ref{ass:tail_of_entries}, i.e. the limiting EESD, if it exists, is the same as that of a GHAM with Gaussian entries. The proof of this result is via a comparison of \emph{Stieltjes transforms}. Recall that the Stieltjes transform $S_{\mu}$ of a probability measure $\mu$ on $\bbR$ is an analytic function defined for $z \in \mathbb{C}_+ := \{z \in \bbC : \Im(z) > 0\}$ as follows:
\[
    S_\mu(z) := \int\frac{d\mu(x)}{x-z}.
\]
Suppose $\{\mu_n\}_{n \ge 1}, \mu$ are probability measures on $\mathbb{R}$ with Stieltjes transforms $\{S_{\mu_n}\}_{n \ge 1}$ and $S_\mu$, respectively. It is well known than $S_{\mu_n}\to S_\mu$ pointwise on $\bbC_+$ if and only if $\mu_n \to \mu$ weakly (see, e.g., \cite{anderson2010introduction}). With this result in mind, the universality of the limiting EESD is given by our first theorem.

\begin{theorem}[Universality of the limiting EESD]\label{thm:univ}
Suppose $r/n \to c \in [0, 1)$. Suppose $\bY = (Y_\ell)_{1 \le \ell \le M}$ satisfies Assumption~\ref{ass:tail_of_entries}. Let $\bZ = (Z_{\ell})_{1 \le \ell \le M}$ be a vector of i.i.d. standard Gaussians. Then for any $z \in \bbC_+$,
    \[
        \lim_{n \to 0} \big| S_{\mubar_{B_n(\bY)}}(z) - S_{\mubar_{B_n(\bZ)}}(z)\big| = 0
    \]
    where $B_n(\bx)$ is any one of the matrices $n^{-1/2} H_n(\bx)$, $(nr)^{-1/2}L_{H_n(\bx)}$ or $n^{-1/2}\tilde{L}_{H_n(\bx)}$. In fact, for the result with $B_n(\bx) = (nr)^{-1/2} L_{H_n(\bx)}$, we only need Assumption~\ref{ass:tail_of_entries_variation} on the entries of $\bY$.
\end{theorem}

We now look at some situations where Assumption~\ref{ass:tail_of_entries} holds.
\begin{example}
    Clearly, the tail condition in \eqref{eq:tail_cond_iid} (and hence Assumption~\ref{ass:tail_of_entries}) is satisfied by i.i.d. bounded random variables $(Y_{\ell})_{1 \le \ell \le M}$ with zero mean and unit variance. This already covers the case of dense hypergraphs, where the hyperedge inclusion probability $p$ is bounded away from zero. 
\end{example}

\begin{example}
    Suppose that $(Y_{\ell})_{1 \le \ell \le M}$ are i.i.d. zero mean unit variance random variables with $\bbE |Y_1|^{2 + \delta} \le C$ for some $\delta, C > 0$. Then Assumption~\ref{ass:tail_of_entries} is satisfied for any sequence $r_n$ satisfying $2 \le r_n < n - \lceil 3 + \frac{2}{\delta} \rceil$. As a consequence, i.i.d. standard Gaussians satisfy Assumption~\ref{ass:tail_of_entries} for $2 \le r_n \le n - 4$.

    Indeed, by H\"{o}lder's inequality, we have for conjugate exponents $p, q \ge 1$ satisfying $1/p + 1/q = 1$,
    \begin{align*}
        \bbE[Y_1^2 \ind(|Y_1| > K_n)] &\le (\bbE |Y_1|^{2p})^{1/p} \,\,\, \bbP(|Y_1| > K_n)^{1/q} \\
        &\le (\bbE |Y_1|^{2p})^{1/p} \,\,\, \bbE(|Y_1|^{2 + \delta})^{1/q} \,\,\, K_n^{-(2 + \delta)/q} \quad \text{(by Markov's inequality)}.
    \end{align*}
    For $2p = 2 + \delta$ (which gives $q = (2 + \delta) / \delta$), this gives
    \[
        \bbE[Y_1^2 \ind(|Y_1| > \varepsilon K_n)] \le \frac{\bbE |Y_1|^{2 + \delta}}{\varepsilon^\delta K_n^\delta} = \frac{\bbE |Y_1|^{2 + \delta}}{\varepsilon^\delta (nN)^{\delta/2} r^{-4\delta}}.
    \]
    The right hand side is $o(r^{-2})$ if
    \[
        \frac{r^{2 + 4 \delta}}{(nN)^{\delta}} = o(1).
    \]
    Now since
    \[
        \frac{r^{2 + 4 \delta}}{(nN)^{\delta}} = \bigg(\frac{r}{n}\bigg)^{\delta} \,\,\, \frac{r^{2 + 3 \delta}}{N^{\delta}} \le \frac{r^{2 + 3 \delta}}{N^{\delta}},
    \]
    we need
    \begin{equation}\label{eq:bin_coeff_bound}
        \binom{n}{r} \gg r^{3 + \frac{2}{\delta}}.
    \end{equation}
    We claim that this is always true. Let $a = \lceil 3 + 2 /\delta \rceil+ 1$. Then for $a \le r \le \frac{n}{2}$,
    \[
        \binom{n}{r} \ge \binom{n}{a} \ge \frac{n^{a}}{a^a} \ge C_\delta n^{\lceil 3 + 2 / \delta \rceil + 1} \gg n^{3 + 2 / \delta} \ge r^{3 + 2/\delta}.
    \]
    On the other hand, for $2 \le r < a$, \eqref{eq:bin_coeff_bound} is trivially true. 

    For $n/2 < r < n - \lceil 3 + \frac{2}{\delta} \rceil$, by symmetry of the binomial coefficients,
    \[
        \binom{n}{r} = \binom{n}{n - r} \ge \binom{n}{a} \gg n^{3 + 2/\delta},
    \]  
    as before.
\end{example}

\begin{example}
    Suppose that $(Y_{\ell})_{1 \le \ell \le M}$ are i.i.d. zero mean unit variance random variables and consider the regime where $r \sim c n$, where $c \in (0, 1)$, we have $N = \binom{n - 2}{r - 2} \sim \frac{c^2 \exp(I(c) n)}{\sqrt{c(1 - c) n}}$, where $I(x) = -x \log x - (1 - x) \log (1 - x)$ is the entropy of a $\Ber(x)$ random variable. In this regime, $K_n = \Theta(e^{C n})$ for some constant $C > 0$ (depending on $c$) and the tail condition \eqref{eq:tail_cond_iid} becomes rather mild:
    \[
        \bbE[Y_1^2 \ind(|Y_1| > \varepsilon e^{C n}) ] = o(n^{-2}) \quad \text{for any } \varepsilon > 0.
    \]
    For instance, a random variable whose density decays as fast as $\frac{1}{|x|^3 (\log |x|)^{3 + \delta}}$ as $|x| \to \infty$, for some $\delta > 0$ will satisfy this condition. Such random variables need not possess a $(2 + \eta)$-th moment for any $\eta > 0$.
\end{example}

\begin{example}[Dependence on sparsity] \label{exm:sparse_hypg}
    Consider the setting of hypergraphs, where $(Y_{\ell})_{1 \le \ell \le M}$ are i.i.d. with $Y_1 = \frac{B - p}{\sqrt{p(1 - p)}}$, where $B \sim \Ber(p)$. If $p$ is not bounded away from $0$ and $1$, then $|Y_1|$ is not uniformly bounded anymore. Instead,
    \[
        |Y_1|\le \max\bigg\{\sqrt{\frac{p}{1-p}},\sqrt{\frac{1-p}{p}}\bigg\} 
        \le \frac{1}{\sqrt{\min\{p, 1-p\}}}.
    \]
    The left hand side of \eqref{eq:tail_cond_iid} vanishes for every $\varepsilon > 0$ as long as
    \begin{equation}\label{eq:sparse_hyp_cond}
        \frac{nN}{r^8}{\min\{p,1-p\}} \to \infty.
    \end{equation}
    Assume now that $p$ is bounded away from $1$. For $d_{\avg} = \binom{n - 1}{r - 1} p$, the average degree of a vertex, we note that
    \[
        d_{\avg} = \binom{n - 1}{r - 1} p \asymp \frac{n}{r} N p \asymp r^7 \cdot \frac{nN}{r^8} \min\{p, 1 - p\}.
    \]
   The condition \eqref{eq:sparse_hyp_cond} thus becomes equivalent to the condition $d_{\avg} / r^7 \to \infty$. For comparison, in the random graph case (i.e. $r = 2$), it is well known that the semi-circle law emerges as the LSD of the adjacency matrix as long as the average degree $d_{\avg} = (n - 1) p \to \infty$. Thus in the setting of fixed $r$, we do get a semi-circle limit as long as $d_{\avg} \to \infty$.

    We note here that Assumption~\ref{ass:tail_of_entries_variation} is implied by the condition $\frac{d_{\avg}}{r^4} \to \infty$. Thus for the Laplacian matrix $(nr)^{-1/2} L_{H_n}$, we only need this weaker condition on $d_{\avg}$.
\end{example}

\subsection{LSD of the Gaussian GHAM}\label{sec:lsd_gaussian_GHAM}
In view of of the universality result of Theorem~\ref{thm:univ}, it sufficient to consider the Gaussian GHAM. Let $(\cM, d)$ be a metric space. For a real-valued function $f$ on $\cM$, define its Lipschitz seminorm by $\|f\|_{\Lip} := \sup_{x \neq y} \frac{|f(x) - f(y)|}{d(x,y)}$. A function $f$ is called $l$-Lipschitz if $\|f\|_{\Lip} \le l$. Define the class of Bounded Lipschitz functions as 
\[
    \cF_{\BL} := \{ f \in \bbR^\cM : \|f\|_{\Lip} + \|f\|_{\infty} \le 1 \}.
\]
Then the bounded Lipschitz metric on the set $\cP(\cM)$ of probability measures on $\cM$ is defined as follows:
\[
    d_{\BL}(\mu, \nu) := \sup_{f \in \cF_{\BL}} \bigg\{\bigg|\int f \,d\mu - \int f \,d\nu\bigg|\bigg\}.
\]
It is well known that $d_{\BL}$ metrises weak convergence of probability measures on $\cP(\cM
)$ (see, e.g., \cite[Chap.~11]{dudley2018real}). Below we have $\cM = \bbR$.

\begin{theorem}[Limit of the EESD in the Gaussian case]\label{thm:eesd_gaussian}
    Let $\bZ = (Z_{\ell})_{1\le \ell \le M}$ be a vector of i.i.d. Gaussians. Suppose $r/n \to c \in [0, 1)$. Then the EESD of $n^{-1/2} H_n$ converges weakly to $\nu_{\sc, (1 - c)^2}$. In fact, one has
    \[
        d_{\BL}(\mubar_{n^{-1/2} H_n(\bZ)}, \nu_{\sc, (1 - c)^2}) = O\bigg(\max\bigg\{\bigg|\frac{r}{n} - c\bigg|, \frac{1}{\sqrt{n}}\bigg\}\bigg).
    \]
\end{theorem}
Theorem~\ref{thm:eesd_gaussian} is proved by representing the Gaussian GHAM via an ANOVA-type decomposition as a low rank perturbation of a scaled Gaussian Wigner matrix minus its diagonal. 

We will now consider the Laplacian matrix $\tilde{L}_{H_n}$. To describe the limit we need the concept of free additive convolution of measures. Let $\mu_1$ and $\mu_2$ be two probability measures on $\bbR$ with Stieltjes transforms $S_{\mu_1}(z)$ and $S_{\mu_2}(z)$ respectively. Then there exists a probability measure $\mu$ whose Stieltjes transform is the unique solution of the system
\begin{align*}
    f(z) &= S_{\mu_1}\bigg( z - \frac{\Delta_2(z)}{f(z)}\bigg), \\
    f(z) &= S_{\mu_2}\bigg( z - \frac{\Delta_1(z)}{f(z)}\bigg), \\
    f(z) &= \frac{1 - \Delta_1(z) - \Delta_2(z)}{-z},
\end{align*}
where the functions $\Delta_i(z), i = 1, 2$, are analytic for $\Im(z) \neq 0$ and satisfy
\[
    \Delta_{i}(z) \to 0 \text{ as } \Im(z) \to \infty.
\]
The measure $\mu$ is denoted by $\mu_1 \boxplus \mu_2$ and  is called the \emph{free additive convolution of $\mu_1$ and $\mu_2$.} For further details, we refer the reader to \cite{biane1997free}, \cite{pastur2000law}. 

\begin{theorem}[Limit of the EESD for Laplacian]\label{thm:eesd_gaussian_laplacian}
Let $H_n$ be a Gaussian GHAM. 
   \begin{enumerate}
       \item [(i)] Suppose that $r$ is fixed. Then EESD of $n^{-1/2}L_{H_n}$ converges weakly to $\nu_{\Gauss, r - 1} \boxplus \nu_{\sc, 1}$ and  the EESD of $n^{-1/2}\tilde{L}_{H_n}$ converges weakly to $\nu_{\Gauss,\frac{1}{r - 1}} \boxplus \nu_{\sc, 1}$.
       \item [(ii)] Suppose $r \to \infty$ such that $r/n \to c \in [0, 1)$. Then the EESD of $(nr)^{-1/2} L_{H_n}$ converges weakly to $\nu_{\Gauss, 1} \boxplus \nu_{\Gauss, c}$ and the EESD of $n^{-1/2}\tilde{L}_{H_n}$ converges weakly to $\nu_{\sc, (1 - c)^2}$.
\end{enumerate}
\end{theorem}
\begin{remark}
    It is well known \cite{bryc2005spectral, ding2010spectral} that the LSD of the Laplacian of Erd\H{o}s-R\'{e}nyi random graphs is the free additive convolution of the standard Gaussian and the standard semi-circle laws. Theorem~\ref{thm:eesd_gaussian_laplacian} generalises this result to the hypergraph setting.
\end{remark}

\subsection{Concentration of the ESD}
In this section, we study concentration of the ESDs around the EESDs. Our main result in this regard is the following.
\begin{theorem}\label{thm:lip_estimate}
~ 
\begin{enumerate}
    \item [(i)] The map $T_1 : (\bbR^M, \| \cdot \|_2) \to (\Mat_n(\bbR), \|\cdot \|_F)$ defined by $T_1(\bx) = n^{-1/2} H_n(\bx)$ is $\frac{r}{\sqrt{n}}$-Lipschitz. 
    \item [(ii)] The map $T_2 : (\bbR^M, \| \cdot \|_2) \to (\Mat_n(\bbR), \|\cdot \|_F)$ defined by $T_2(\bx) = (nr)^{-1/2}L_{H_n(\bx)}$ is $\sqrt{r}$-Lipschitz.
    \item [(iii)] The map $T_3 : (\bbR^M, \| \cdot \|_2) \to (\Mat_n(\bbR), \|\cdot \|_F)$ defined by $T_3(\bx) = n^{-1/2} \tilde{L}_{H_n(\bx)}$ is $\sqrt{\frac{r}{r - 1} + \frac{r^2}{n}}$-Lipschitz.
\end{enumerate}
\end{theorem}
With the aid of Theorem~\ref{thm:lip_estimate}, one can easily obtain concentration of the ESD around the EESD via standard techniques (see, e.g., \cite{guionnet2000concentration}).

Recall that a probability measure $\nu$ on $\bbR$ satisfies a logarithmic Sobolev  inequality (LSI) with (not necessarily optimal) constant $\kappa$, if for any differentiable function $f$,
\[
    \int f^2 \log \frac{f^2}{\int f^2 \,d\nu} \,d\nu \le 2\kappa \int |f'|^2 \, d\nu.
\]
We say in this case that $\nu$ satisfies $\LSI_{\kappa}$. Via the so-called Herbst argument, a measure $\nu$ satisfying an LSI can be shown to possess sub-Gaussian tails (see, e.g., \cite{ledoux2006concentration}). Examples of probability measures satisfying an LSI include the Gaussian \cite{gross1975logarithmic} and any probability measure $\nu$ that is absolutely continuous with respect to the Lebesgue measure and satisfies the so-called Bobkov and G\"{o}tze condition \cite{bobkov1999exponential}. See, e.g., \cite{ledoux2006concentration} for more on logarithmic Sobolev inequalities in the context of concentration of measure.
\begin{corollary}(Concentration of the ESD)\label{cor:conc_esd}
~ 
Suppose each entry of $\bY = (Y_\ell)_{1 \le \ell \le M}$ satisfies $\LSI_{\kappa}$ for some $\kappa > 0$. There are universal constants $C_j, \hat{C}_j, \tilde{C}_j > 0$, $j = 1, 2$, such that for any $t > 0$,
\begin{enumerate}
    \item[(i)] $\bbP(d_{\BL}(\mu_{n^{-1/2}H_n(\bY)},\mubar_{n^{-1/2}H_n(\bY)}) > t) \le \frac{C_1}{t^{3/2}} \exp\big(- \frac{C_2 n^2 t^2}{\kappa r^2}\big)$;
    \item [(ii)] $\bbP(d_{\BL}(\mu_{(nr)^{-1/2} L_{H_n(\bY)}}, \mubar_{(nr)^{-1/2} L_{H_n(\bY)}}) > t) \le \frac{\hat{C}_1}{t^{3/2}} \exp\big(-\frac{\hat{C}_2 n t^2}{\kappa r}\big)$;
    \item [(iii)] $\bbP(d_{\BL}(\mu_{n^{-1/2}\tilde{L}_{H_n(\bY)}},\mubar_{n^{-1/2}\tilde{L}_{H_n(\bY)}}) > t) \le \frac{\tilde{C}_1}{t^{3/2}} \exp\big(- \frac{\tilde{C}_2 \min\{n, n^2/r^2\} t^2}{\kappa}\big)$.
\end{enumerate}
Suppose now that the entries of $\bY$ are uniformly bounded by $K > 0$. There exist absolute constants $C_j$, $\hat{C}_j$, $\tilde{C}_j > 0$, $j = 3, 4, 5, 6$, such that with $\delta_1(n) = \frac{C_5 K}{n}$, $\hat{\delta}_1(n) = \frac{\hat{C}_5 K}{n}$, $\tilde{\delta}_1(n) = \frac{\tilde{C}_5 K}{n}$ and $Q = C_6 K$, $\hat{Q} = \hat{C}_6 K$, $\tilde{Q} = \tilde{C}_6 K$ we have for any $t, \hat{t}, \tilde{t} > 0$ satisfying $t > (C_3 (Q + \sqrt{t}) \delta_1(n))^{2/5}$, $\hat{t} > (\hat{C}_3 (\hat{Q} + \sqrt{\hat{t}}) \hat{\delta}_1(n))^{2/5}$ and $ \tilde{t} > ( \tilde{C}_3 (\tilde{Q} + \sqrt{\tilde{t}})\tilde{\delta}_1(n))^{2/5}$ that
\begin{enumerate}
    \item[(iv)] $\bbP(d_{\BL}(\mu_{n^{-1/2}H_n(\bY)}, \mubar_{n^{-1/2}H_n(\bY)}) > t) \le \frac{C_3 (Q + \sqrt{t})}{t^{3/2}}\exp\big(- \frac{C_4 n^2}{r^2} \cdot \big(\frac{t^{5/2}}{C_3 (Q + \sqrt{t})} - \delta_1(n)\big)^2\big)$;
    \item[(v)] $\bbP(d_{\BL}(\mu_{(nr)^{-1/2} L_{H_n(\bY)}}, \mubar_{(nr)^{-1/2} L_{H_n(\bY)}}) > \hat{t}) \le \frac{\hat{C}_3 (\hat{Q} + \sqrt{\hat{t}})}{\hat{t}^{3/2}}\exp\big(- \frac{\hat{C}_4 n}{r} \cdot \big(\frac{\hat{t}^{5/2}}{\hat{C}_3 (\hat{Q} + \sqrt{\hat{t}})} - \hat{\delta}_1(n)\big)^2\big)$;
    \item[(vi)] $\bbP(d_{\BL}(\mu_{n^{-1/2}\tilde{L}_{H_n(\bY)}}, \mubar_{n^{-1/2}\tilde{L}_{H_n(\bY)}}) > \tilde{t}) \le \frac{\tilde{C}_3 (\tilde{Q} + \sqrt{\tilde{t}})}{\tilde{t}^{3/2}}\exp\big(- \tilde{C}_4 \cdot \min\{n, \frac{n^2}{r^2}\} \cdot \big(\frac{\tilde{t}^{5/2}}{\tilde{C}_3 (\tilde{Q} + \sqrt{\tilde{t}})} - \tilde{\delta}_1(n)\big)^2\big)$.
\end{enumerate}
\end{corollary}

\begin{remark}
    We note that the concentration inequalities of Corollary~\ref{cor:conc_esd} are only effective in the regime $r/n \to 0$. As such we are able to show in-probability convergence of the ESDs only in this regime.
\end{remark}

\begin{corollary}\label{cor:conv_esd}
    Suppose that the entries of $\bY = (Y_\ell)_{1 \le \ell \le M}$ satisfy Assumption~\ref{ass:tail_of_entries}. Further suppose that they are either uniformly bounded or each of them satisfy $\LSI_{\kappa}$ for some $\kappa > 0$.
    \begin{enumerate}
        \item [(i)] If $r/n \to 0$, then
            \[
                d_{\BL}(\mu_{n^{-1/2} H_n(\bY)}, \nu_{\sc, 1}) \xrightarrow{p} 0,
            \]
            and if $\frac{r \sqrt{\log n}}{n} \to 0$, then
            \[
                d_{\BL}(\mu_{n^{-1/2} H_n(\bY)}, \nu_{\sc, 1}) \xrightarrow{\text{a.s.}} 0.
            \]
        \item [(ii)] For the Laplacian $L_{H_n}$, we have the following:
            For fixed $r$,
            \[
                d_{\BL}(\mu_{n^{-1/2} L_{H_n(\bY)}}, \nu_{\Gauss, r - 1} \boxplus \nu_{\sc, 1}) \xrightarrow{\text{a.s.}} 0.
            \]
            If $r \to \infty$ and $\frac{r}{n} \to 0$, then
            \[
                d_{\BL}(\mu_{(nr)^{-1/2} L_{H_n(\bY)}}, \nu_{\Gauss, 1}) \xrightarrow{p} 0,
            \]
            and if $r \to \infty$ and $\frac{r \log n}{n} \to 0$, then
            \[
                d_{\BL}(\mu_{(nr)^{-1/2} L_{H_n(\bY)}}, \nu_{\Gauss, 1}) \xrightarrow{\text{a.s}} 0.
            \]
        \item [(iii)] For the Laplacian $\tilde{L}_{H_n}$, we have the following:
            For fixed $r$,
            \[
                d_{\BL}(\mu_{n^{-1/2} \tilde{L}_{H_n(\bY)}}, \nu_{\Gauss,\frac{1}{r - 1}} \boxplus \nu_{\sc, 1}) \xrightarrow{\text{a.s.}} 0.
            \]
            If $r \to \infty$ and $\frac{r}{n} \to 0$, then
            \[
                d_{\BL}(\mu_{n^{-1/2} \tilde{L}_{H_n(\bY)}}, \nu_{\sc, 1}) \xrightarrow{p} 0,
            \]
            and if $r \to \infty$ and $\frac{r\sqrt{\log n}}{n} \to 0$, then 
            \[
                d_{\BL}(\mu_{n^{-1/2} \tilde{L}_{H_n(\bY)}}, \nu_{\sc, 1}) \xrightarrow{\text{a.s.}} 0.
            \]
    \end{enumerate}
\end{corollary}

\begin{remark}
Recall that a $d$-dimensional random vector $\bX$ satisfies a Poincar\'{e} inequality with constant $\sigma^2$ if for any bounded smooth function $g$ on $\bbR^d$ one has
\[
    \Var(g(\bX)) \le \sigma^2 \bbE \|\nabla g(\bX)\|_2^2,
\]
We say that $\bX$ satisfies $\PI(\sigma^2)$. Suppose each of the entries $Y_{\ell}$ satisfies $\PI(\sigma^2)$. Combining Theorem~\ref{thm:lip_estimate} with Lemma~7.1 of \cite{bobkov2010concentration}, it follows that the $n$-dimensional vector of the eigenvalues of $n^{-1/2}H_n(\bY)$ satisfies $\PI(\sigma^2 r^2 / n)$. Let $d_{\KS}(\mu, \nu) := \sup_{x} |F_\mu(x) - F_{\nu}(x)|$ denote the Kolmogorov-Smirnov distance between two probability measures $\mu$ and $\nu$ on $\bbR$, where $F_{\mu}(x) = \mu((-\infty, x])$ is the distribution function of $\mu$. As a consequence of Corollary~6.2 of \cite{bobkov2010concentration}, we have
\[
    \bbE [d_{\KS}(\mu_{n^{-1/2}H_n(\bY)}, \mubar_{n^{-1/2} H_n(\bY)})] \le C \bigg(\frac{\sigma r}{n}\bigg)^{2/3} \log^2\bigg(\frac{n}{\sigma r}\bigg)
\]
for some constant $C > 0$. Unfortunately, $C$ depends on the Lipschitz constant of $F_{\mubar_{n^{-1/2}H_n(\bY)}}$. 
\end{remark}

\subsection{Asymptotics for the edge eigenvalues}\label{sec:edge}
Using the low rank representation of a Gaussian GHAM, we study its spectral edge.  Our results generalise the results of \cite{ding2010spectral} for Erd\H{o}s-R\'{e}nyi random graphs.

For any $n \times n$ symmetric matrix $B$, let $\lambda_1(B) \ge \lambda_2(B) \ge \cdots \ge \lambda_n(B)$ denote its eigenvalues in non-increasing order. We first present our results on the extreme eigenvalues of a Gaussian GHAM. 
\begin{theorem}\label{thm:extreme_eigen}
     Let $H_n$ be a Gaussian GHAM. Let $c_n := r/n$ and suppose that $\limsup c_n <1$. Let $\zeta$ denote a standard Gaussian variable.
     \begin{enumerate}
     \item [(i)]
     If $c_n \to c\in (0,1)$, then
    \begin{align}
        &\frac{\lambda_1(H_n)}{n} \xrightarrow{d} \frac{c}{2}\zeta + \sqrt{\frac{c^2}{4}\zeta^2 + c(1-c)}\label{eq:gham_extreme_1},\\ 
        &\frac{\lambda_n(H_n)}{n} \xrightarrow{d} \frac{c}{2}\zeta -\sqrt{\frac{c^2}{4}\zeta^2 +c(1-c)}. \label{eq:gham_extreme_2}
        \end{align}
    \item [(ii)]
    If $r \to \infty$ and $c_n\to 0$, then
    \begin{align}
        &\frac{\lambda_1(H_n)}{\sqrt{nr}} \xrightarrow{d} 1, \\
        &\frac{\lambda_n(H_n)}{\sqrt{nr}} \xrightarrow{d} -1.
    \end{align}
    \item [(iii)] If $r$ is fixed, then
    \begin{align}
        \frac{\lambda_1(H_n)}{\sqrt{n}} &\convp
        \begin{cases}
            \sqrt{r-2} + \frac{1}{\sqrt {r-2}} & \text{ if } r \geq 4, \\
            2 &\text{ if } r \leq 3,
        \end{cases} \\
        \frac{\lambda_n(H_n)}{\sqrt{n}}&\convp
        \begin{cases}
            -\sqrt{r-2} - \frac{1}{\sqrt {r-2}} & \text{ if } r \geq 4, \\
            -2 & \text{ if } r \leq 3.
        \end{cases}
    \end{align}
    \item [(iv)]
    Let $k \in \bbN$. Then, 
    \begin{align}
        &\sqrt{\frac{n}{\log n}} \bigg(\frac{\lambda_{1+k}(H_n)}{\sqrt n} - 2(1-c_n)\bigg) = O_P(1), \label{eq:gham_extreme_3}\\ 
        &\sqrt{\frac{n}{\log n}} \bigg(\frac{\lambda_{n-k}(H_n)}{\sqrt n} + 2(1-c_n)\bigg) = O_P(1).  \label{eq:gham_extreme_4}
    \end{align}
    \end{enumerate}
\end{theorem}
As mentioned in Section~\ref{sec:dep_wig}, these results remain valid for a GHAM where the entries $Y_{\ell}$ are uniformly bounded, by say $K$, and $r \ll \frac{n^{1/4}}{\sqrt{K} \log n}$.

Now we present our results on the extreme eigenvalues of the two types of Laplacians of a Gaussian GHAM.
\begin{theorem}\label{thm:extreme_eigen_laplacian}
    Let $H_n$ be a Gaussian GHAM. Let $c_n := r/n$ and suppose that $\limsup c_n < 1$. Let $k \in \bbN$ and $\zeta$ be a standard Gaussian variable.
    \begin{enumerate}
    \item [(A)] We have the following estimates for the extreme eigenvalues of $L_{H_n}$:
    \begin{align} 
        \sqrt{ \log n}\bigg( \frac{\lambda_k(L_{H_n})}{n\sqrt { 2 \log n}} -\sqrt{c_n(1-c_n)} \bigg) &=O_P(1), \label{eq:gham_extreme_5}\\
        \sqrt{ \log n}\bigg( \frac{\lambda_{n+1-k}(L_{H_n})}{n\sqrt { 2 \log n}} +\sqrt{c_n(1-c_n)} \bigg) &=O_P(1). \label{eq:gham_extreme_6}
    \end{align}
    \item [(B)] We have the following estimates for the largest and smallest eigenvalues of $\tilde L_{H_n}$. 
    \vskip5pt
    \begin{enumerate}
    \item [(i)]
    If $r \ll \sqrt{\log n}$, then
    \begin{align}
        \frac{\sqrt {\log n}}{r}\bigg(\frac{(r-1) \lambda_1(\tilde L_{H_n})}{n\sqrt{2 \log n}} -\sqrt{c_n(1-c_n)}\bigg) &= O_P(1), \label{eq:gham_extreme_7}\\
        \frac{\sqrt {\log n}}{r}\bigg(\frac{(r-1) \lambda_n(\tilde L_{H_n})}{n\sqrt{2 \log n}} +\sqrt{c_n(1-c_n)}\bigg) &= O_P(1). \label{eq:gham_extreme_8}
    \end{align}
    \item [(ii)]
    If $c_n \to c \in (0,1)$, then
    \begin{align}
        \frac{\lambda_1(\tilde L_{H_n})}{n} \xrightarrow{d} \frac{c}{2}\zeta +\sqrt{\frac{c^2}{4}\zeta +c(1-c)}, \label{eq:gham_extreme_9}\\
        \frac{\lambda_1(\tilde L_{H_n})}{n} \xrightarrow{d} \frac{c}{2}\zeta -\sqrt{\frac{c^2}{4}\zeta +c(1-c)}. \label{eq:gham_extreme_10}
    \end{align}
    \end{enumerate}
    \item [(C)]
    We have the following estimates for the other extreme eigenvalues of $\tilde L_{H_n}$.
    \vskip5pt
    \begin{enumerate}
    \item [(i)]
    If $r \ll \sqrt n$, then
    \begin{align}
        \min\bigg\{\frac{\sqrt{n \log n}}{r}, \sqrt {\log n} \bigg\}\bigg(\frac{(r-1) \lambda_{1+k}(\tilde L_{H_n})}{n\sqrt{2 \log n}} -\sqrt{c_n(1-c_n)}\bigg) &= O_P(1), \label{eq:gham_extreme_13}\\
        \min\bigg\{\frac{\sqrt{n \log n}}{r}, \sqrt {\log n} \bigg\} \bigg(\frac{(r-1) \lambda_{n-k}(\tilde L_{H_n})}{n\sqrt{2 \log n}} +\sqrt{c_n(1-c_n)}\bigg) &= O_P(1). \label{eq:gham_extreme_14}
    \end{align}
    \item [(ii)]
    If $r \gg \sqrt{n \log n}$, then
    \begin{align}
        \frac{r}{\sqrt{n \log n}}\bigg(\frac{\lambda_{1+k}(\tilde L_{H_n})}{\sqrt n} -2 (1-c_n)\bigg) &=O_P(1), \label{eq:gham_extreme_17}\\ 
        \frac{r}{\sqrt{n \log n}}\bigg(\frac{\lambda_{n-k}(\tilde L_{H_n})}{\sqrt n} +2 (1-c_n)\bigg) &=O_P(1). \label{eq:gham_extreme_18}
    \end{align}
    \end{enumerate}
    \end{enumerate}
\end{theorem}

\section{Proofs}\label{sec:proofs}
\subsection{Proof of Theorem~\ref{thm:univ}}
The proofs are based on Propositions~\ref{prop:bulk_universality} and \ref{prop:bulk_universality_laplacian} below, whose proofs are extensions of the argument given in \cite[Section 2]{chatterjee2005simple} to the GHAM setting. We first recall the main Lindeberg swapping result for \cite{chatterjee2005simple}.
Let $\mathbf{X} = (X_1, \ldots, X_n)$ and $\mathbf{Y} = (Y_1,\ldots, Y_n) $ be two vectors of independent random variables with finite second moments, taking values in some open interval $I$ and satisfying, for each $i$, $\bbE X_i = \bbE Y_i$ and $\bbE X^2_i = \bbE Y^2_i$.
\begin{theorem}[\protect{\cite[Theorem 1.1]{chatterjee2005simple}}]
Let $f: I^n \to \bbR$ be thrice differentiable in each argument. If we set $U= f(\mathbf{X})$ and $V= f(\mathbf{Y})$, then for any thrice differentiable $g : \bbR \to \bbR$ and any $K >0$,
\begin{align*}
    |\bbE g(U) - \bbE g(V)| &\le C_1(g) \lambda_2(f) \sum_{i = 1}^n \big[\bbE[X_i^2 \ind(|X_i|> K)] +  \bbE[Y_i^2 \ind(|Y_i|> K)] \big] \\
    &\qquad + C_2(g) \lambda_3(f) \sum_{i = 1}^n \big[\bbE[|X_i|^3 \ind(|X_i|\le K)] + \bbE[|Y_i|^3 \ind(|Y_i| \le K)] ], 
 \end{align*}
where $C_1(g) = \|g^{\prime}\|_{\infty} + \|g^{\prime \prime }\|_{\infty}$, $C_2(g) = \frac{1}{6} \|g^{\prime}\|_{\infty} + \frac{1}{2}\|g^{\prime \prime }\|_{\infty} + \frac{1}{6}\|g^{\prime \prime \prime}\|_{\infty}$ and 
\[
    \lambda_s(f) ;= \sup\{|\partial^q_i f(\bx)|^{\frac{s}{q}} : 1 \le i \le n, 1 \le q \le s, x \in I^n\},
\]
where $\partial^{q}_i$ denotes $q$-fold differentiation with respect to the $i$-th coordinate. 
\end{theorem}\label{thm:chatterjee}
\begin{proposition}\label{prop:bulk_universality}
Let $z = u + iv \in \bbC_+$. For $\mathbf{x} = (x_1, \ldots, x_{M})^\top$, consider the functions $H_n(\mathbf{x}) = \frac{1}{\sqrt{N}}\sum_{\ell = 1}^M x_{\ell} Q_{\ell}$, $R_n(\mathbf{x}) = (\frac{1}{\sqrt{n}}H_n(\mathbf{x}) - z I)^{-1}$ and $G_n(\mathbf{x}) = \frac{1}{n} \Tr R_n(\mathbf{x})$. Let $\mathbf{Y} = (Y_1, Y_2, \ldots, Y_M)^\top$ and $\mathbf{Z} = (Z_1, Z_2, \ldots, Z_M)^\top$, where the $Y_{\ell}$'s are independent zero mean and unit variance random variables and they satisfy Assumption~\ref{ass:tail_of_entries} and  the $Z_{\ell}$'s are i.i.d. standard Gaussians. Then, we have for any $K > 0$,
\begin{align}
\label{eq:upper_bound_univ} \nonumber
    |\bbE(G_n&(\mathbf{Y})) - \bbE(G_n(\mathbf{Z}))| \\ \nonumber
    &\le 4 \max(v^{-3}, v^{-4}) \frac{r^2(r - 1)^2}{n^2 N} \sum_{\ell = 1}^M \bigg[ \bbE[Y^2_{\ell}\ind(|Y_{\ell}| > K)] + \bbE[Z^2_{\ell}\ind(|Z_{\ell}| > K)] \bigg] \\
    &\quad + 12 \max(v^{-6}, v^{-\frac{9}{2}}, v^{-4}) \frac{r^3(r - 1)^3}{n^{5/2} N^{3/2}} \sum_{\ell = 1}^M \bigg[ \bbE[|Y_{\ell}|^3\ind(|Y_{\ell}| \le K)] + \bbE[|Z_{\ell}|^3\ind(|Z_{\ell}| \le K)] \bigg].
\end{align}
\end{proposition}

\begin{proof}
Let $f(\bx) = G_n(\bx)$. We will apply Theorem~\ref{thm:chatterjee} separately on the real and imaginary parts of $f$. Writing $f(\bx) = \Re f(\bx) + \iota \Im f(\bx)$ (with $\iota := \sqrt{-1}$), it is easy to see that $\max\{|\partial^q_i \Re f|, |\partial^q_i \Im f|\} \le |\partial^q_i f|$ for any $1 \le i \le M$ and $1 \le q \le 3$. As a result,
\[
    \max\big\{\lambda_2(\Re f),  \lambda_2(\Im f) \big\} \le \lambda_2(f) \ \ \text{and} \ \ \max\big\{\lambda_3(\Re f), \lambda_3(\Im f)\big\} \le \lambda_3(f).
\]
In view of this, it is enough to show that 
\begin{align*}
    \lambda_2(f) &= \sup_{\bx}\bigg\{|\partial_{\ell} f(\bx)|^2, |\partial^2_{\ell} f(\bx)| : 1 \le \ell \le M \bigg\} \le 2\max( v^{-3}, v^{-4}) \frac{r^2(r - 1)^2}{n^2N}, \\
    \lambda_3(f) &= \sup_{\bx}\bigg\{ |\partial_\ell f(\bx)|^{3}, |\partial^2_{\ell} f(\bx)|^{3/2}, |\partial^3_{\ell} f(\bx)| : 1 \le \ell \le M \bigg\} \le 6 \max( v^{-6}, v^{-\frac{9}{2}}, v^{-4}) \frac{r^3(r - 1)^3}{n^{5/2}N^{3/2}}.
\end{align*}
First note that
\[
    \frac{\partial^{q} f(\bx)}{\partial x_{\ell}^{q}} = \frac{1}{n} \Tr \frac{\partial^{q} R_n(\bx)}{\partial x_{\ell}^{q}}, \quad q \ge 1.
\]
Differentiating the relation $\big(\frac{1}{\sqrt{n}} H_n(\bx) - z I\big) R_n(\bx) = I$, one obtains that
\begin{align*}
    \frac{\partial R_n(\bx) }{\partial x_{\ell}} &= - \frac{1}{\sqrt{n}}  R_n(\bx) \frac{\partial H_n(\bx) }{\partial x_{\ell}} R_n(\bx), \\
    \frac{\partial^2 R_n(\bx) }{\partial x^2_{\ell}} &= \frac{2}{n}  R_n(\bx) \frac{\partial H_n(\bx) }{\partial x_{\ell}} R_n(\bx) \frac{\partial H_n(\bx) }{\partial x_{\ell}} R_n(\bx), \\
    \frac{\partial^3 R_n(\bx) }{\partial x^3_{\ell}} &= - \frac{6}{n^{\frac{3}{2}}}  R_n(\bx) \frac{\partial H_n(\bx) }{\partial x_{\ell}} R_n(\bx) \frac{\partial H_n(\bx) }{\partial x_{\ell}} R_n(\bx)\frac{\partial H_n(\bx) }{\partial x_{\ell}} R_n(\bx).
\end{align*}
Therefore
\begin{align}
    \frac{\partial f(\bx) }{\partial x_{\ell}} = - \frac{1}{n\sqrt{n}} \Tr\bigg( R_n(\bx) \frac{\partial H_n(\bx) }{\partial x_{\ell}} R_n(\bx)\bigg) = - \frac{1}{n^{\frac{3}{2}}} \Tr\bigg(\frac{\partial H_n(\bx) }{\partial x_{\ell}} R^2_n(\bx)\bigg),
\end{align}
and similarly,
\begin{align}
    \frac{\partial^2 f(\bx) }{\partial x^2_{\ell}} &= \frac{2}{n^2} \Tr\bigg( \frac{\partial H_n(\bx) }{\partial x_{\ell}} R_n(\bx) \frac{\partial H_n(\bx) }{\partial x_{\ell}} R^2_n(\bx)\bigg), \\
    \frac{\partial^3 f(\bx) }{\partial x^3_{\ell}} &= - \frac{6}{n^{\frac{5}{2}}} \Tr\bigg(\frac{\partial H_n(\bx)}{\partial x_{\ell}} R_n(\bx)  \frac{\partial H_n(\bx) }{\partial x_{\ell}} R_n(\bx) \frac{\partial H_n(\bx) }{\partial x_{\ell}} R^2_n(\bx)\bigg).
\end{align}
Observe that 
\[
    \Tr\bigg(\frac{\partial H_n(\bx) }{\partial x_{\ell}} R^2_n(\bx)\bigg) = \sum_{1 \le i, j \le n} \bigg( \frac{\partial H_n(\bx) }{\partial x_{\ell}} \bigg)_{ij} \big( R^2_n(\bx)\big)_{ji}.
\]
Let $\Lambda = \diag( \lambda_1, \ldots, \lambda_n)$, where $\lambda_i$'s are the eigenvalues of $\frac{1}{\sqrt{n}} H_n(\bx)$. Then from the spectral decomposition $\frac{1}{\sqrt{n}}H_n(\bx) = U \Lambda U^\top$, we have $R^2_n(\bx) = U (\Lambda - zI)^{-2} U^{\top}$. Note that $|((\Lambda - zI)^{-2})_{ii}| = |\frac{1}{\lambda_i - z}|^2 \le \frac{1}{v^2}$. Therefore
\[
    |(R^2_n(\bx))_{ij}| = |\sum_{k} ((U - zI)^{-2})_{kk} U_{ik} U_{kj}| \le \frac{1}{v^2} \sum_{k}|U_{ik}| |U_{kj}| \le \frac{1}{v^2} \big(\big(\sum_k U_{ik}^2) \big(\sum_k U_{kj}^2\big)\big)^{1/2} = \frac{1}{v^2}.
\]
Hence
\[
    \bigg|\Tr\bigg(\frac{\partial H_n(\bx) }{\partial x_{\ell}} R^2_n(\bx)\bigg) \bigg| \le \frac{1}{v^2}  \sum_{1 \le i, j \le n} \bigg| \bigg( \frac{\partial H_n(\bx) }{\partial x_{\ell}} \bigg)_{ij} \bigg|.
\]
Note that $\frac{\partial H_n(\bx) }{\partial x_{\ell}} = \frac{1}{\sqrt{N}} Q_{\ell}$ and for any $\ell$,
\[
    \sum_{1 \le i, j \le n} \bigg|\bigg( \frac{\partial H_n(\bx) }{\partial x_{\ell}} \bigg)_{ij} \bigg| = \frac{r(r - 1)}{\sqrt{N}}.
\]
Therefore
\begin{equation}\label{eq:partial_1}
    \bigg\|\frac{\partial f(\bx) }{\partial x_{\ell}} \bigg\|_{\infty} \le \frac{1}{v^2} \frac{r(r - 1)}{n^{3/2}\sqrt{N}}.
\end{equation}
For bounding $\|\frac{\partial^{q} f(\bx)}{\partial x_{\ell}^{q}}\|_{\infty}$, $q = 2, 3$, we will need the the Frobenius and operator norms of $\frac{\partial H_n(\bx)}{\partial x_{\ell}}$. The Frobenius norm is easy to calculate:
\[
    \bigg\|\frac{\partial H_n(\bx)}{\partial x_{\ell}}\bigg\|_F = \bigg\|\frac{1}{\sqrt{N}}Q_{\ell}\bigg\|_F = \sqrt{\frac{r(r - 1)}{N}}.
\]
Also, relabeling the vertices if necessary, we may assume that
\[
    Q_{\ell} = \begin{pmatrix}
        J_r - I_r & 0 \\
        0 & 0
    \end{pmatrix}.
\]
Hence $\|Q_{\ell}\|_{\mathrm{op}} = r - 1$ and so
\[
    \bigg\|\frac{\partial H_n(\bx)}{\partial x_{\ell}}\bigg\|_{\mathrm{op}} = \bigg\|\frac{1}{\sqrt{N}}Q_{\ell}\bigg\|_{\mathrm{op}} = \frac{r - 1}{\sqrt{N}}.
\]
Now,
\begin{align*}
    \Tr\bigg(\frac{\partial H_n(\bx)}{\partial x_{\ell}} &R_n(\bx) \frac{\partial H_n(\bx)}{\partial x_{\ell}} R^2_n(\bx)\bigg) \\
    &\le \bigg\|\frac{\partial H_n(\bx)}{\partial x_{\ell}}\bigg\|_F \cdot \bigg\|R_n(\bx) \frac{\partial H_n(\bx)}{\partial x_{\ell}} R^2_n(\bx)\bigg\|_F \qquad \text{(by Cauchy-Schwartz)} \\
    &\le \bigg\|\frac{\partial H_n(\bx)}{\partial x_{\ell}}\bigg\|_F \cdot \|R_n(\bx)\|_{\mathrm{op}} \cdot \bigg\|\frac{\partial H_n(\bx)}{\partial x_{\ell}} R^2_n(\bx)\bigg\|_F \\
    &\le \bigg\|\frac{\partial H_n(\bx)}{\partial x_{\ell}}\bigg\|_F \cdot \|R_n(\bx)\|_{\mathrm{op}} \cdot \bigg\|\frac{\partial H_n(\bx)}{\partial x_{\ell}}\bigg\|_F \cdot \|R^2_n(\bx)\|_{\mathrm{op}} \\
    &\le \bigg\|\frac{\partial H_n(\bx)}{\partial x_{\ell}}\bigg\|_F^2 \cdot \|R_n(\bx)\|_{\mathrm{op}}^3 \\
    &\le \frac{r(r - 1)}{N} \frac{1}{v^3}.
\end{align*}
This gives
\begin{equation}\label{eq:partial_2}
    \bigg\|\frac{\partial^2 f(\bx) }{\partial x^2_{\ell}} \bigg\|_{\infty} \le \frac{2 r(r - 1)}{n^2 N v^3}.
\end{equation}
For the third derivative, we do the following
\begin{align*}
    \Tr\bigg(\frac{\partial H_n(\bx)}{\partial x_{\ell}} & R_n(\bx)\frac{\partial H_n(\bx)}{\partial x_{\ell}} R_n(\bx) \frac{\partial H_n(\bx)}{\partial x_{\ell}} R^2_n(\bx)\bigg) \\
    &\le \bigg\|\frac{\partial H_n(\bx)}{\partial x_{\ell}} R_n(\bx) \frac{\partial H_n(\bx)}{\partial x_{\ell}} \bigg\|_F \cdot \bigg\| R_n(\bx) \frac{\partial H_n(\bx)}{\partial x_{\ell}} R^2_n(\bx)\bigg\|_F \qquad \text{(by Cauchy-Schwartz)} \\
    &\le \bigg\|\frac{\partial H_n(\bx)}{\partial x_{\ell}}\bigg\|_{\mathrm{op}} \cdot \bigg\|R_n(\bx) \frac{\partial H_n(\bx)}{\partial x_{\ell}}\bigg\|_F \cdot \bigg\|R_n(\bx)\frac{\partial H_n(\bx)}{\partial x_{\ell}} R^2_n(\bx)\bigg\|_F \\
    &\le \bigg\|\frac{\partial H_n(\bx)}{\partial x_{\ell}}\bigg\|_{\mathrm{op}} \cdot \|R_n(\bx)\|_{\mathrm{op}} \cdot \bigg\|\frac{\partial H_n(\bx)}{\partial x_{\ell}}\bigg\|_F \cdot \bigg\|R_n(\bx)\frac{\partial H}{\partial x_{\ell}} R^2\bigg\|_F \\
    &\le \bigg\|\frac{\partial H_n(\bx)}{\partial x_{\ell}}\bigg\|_{\mathrm{op}} \cdot \bigg\|\frac{\partial H_n(\bx)}{\partial x_{\ell}}\bigg\|_F^2 \cdot \|R_n(\bx)\|_{\mathrm{op}}^4 \\
    &\le \frac{r - 1}{\sqrt{N}} \cdot \frac{r(r - 1)}{N} \cdot \frac{1}{v^4} \\
    &\le \frac{r(r - 1)^2}{N^{3/2}} \cdot \frac{1}{v^4}.
\end{align*}
This yields
\begin{equation}\label{eq:partial_3}
    \bigg\|\frac{\partial^3 f(\bx) }{\partial x^3_{\ell}} \bigg\|_{\infty} \le \frac{6 r(r - 1)^2}{n^{5/2} N^{3/2} v^4}.
\end{equation}
Using \eqref{eq:partial_1}, \eqref{eq:partial_2} and \eqref{eq:partial_3}, we obtain
\begin{align*}
    \lambda_2(f) &\le  2\max( v^{-3}, v^{-4}) \frac{r^2(r - 1)^2}{n^2N}, \\
    \lambda_3(f) &\le 6 \max( v^{-6}, v^{-\frac{9}{2}}, v^{-4}) \frac{r^3(r - 1)^3}{n^{\frac{5}{2}}N^{\frac{3}{2}}}.
\end{align*}
This completes the proof.
\end{proof}

For the two Laplacian matrices $(nr)^{-1/2}L_{H_n}$ and $n^{-1/2}\tilde{L}_{H_n}$, we have the following results. Their proofs are similar to the proof of Proposition~\ref{prop:bulk_universality} and hence are relegated to the appendix.
\begin{proposition}\label{prop:bulk_universality_laplacian_orig}
Let $z = u + iv \in \bbC_+$. For $\bx = (x_1, \ldots, x_{M})^\top$, consider the functions $H_n(\bx) = \frac{1}{\sqrt{N}}\sum_{\ell = 1}^M x_{\ell} Q_{\ell}$, $L_{H_n}(\bx) \equiv L_{H_n(\bx)} = \diag(H_n(\bx) \bone) - H_n(\bx)$, $\hat{R}_n(\bx) = (\frac{1}{\sqrt{nr}}L_{H_n}(\bx) - z I)^{-1}$ and $\hat{G}_n(\bx) = \frac{1}{n} \Tr \hat{R}_n(\bx)$. Let $\bY = (Y_1, Y_2, \ldots, Y_M)^\top$ and $\bZ = (Z_1, Z_2, \ldots, Z_M)^\top$, where the $Y_{\ell}$'s are independent random variables with zero mean and unit variance satisfying Assumption~\ref{ass:tail_of_entries} and  the $Z_{\ell}$'s are i.i.d. standard Gaussians. Then, we have for any $K > 0$,
\begin{align}
\label{eq:upper_bound_univ_laplacian_orig} \nonumber
    |\bbE(\hat{G}_n&(\mathbf{Y})) - \bbE(\hat{G}_n(\mathbf{Z}))| \\ \nonumber
    &\le 8 \max(v^{-3}, v^{-4}) \frac{r^{3/2}}{n^2 N} \sum_{\ell = 1}^M \bigg[ \bbE[Y^2_{\ell}\ind(|Y_{\ell}| > K)] + \bbE[Z^2_{\ell}\ind(|Z_{\ell}| > K)] \bigg] \\
    &\quad + 16 \max(v^{-6}, v^{-\frac{9}{2}}, v^{-4}) \frac{r^{9/2}}{n^{5/2} N^{3/2}} \sum_{\ell = 1}^M \bigg[ \bbE[|Y_{\ell}|^3\ind(|Y_{\ell}| \le K)] + \bbE[|Z_{\ell}|^3\ind(|Z_{\ell}| \le K)] \bigg].
\end{align}
\end{proposition}
\begin{proposition}\label{prop:bulk_universality_laplacian}
Let $z = u + iv \in \bbC_+$. For $\bx = (x_1, \ldots, x_{M})^\top$, consider the functions $H_n(\bx) = \frac{1}{\sqrt{N}}\sum_{\ell = 1}^M x_{\ell} Q_{\ell}$, $\tilde{L}_{H_n}(\bx) \equiv \tilde{L}_{H_n(\bx)} = \frac{\diag(H_n(\bx) \bone)}{r - 1} - H_n(\bx)$, $\tilde{R}_n(\bx) = (\frac{1}{\sqrt{n}}\tilde{L}_{H_n}(\bx) - z I)^{-1}$ and $\tilde{G}_n(\bx) = \frac{1}{n} \Tr \tilde{R}_n(\bx)$. Let $\bY = (Y_1, Y_2, \ldots, Y_M)^\top$ and $\bZ = (Z_1, Z_2, \ldots, Z_M)^\top$, where the $Y_{\ell}$'s are independent random variables with zero mean and unit variance satisfying Assumption~\ref{ass:tail_of_entries} and  the $Z_{\ell}$'s are i.i.d. standard Gaussians. Then, we have for any $K > 0$,
\begin{align}
\label{eq:upper_bound_univ_laplacian} \nonumber
    |\bbE(\tilde{G}_n&(\mathbf{Y})) - \bbE(\tilde{G}_n(\mathbf{Z}))| \\ \nonumber
    &\le 4 \max(v^{-3}, v^{-4}) \frac{r^4}{n^2 N} \sum_{\ell = 1}^M \bigg[ \bbE[Y^2_{\ell}\ind(|Y_{\ell}| > K)] + \bbE[Z^2_{\ell}\ind(|Z_{\ell}| > K)] \bigg] \\
    &\quad + 12 \max(v^{-6}, v^{-\frac{9}{2}}, v^{-4}) \frac{r^6}{n^{5/2} N^{3/2}} \sum_{\ell = 1}^M \bigg[ \bbE[|Y_{\ell}|^3\ind(|Y_{\ell}| \le K)] + \bbE[|Z_{\ell}|^3\ind(|Z_{\ell}| \le K)] \bigg].
\end{align}
\end{proposition}

\begin{proof}[Proof of Theorem~\ref{thm:univ}]
In the notation of Proposition~\ref{prop:bulk_universality} we have $S_{\mubar_{n^{-1/2}H_n(\bY)}} = \bbE G_n(\bY)$ and $S_{\mubar_{n^{-1/2}H_n(\bZ)}} = \bbE G_n(\bZ)$. We take $K = \varepsilon K_n$ where $K_n = \frac{\sqrt{nN}}{r^4}$. Observe that
\begin{align*}
    \sum_{\ell = 1}^M \bigg[ \bbE[|Y_{\ell}|^3\ind(|Y_{\ell}| \le \varepsilon K_n)] &+ \bbE[|Z_{\ell}|^3\ind(|Z_{\ell}| \le \varepsilon K_n)] \bigg] \\
    &\le \varepsilon K_n \sum_{\ell = 1}^M \bigg[ \bbE[|Y_{\ell}|^2\ind(|Y_{\ell}| \le K_n)] + \bbE[|Z_{\ell}|^2\ind(|Z_{\ell}| \le K_n)] \bigg] \\ &\le 2 \varepsilon K_n M.
\end{align*}
Therefore the second term in \eqref{eq:upper_bound_univ} is at most
\begin{align*}
    \frac{2 r^3 (r - 1)^3 \varepsilon K_n M}{n^{5/2}N^{3/2}} \le \frac{2 \varepsilon K_n r^4}{\sqrt{n N}}.
\end{align*}
Therefore, with our choice of $K_n$, the second term in \eqref{eq:upper_bound_univ} is at most $2\varepsilon$. Now, i.i.d standard Gaussian random variables satisfy Assumption~\ref{ass:tail_of_entries}. By Assumption~\ref{ass:tail_of_entries} on $(Y_{\ell})_{1 \le \ell \le M}$, the first term in \eqref{eq:upper_bound_univ} goes to zero.

The proof for the Laplacian matrix $n^{-1/2}\tilde{L}_{H_n}$ follows from Proposition~\ref{prop:bulk_universality_laplacian} in an identical fashion.

To deal with $(nr)^{-1/2} L_{H_n}$, it follows from Proposition~\ref{prop:bulk_universality_laplacian_orig} via the same argument as above that we want
\begin{equation}\label{eq:cond_lap_orig}
    \lim_{n \to \infty} \frac{r^{3/2}}{n^2 N} \sum_{\ell = 1}^M \bbE[|Y_{\ell}|^2 \ind(|Y_{\ell}| > \varepsilon K_n')] = 0,
\end{equation}
where $K_n' = \frac{\sqrt{nN}}{r^{5/2}}$. This is precisely Assumption~\ref{ass:tail_of_entries_variation}.
\end{proof}
\subsection{Proofs of Theorems~\ref{thm:eesd_gaussian} and \ref{thm:eesd_gaussian_laplacian}}
We first state and prove a perturbation inequality that will be used repeatedly in the proofs. For this we will use the $1$-Wasserstein metric. Recall that for $q \ge 1$, the $q$-Wasserstein distance between probability measures $\mu$ and $\nu$ on $\bbR$ is given by 
\[
    d_{W_q}(\mu, \nu) := \inf_{\pi} \int |x - y|^q \, d\pi(x, y), 
\]
where the infimum is over all coupings $\pi$ of $\mu$ and $\nu$. By the Kantorovich-Rubinstein duality (see, e.g., \cite[Chap. 11]{dudley2018real}), one also has the following representation for $d_{W_1}$:
\[
    d_{W_1}(\mu, \nu) = \sup_{\|f\|_{\Lip} \le 1} \bigg\{\bigg|\int f \,d\mu - \int f \,d\nu\bigg|\bigg\}.
\]
From this it is evident that $d_{\BL} \le d_{W_1}$.
\begin{lemma}\label{lem:wass-1-bound}
    Let $A$ and $B$ be $n \times n$ Hermitian random matrices. Then
    \begin{equation}\label{eq:wass-1-bound}
        d_{W_1}(\mubar_A, \mubar_B) \le \frac{\bbE \|A - B\|_F}{\sqrt{n}} \le \bigg[\frac{\bbE \|A - B\|_F^2}{n}\bigg]^{1/2}.
    \end{equation}
\end{lemma}
\begin{proof}
Using the Kantorovich-Rubinstein duality,
\begin{align*}
    d_{W_1}(\mubar_{A}, \mubar_{B}) &= \sup_{\|f\|_{\Lip} \le 1} \bigg|\int f \, d\mubar_{A} - \int f \, d\mubar_{B}\bigg| \\
    &= \sup_{\|f\|_{\Lip} \le 1} \bigg|\bbE \int f \, d\mu_{A} - \bbE \int f \, d\mu_{B}\bigg| \\
    &\le \bbE \sup_{\|f\|_{\Lip} \le 1} \bigg|\int f \, d\mu_{A} - \int f \, d\mu_{B}\bigg| \\
    &= \bbE [d_{W_1}(\mu_{A}, \mu_{B})].
\end{align*}
The first inequality in \eqref{eq:wass-1-bound} follows now from the fact that $d_{W_1} \le d_{W_2}$ and the Hoffman-Wielandt inequality (Lemma~\ref{lem:hoffman-wielandt}). The final inequality in \eqref{eq:wass-1-bound} is just Jensen's inequality applied to the function $x \mapsto \sqrt{x}$.
\end{proof}

Recall the L\'{e}vy distance between two probability measures $\mu$ and $\nu$ on $\bbR$:
\[
    d_{\levy}(\mu, \nu) := \inf\{\epsilon >0 : F_{\mu}(x - \epsilon) - \epsilon \le G_{\nu}(x) \le  F_{\mu}(x + \epsilon) + \epsilon , \forall x \in \bbR\}.
\]
It is well known that (see, e.g., \cite{gibbs2002choosing})
\[
    d_{\BL} \le 2 d_{\levy} \le 2 d_{\KS}.
\]
We will make use of the inequality $d_{\BL} \le 2 d_{\KS}$ several times below.

\begin{proof}[Proof of Theorem~\ref{thm:eesd_gaussian}]
Consider i.i.d. standard Gaussian random variables $U$, $(V_i)_{1 \le i \le n}$ and an independent GOE random matrix $n^{-1/2} Z_n$ (i.e. $(Z_{ij})_{1 \le i < j \le n}$ are i.i.d. standard Gaussians and $(Z_{ii})_{1 \le i \le n}$ is another independent collection of i.i.d. Gaussians with variance $2$). Consider a symmetric matrix $G_n$ with entries
\[
    G_{n,ij} = \alpha_n U + \beta_n (V_i + V_j) + \theta_n Z_{ij}.
\]
Then we may write $G_n$ as follows:
\begin{equation}\label{eq:GHAM_decomp}
    G_n = \alpha_n U_n \bone_n \bone_n^\top + \beta_n (\bV \bone_n^\top + \bone_n \bV^\top) + \theta_n Z_n,
\end{equation}
where $\bone_n$ is an $n \times n$ vector of $1$'s, $\bV = (V_i)_{1\le \ell \le M}$. We shall choose $\alpha_n$, $\beta_n$, $\theta_n$ in such a way that 
\begin{equation}\label{eq:GHAM_decomp_w/o_diag}
    G'_n = G_n- \diag({G_n})
\end{equation}
yields the same covariance structure as a GHAM. To this end, note that for $i \neq j$ and $i' \neq j'$,
\begin{align*}
   \Cov(G'_{n,ij}, G'_{n,i'j'}) &= \begin{cases}
        \alpha_n^2 & \text{if } |\{i, j\} \cap \{i', j'\}| = 0, \\
        \alpha_n^2 + \beta_n^2 & \text{if } |\{i, j\} \cap \{i', j'\}| = 1, \\
        \alpha_n^2 + 2 \beta_n^2 + \theta_n^2 & \text{if } |\{i, j\} \cap \{i', j'\}| = 2.
   \end{cases}
\end{align*}
In addition, 
\[
    \Var(G_{n,ii}) = \alpha_n^2 + 4 \beta_n^2 + 2\theta_n^2.
\]
To match this with our Gaussian GHAM, we must set
\begin{align*}
    \alpha_n^2 &= \rho_n, \\
    \alpha_n^2 + \beta_n^2 &= \gamma_n, \\
    \alpha_n^2 + 2\beta_n^2 + \theta_n^2 &= 1,
\end{align*}
which yields
\[
    \alpha_n = \sqrt{\rho_n}, \quad \beta_n = \sqrt{\gamma_n - \rho_n}, \quad \theta_n = \sqrt{1 - 2 \gamma_n + \rho_n}. 
\]
As $r/n \to c$, we have $\gamma_n \to c$ and $\rho_n \to c^2$. It follows that $\theta_n \to 1 -c$. With these parameter choices, the entries of $G'_n$ and $H_n$ have the same joint distribution. Now we write
\begin{align*}
    d_{\BL}(\mubar_{n^{-1/2} H_n}, \nu_{\sc, 1 - c}) &\le d_{\BL} (\mubar_{n^{-1/2} H_n}, \mubar_{n^{-1/2} G'_n}) + d_{\BL}(\mubar_{n^{-1/2} G'_n}, \mubar_{n^{-1/2} G_n)}) \\
    &\qquad\qquad + d_{\BL}(\mubar_{n^{-1/2} G_n}, \mubar_{n^{-1/2} \theta_n Z_n}) + d_{\BL} (\mubar_{n^{-1/2} \theta_n Z_n}, \mubar_{n^{-1/2} (1-c) Z_n}) \\
    &\qquad\qquad\qquad\qquad + d_{\BL}( \mubar_{n^{-1/2} (1-c) Z_n}, \nu_{\sc, 1 - c}).
\end{align*}
Now we bound each of the terms on the right hand side. The first term vanishes since the entries of $H_n$ and $G'_n$ have the same joint distribution.

For the second term, we use the inequality $d_{\BL} \le d_{W_1}$ and then appeal to Lemma~\ref{lem:wass-1-bound} to get
\[
    d_{\BL}(\mubar_{n^{-1/2} G'_n}, \mubar_{n^{-1/2} G_n}) \le d_{W_1}(\mubar_{n^{-1/2} G'_n}, \mubar_{n^{-1/2} G_n}) \le \frac{(\bbE \|\diag(G_n)\|_F^2)^{1/2}}{n}  = \sqrt{\frac{\Var(G_{n, 11})}{n}} \le \sqrt{\frac{4}{n}},
\]
where the last inequality holds because $\Var(G_{n, 11}) = \alpha_n^2 + 4 \beta_n^2 + 2 \theta_n^2 = 1 + 2 \beta_n^2 + \theta_n^2 \le 4$.

For the third term, using the fact that $\rank(G_n - \theta_n Z_n) \le 2$, we obtain
\begin{align*}
    d_{\BL}(\mubar_{n^{-1/2} G_n}, \mubar_{n^{-1/2} \theta_n Z_n}) & \le 
    2d_{\KS}(\mubar_{n^{-1/2} G_n}, \mubar_{n^{-1/2} \theta_n Z_n}) \\
    &= \sup_x |\bbE F_{\mu_{n^{-1/2}G_n}}(x)-\bbE F_{\mu_{n^{-1/2}\theta_n Z_n}}(x)| \\
    &\le \bbE \sup_x |F_{\mu_{n^{-1/2}G_n}}(x)- F_{\mu_{n^{-1/2}\theta_n Z_n}}(x)| \\
    &= \bbE [ d_{\KS}(\mu_{n^{-1/2} G_n}, \mu_{n^{-1/2} \theta_n Z_n}) ] \\
    &\le \bbE \bigg[\frac{\rank(n^{-1/2}(G_n - \theta_n Z_n))}{n} \bigg] \quad \text{(using Lemma~\ref{lem:rank_ineq})}\\
    &\le \frac{2}{n}.
\end{align*}
The fourth term is bounded using the same strategy as the second term:
\[
    d_{\BL} (\mubar_{n^{-1/2} \theta_n Z_n}, \mubar_{n^{-1/2} (1-c) Z_n}) \le |\theta_n -(1-c)| \frac{1}{n}(\mathbb{E}\|Z_n\|^2_F)^{1/2} = |\theta_n-(1-c)|.
\]
Since $\theta_n$, $(1-c)$ are bounded away form 0 and bounded above by $1$, using the Lipschitzness of the function $x \mapsto \sqrt{x}$ away from $0$,
\[
    |\theta_n-(1-c)| =\bigg|\sqrt{1-2\bigg(\frac{r}{n}\bigg) +\bigg(\frac{r}{n}\bigg)^2 +O\bigg(\frac{1}{n}\bigg)}-\sqrt{1-2c+c^2}\bigg| = O\bigg(\max\bigg\{\bigg|\frac{r}{n}-c\bigg|,\frac{1}{n}\bigg\}\bigg).
\]
Therefore
\[
    d_{\BL} (\mubar_{n^{-1/2} \theta_n Z_n}, \mubar_{n^{-1/2} (1-c) Z_n}) \le O\bigg(\max\bigg\{\bigg|\frac{r}{n}-c\bigg|,\frac{1}{n}\bigg\}\bigg).
\]
For the final term, using Theorem~1.6 of \cite{gotze2018local} we get that
\[
    d_{\BL}(\mubar_{n^{-1/2} (1-c) Z_n}, \nu_{\sc, 1 - c}) \le 2 d_{\KS}( \mubar_{n^{-1/2} (1-c) Z_n}, \nu_{\sc, 1 - c}) \le \frac{C}{n}.
\]
for some absolute constant $C > 0$. This concludes the proof.
\end{proof}

Now we restate Theorem 2.1 of \cite{pastur2000law} in our notation for orthogonally invariant matrices (see the discussion on page 280 of \cite{pastur2000law}).
\begin{theorem}\protect{\cite[Theorem 2.1]{pastur2000law}}\label{thm:pastur_free_conv}
    Let $B_n = U_n^\top F_{1,n} U_n + V_n^\top F_{2,n} V_n$, where $F_{1,n}$ and $F_{2,n}$ are (potentially) random $n \times n$ symmetric matrices with arbitrary distributions and $U_n, V_n$ are orthogonal matrices uniformly distributed over the orthogonal group $\cO(n)$ according to the Haar measure. We assume that the matrices $F_{1, n}, F_{2, n}, U_n, V_n$ are independent. Assume that the ESDs $\mu_{n^{-1/2} F_{r,n}}, r = 1, 2,$ converge weakly in probability as $n \to \infty$ to non-random probability measures $\mu_r, r = 1, 2$, respectively and that 
    \[
        \sup_n \int |\lambda| \, \mubar_{n^{-1/2} F_{r, n}}(d\lambda) < \infty
    \]
    for at least one value of $r \in \{1, 2\}$. Then $\mu_{n^{-1/2} B_n}$ converges weakly in probability to $\mu_1 \boxplus \mu_2$, the free additive convolution of the measures $\mu_1$ and $\mu_2$.
\end{theorem}
\begin{proof}[Proof of Theorem~\ref{thm:eesd_gaussian_laplacian}]
    To prove (i), we use the low rank representation of a Gaussian GHAM. We only give the details for $\tilde{L}_{H_n}$. The stated result for $L_{H_n}$ follows in a similar manner.

    First note that
    \[
        \tilde{L}_{G_n} - \tilde{L}_{G_n'} = -\frac{r - 2}{r - 1} \diag(G_n),
    \]
    and thus
    \[
        d_{W_1}(\mubar_{n^{-1/2}\tilde{L}_{G'_n}}, \mubar_{n^{-1/2}\tilde{L}_{G_n}})^2 \le \frac{(r - 2)^2}{n^2 (r - 1)^2}\bbE \|\diag(G_n)\|_F^2 = O\bigg(\frac{1}{n}\bigg).
    \]
    Therefore it is enough to consider $\tilde{L}(G_n)$. Since the map $X \mapsto \tilde{L}_X$ is linear, we have
    \[
        \tilde{L}_{G_n} = \tilde{L}_{P_n} + \theta_n \tilde{L}_{Z_n}.
    \]
    where $P_n = \alpha_n U \bone\bone^\top + \beta_n (\bone \bV^\top + \bV \bone^\top)$. Further, using the rank-inequality (Lemma~\ref{lem:rank_ineq}), it is clear that we may focus on
    \[
        X_n = \frac{\diag(P_n \bone)}{r - 1} + \theta \tilde{L}_{Z_n}.
    \]
    Now
    \[
        \diag(P_n \bone) = \diag(n \alpha_n U \bone + n\beta_n \bar{V} \bone + n\beta_n \bV).
    \]
    Consider the matrix
    \[
        \tilde{X}_n = \frac{n\beta_n}{r - 1} \diag(\bV) + \theta_n L_{Z_n}.
    \]
    Using Lemma~\ref{lem:wass-1-bound}, we see that
    \[
        d_{W_1}(\mubar_{n^{-1/2} \tilde{X}_n}, \mubar_{n^{-1/2} X_n})^2 \le \frac{1}{n^2(r - 1)^2} \bbE\|\diag(n\alpha_n U \bone + n\beta_n \bar{V} \bone)\|_F^2 = \frac{n}{(r - 1)^2} \bbE(\alpha_n U + \beta_n \bar{V})^2.
    \]
    Now
    \[
        \bbE(\alpha_n U + \beta_n \bar{V})^2 = \Var (\alpha_n U + \beta_n \bar{V}) = \alpha_n^2 + \beta_n^2 / n  = O(1/n^2).
    \]
    Hence
    \[
        d_{W_1}(\mubar_{n^{-1/2} \tilde{X}_n}, \mubar_{n^{-1/2} X_n}) = O_P\bigg(\frac{1}{\sqrt{n}}\bigg).
    \]
    In fact, since $\beta_n = \frac{\sqrt{r - 2}}{\sqrt{n}} +O(1/n)$ and $\theta_n = O(1/n)$, it is enough to consider (again by Hoffman-Wielandt) the matrix
    \[
        \hat{W}_n = \frac{\sqrt{n(r - 2)}}{r - 1} \diag(\bV) + \tilde{L}_{Z_n} = \frac{\sqrt{n(r - 2)}}{r - 1} \diag(\bV) + \frac{\diag(Z_n\bone)}{r - 1} - Z_n.
    \]
    By a modification of the proof of Lemma~4.12 of \cite{bryc2005spectral} one can instead consider the matrix (see Lemma~\ref{lem:BDJ_replacement} for a precise statement of this replacement and its proof)
    \[
        \tilde{W}_n = \frac{\sqrt{n(r - 2)}}{r - 1} \diag(\bV) + \frac{\sqrt{n} \, \diag(\bg)}{r - 1} + Z_n, 
    \]
    where $\bg$ is a vector of i.i.d. standard Gaussians independent of $\bV$ and $Z_n$. Recall that $n^{-1/2} Z_n$ is a GOE random matrix which is orthogonally invariant.
    It is clear using the strong law of large numbers that the ESD of
    \[
        n^{-1/2} \bigg[\frac{\sqrt{n(r - 2)}}{r - 1} \diag(\bV) + \frac{\sqrt{n} \diag(\bg)}{r - 1}\bigg]
    \]
    converges weakly almost surely to $\nu_{\Gauss, \frac{1}{r - 1}}$. Now, using Theorem~\ref{thm:pastur_free_conv} it follows that the ESD of
    $n^{-1/2} \bigg[ \frac{\sqrt{n(r - 2)}}{r - 1} \diag(\bV) + \frac{\sqrt{n} \, \diag(\bg)}{r - 1} + Z_n \bigg]$ convereges weakly in probability to $\nu_{\Gauss, \frac{1}{r - 1}} \boxplus \nu_{\sc}$. This proves (i) (the weak convergence of the EESD follows from the in-probability weak convergence of the ESD via the dominated convergence theorem).
    
    Now we prove (ii). First consider the matrix $\tilde{L}_{H_n}$. We invoke Lemma~\ref{lem:wass-1-bound}:
    \[
        d_{W_1}(\mubar_{n^{-1/2}\tilde{L}_{H_n}}, \mubar_{n^{-1/2} H_n})^2 \le \frac{1}{n} \bbE \bigg\|\frac{\diag(n^{-1/2}H_n \bone)}{r - 1}\bigg\|_F^2 = \frac{1}{(r - 1)^2 n^2} \bbE \sum_i \bigg(\sum_j H_{n, ij}\bigg)^2. 
    \]
    Since
    \[
        \bbE \sum_i \bigg(\sum_j H_{n, ij}\bigg)^2 = \sum_i \Var\bigg(\sum_j H_{n, ij}\bigg) = \sum_i (n + n(n - 1) \gamma_n) = O(n^2 r),
    \]
    we conclude that
    \[
        d_{W_1}(\mubar_{n^{-1/2}\tilde{L}_{H_n}}, \mubar_{n^{-1/2} H_n}) = O\bigg(\frac{1}{\sqrt{r}}\bigg).
    \]
    Since $r \to \infty$, it follows that $\mubar_{n^{-1/2} \tilde{L}_{H_n}}$ and $\mubar_{n^{-1/2} H_n}$ have the same weak limit, namely $\nu_{\sc, (1 - c)^2}$.

    Now we consider the usual Laplacian $L_{H_n}$. Again by Lemma~\ref{lem:wass-1-bound},
    \begin{align*}
        d_{W_1}(\mubar_{(nr)^{-1/2} L_{G_n}}, \mubar_{(nr)^{-1/2} \diag(G_n \bone)}) &\le \frac{1}{n} \bbE \|(nr)^{-1/2} G_n\|_F^2 \\
        &=\frac{1}{n^2 r} \sum_{i, j} \bbE [G_{n, ij}^2] \\
        &=O(r^{-1}).
    \end{align*}
    Now
    \[
        G_n \bone = n \alpha_n U \bone + n \beta_n \bar{V} \bone + n\beta_n \bV + \theta_n Z_n \bone.
    \]
    Another application of Lemma~\ref{lem:wass-1-bound} gives
    \begin{align*}
        d_{W_1}(\mubar_{(nr)^{-1/2} \diag(G_n \bone)}, \mubar_{(nr)^{-1/2} \diag(n\alpha_n U \bone + n\beta_n \bV)}) &\le \frac{1}{n} \bbE \|(nr)^{-1/2} \diag(n\beta_n \bar{V} \bone + \theta_n Z_n \bone)\|_F^2 \\
        &=\frac{1}{n^2 r} \sum_{i} \bbE [n\beta_n \bar{V} + \theta_n \sum_{j} Z_{ij}]^2 \\
        &=\frac{1}{n r} \Var\bigg(n\beta_n \bar{V} + \theta_n \sum_{j} Z_{1j}\bigg) \\
        &=\frac{1}{n r} (n \beta_n^2 + n \theta_n^2) \\
        &=O(r^{-1}).
    \end{align*}
    Finally note that
    \[
        (nr)^{-1/2} \diag(n \alpha_n U\bone + n\beta_n \bV) = \diag\bigg(\sqrt{\frac{r}{n}} U \bone + \bV\bigg) = \sqrt{\frac{r}{n}} U I + \diag(\bV).
    \]
    Since the first matrix on the right hand side above is orthogonally invariant, using Theorem~\ref{thm:pastur_free_conv} we conclude that
    \[
        \mubar_{(nr)^{-1/2} \diag(n \alpha_n U\bone + n\beta_n \bV)} \xrightarrow{d} \nu_{\Gauss, c} \boxplus \nu_{\Gauss, 1}.
    \]
    This completes the proof.
\end{proof}

\subsection{Proof of Theorem~\ref{thm:lip_estimate}}
Theorem~\ref{thm:lip_estimate} follows from Lemmas~\ref{lem:lipschitzness} and \ref{lem:lipschitz_const_estimate} below.
\begin{lemma}\label{lem:lipschitzness}
We have the following.
\begin{enumerate}
    \item [(i)] The map $\bx \mapsto n^{-1/2}H_n(\bx)$ is $\Delta_{n, r}$-Lipschitz, where
    \[
        \Delta_{n, r}^2 = \frac{1}{nN} \sum_{s = 0}^r (s^2 - s) \binom{r}{s}\binom{n - r}{r - s}.
    \]
    \item (ii) The map $\bx \mapsto (nr)^{-1/2} L_{H_n(\bx)}$ is $\Gamma_{n,r}$-Lipschitz, where
    \[
        \Gamma_{n,r}^2 = \frac{1}{n r N} \sum_{s = 0}^r ((r^2 - 2r) s + s^2) \binom{r}{s}\binom{n - r}{r - s}. 
    \]
    \item (iii) The map $\bx \mapsto n^{-1/2}\tilde{L}_{H_n(\bx)}$ is $\Xi_{n,r}$-Lipschitz, where
    \[
        \Xi_{n,r}^2 = \frac{1}{nN} \sum_{s = 0}^r s^2 \binom{r}{s}\binom{n - r}{r - s}. 
    \]
\end{enumerate}
\end{lemma}
\begin{proof}
We first prove (i). Note that
\begin{align*}
    \Tr(Q_{\ell}Q_{\ell'}) &= \Tr((\ba_{\ell}\ba_{\ell}^\top - \diag(\ba_{\ell})) (\ba_{\ell'}\ba_{\ell'}^\top - \diag(\ba_{\ell'})) \\
    &= \Tr(\ba_{\ell}\ba_{\ell}^\top \ba_{\ell'}\ba_{\ell'}^\top - \diag(\ba_{\ell})\ba_{\ell'}\ba_{\ell'}^\top - \ba_{\ell}\ba_{\ell}^\top\diag(\ba_{\ell'}) + \diag(\ba_{\ell}) \diag(\ba_{\ell'})) \\
    &= (\ba_{\ell}^\top \ba_{\ell'})^2 - \ba_{\ell}^\top \ba_{\ell'} \\
    &= |e_\ell \cap e_{\ell'}|^2 - |e_\ell \cap e_{\ell'}|.
\end{align*}
Therefore
\begin{align*}
    \|n^{-1/2}H_n(\bx) - n^{-1/2}H_n(\bx')\|_F^2 &= \frac{1}{nN} \sum_{\ell, \ell'} (\bx - \bx')_{\ell} (\bx - \bx')_{\ell'} \,\, \Tr(Q_{\ell}Q_{\ell'}) \\
    &= \frac{1}{nN} \sum_{\ell, \ell'} (\bx - \bx')_{\ell} (\bx - \bx')_{\ell'} \,\, (|e_\ell \cap e_{\ell'}|^2 - |e_\ell \cap e_{\ell'}|) \\
    &= \frac{1}{nN} \sum_{s = 0}^r (s^2 - s) \sum_{\ell, \ell' : |e_\ell \cap e_{\ell'}|= s} (\bx - \bx')_{\ell} (\bx - \bx')_{\ell'}.
\end{align*}
Now two applications of Cauchy-Schwartz gives
\begin{align}\label{eq:qform_bound} \nonumber
    \sum_{\ell, \ell' : |e_\ell \cap e_{\ell'}|= s} (\bx - \bx')_{\ell} (\bx - \bx')_{\ell'} &= \sum_{\ell = 1}^M (\bx - \bx')_{\ell} \sum_{\ell' : |e_\ell \cap e_{\ell'}|= s}(\bx - \bx')_{\ell'} \\ \nonumber
    &\le \bigg(\sum_{\ell = 1}^M (\bx - \bx')_{\ell}^2\bigg)^{1/2} \bigg(\sum_{\ell = 1}^M \bigg(\sum_{\ell' : |e_\ell \cap e_{\ell'}|= s}(\bx - \bx')_{\ell'}\bigg)^2\bigg)^{1/2} \\ \nonumber
    &\le \|\bx - \bx'\| \bigg[\sum_{\ell = 1}^M \binom{r}{s}\binom{n - r}{r - s} \sum_{\ell' : |e_\ell \cap e_{\ell'}|= s}(\bx - \bx')_{\ell'}^2\bigg]^{1/2} \\ \nonumber
    &= \|\bx - \bx'\| \bigg[\binom{r}{s}\binom{n - r}{r - s} \sum_{\ell' = 1}^M (\bx - \bx')_{\ell'}^2 \sum_{\ell : |e_\ell \cap e_{\ell'}|= s} 1 \bigg]^{1/2} \\
    &= \binom{r}{s}\binom{n - r}{r - s} \|\bx - \bx'\|^2.
\end{align}
Therefore
\[
    \|n^{-1/2}H_n(\bx) - n^{-1/2}H_n(\bx')\|_F^2 \le \bigg[\frac{1}{nN} \sum_{s = 0}^r (s^2 - s) \binom{r}{s}\binom{n - r}{r - s} \bigg] \|\bx - \bx'\|^2 = \Delta_{n, r}^2 \|\bx - \bx'\|^2.
\]
This completes the proof of (i).

Let us now prove (ii). Notice that
\begin{align*}
    L_{H_n(\bx)} &= \diag(H_n(\bx) \bone) - H_n(\bx) \\
    &= \frac{1}{\sqrt{N}} \sum_{\ell} \bx_{\ell} \, \diag( Q_{\ell} \bone) - \frac{1}{\sqrt{N}} \sum_{\ell} \bx_{\ell} Q_{\ell} \\
    &= \frac{1}{\sqrt{N}} \sum_{\ell} (r - 1) \bx_{\ell} \, \diag( \ba_{\ell}) -  \frac{1}{\sqrt{N}} \sum_{\ell} \bx_{\ell} Q_{\ell} \\
    &= \frac{1}{\sqrt{N}} \sum_{\ell} \bx_{\ell} ((r - 1)\diag( \ba_{\ell}) -  Q_{\ell}) \\
    &= \frac{1}{\sqrt{N}} \sum_{\ell} \bx_{\ell} (r \, \diag( \ba_{\ell}) - \ba_{\ell} \ba_{\ell}^\top) \\
     &= \frac{1}{\sqrt{N}} \sum_{\ell} \bx_{\ell} \hat{Q}_{\ell},
\end{align*}
where $\hat{Q}_{\ell} := r \, \diag( \ba_{\ell}) - \ba_{\ell} \ba_{\ell}^\top $.
Note that 
\begin{align*}
    \Tr(\hat{Q}_{\ell} \hat{Q}_{\ell'} ) &= \Tr((r \, \diag( \ba_{\ell}) - \ba_{\ell} \ba_{\ell}^\top) (r \, \diag( \ba_{\ell'}) - \ba_{\ell'} \ba_{\ell'}^\top )) \\
    &= \Tr(r^2 \, \diag(\ba_{\ell})\diag(\ba_{\ell'}) - r \, \diag(\ba_{\ell}) \ba_{\ell'} \ba_{\ell'}^\top - r \diag( \ba_{\ell'}) \ba_{\ell} \ba_{\ell}^\top + \ba_{\ell} \ba_{\ell}^\top\ba_{\ell'} \ba_{\ell'}^\top ) \\
    &= r^2 \, \ba_{\ell}^\top\ba_{\ell'} - r \, \ba_{\ell}^\top\ba_{\ell'} - r \,  \ba_{\ell}^\top\ba_{\ell'} + (\ba_{\ell}^\top\ba_{\ell'})^2 \\
    &= (r^2 - 2r) \ba_{\ell}^\top\ba_{\ell'} + (\ba_{\ell}^\top\ba_{\ell'})^2 \\
    &= (r^2 - 2r) |e_\ell \cap e_{\ell'}| + |e_\ell \cap e_{\ell'}|^2.
\end{align*}
Now using the bound \eqref{eq:qform_bound}, we get
\begin{align*}
    \|(nr)^{-1/2} L_{H_n(\bx)} - (nr)^{-1/2} L_{H_n(\bx')}\|_F^2 &= \frac{1}{n r N} \sum_{\ell, \ell'} (\bx - \bx')_{\ell} (\bx - \bx')_{\ell'} \,\, \Tr(\hat{Q}_{\ell}\hat{Q}_{\ell'}) \\
    &= \frac{1}{n r N} \sum_{\ell, \ell'} (\bx - \bx')_{\ell} (\bx - \bx')_{\ell'} \,\, [(r^2 - 2r) |e_\ell \cap e_{\ell'}| + |e_\ell \cap e_{\ell'}|^2 ]  \\
    &= \frac{1}{n r N} \sum_{s = 0}^r ((r^2 - 2r) s + s^2) \sum_{\ell, \ell' : |e_\ell \cap e_{\ell'}|= s} (\bx - \bx')_{\ell} (\bx - \bx')_{\ell'} \\
    & \le \bigg[\frac{1}{n r N} \sum_{s = 0}^r ((r^2 - 2r) s + s^2) \binom{r}{s}\binom{n - r}{r - s} \bigg] \|\bx - \bx'\|^2 \\
    &= \Gamma_{n,r}^2 \|\bx - \bx'\|^2. 
\end{align*}
This proves (ii). Finally, we prove (iii). Notice that
\begin{align*}
    \tilde{L}_{H_n(\bx)} &= \frac{\diag( H_n(\bx) \bone)}{r - 1} - H_n(\bx) \\
    &= \frac{1}{\sqrt{N}(r - 1)} \sum_{\ell} \bx_{\ell} \, \diag( Q_{\ell} \bone) - \frac{1}{\sqrt{N}} \sum_{\ell} \bx_{\ell} Q_{\ell} \\
    &=  \frac{1}{\sqrt{N}} \sum_{\ell} \bx_{\ell} \diag( \ba_{\ell}) -  \frac{1}{\sqrt{N}} \sum_{\ell} \bx_{\ell} Q_{\ell} \\
    &= \frac{1}{\sqrt{N}} \sum_{\ell} \bx_{\ell} ( \diag( \ba_{\ell}) -  Q_{\ell}) \\
    &= \frac{1}{\sqrt{N}} \sum_{\ell} \bx_{\ell} (2 \, \diag( \ba_{\ell}) - \ba_{\ell} \ba_{\ell}^\top) \\
     &= \frac{1}{\sqrt{N}} \sum_{\ell} \bx_{\ell} \tilde{Q}_{\ell},
\end{align*}
where $\tilde{Q}_{\ell} := 2 \, \diag( \ba_{\ell}) - \ba_{\ell} \ba_{\ell}^\top $.
Note that 
\begin{align*}
    \Tr(\tilde{Q}_{\ell} \tilde{Q}_{\ell'} ) &= \Tr((2 \, \diag( \ba_{\ell}) - \ba_{\ell} \ba_{\ell}^\top) (2 \, \diag( \ba_{\ell'}) - \ba_{\ell'} \ba_{\ell'}^\top )) \\
    &= \Tr(4 \, \diag(\ba_{\ell})\diag(\ba_{\ell'}) - 2 \, \diag(\ba_{\ell}) \ba_{\ell'} \ba_{\ell'}^\top - 2 \diag( \ba_{\ell'}) \ba_{\ell} \ba_{\ell}^\top + \ba_{\ell} \ba_{\ell}^\top\ba_{\ell'} \ba_{\ell'}^\top ) \\
    &= 4 \, \ba_{\ell}^\top\ba_{\ell'} - 2 \, \ba_{\ell}^\top\ba_{\ell'} - 2 \,  \ba_{\ell}^\top\ba_{\ell'} + ( \ba_{\ell}^\top\ba_{\ell'} )^2 \\
    &= ( \ba_{\ell}^\top\ba_{\ell'} )^2 \\
    &= |e_\ell \cap e_{\ell'}|^2.
\end{align*}
Now using the bound \eqref{eq:qform_bound}, we get
\begin{align*}
    \|n^{-1/2} \tilde{L}_{H_n(\bx)} - n^{-1/2} \tilde{L}_{H_n(\bx')}\|_F^2 &= \frac{1}{nN} \sum_{\ell, \ell'} (\bx - \bx')_{\ell} (\bx - \bx')_{\ell'} \,\, \Tr(\tilde{Q}_{\ell}\tilde{Q}_{\ell'}) \\
    &= \frac{1}{nN} \sum_{\ell, \ell'} (\bx - \bx')_{\ell} (\bx - \bx')_{\ell'} \,\, |e_\ell \cap e_{\ell'}|^2  \\
    &= \frac{1}{nN} \sum_{s = 0}^r s^2 \sum_{\ell, \ell' : |e_\ell \cap e_{\ell'}|= s} (\bx - \bx')_{\ell} (\bx - \bx')_{\ell'} \\
    & \le \bigg[\frac{1}{nN} \sum_{s = 0}^r s^2 \binom{r}{s}\binom{n - r}{r - s} \bigg] \|\bx - \bx'\|^2 \\
    &= \Xi_{n,r}^2 \|\bx - \bx'\|^2. 
\end{align*}
This proves (iii).
\end{proof}

\begin{lemma}\label{lem:lipschitz_const_estimate}
    For any $r \ge 2$, we have the following estimates.
    \begin{enumerate}
        \item [(i)] $\Delta_{n, r}^2 \le \frac{r^2}{n}$.
        \item [(ii)] $\Gamma_{n,r}^2 \le r$.
        \item [(iii)] $\Xi_{n,r}^2 \le \frac{r}{r - 1} + \frac{r^2}{n}$.
    \end{enumerate}
\end{lemma}
\begin{proof}
Let $X \sim \Hyper(n, r, r)$. Then
\begin{align*}
    \bbE[X] &= \frac{r^2}{n}, \\
    \Var(X) &= \frac{r^2}{n} \cdot \frac{n - r}{n} \cdot \frac{n - r}{n - 1} = \frac{r^2}{n} \bigg(1 - \frac{r}{n}\bigg)^2 \bigg(1 - \frac{1}{n}\bigg)^{-1} \le \frac{r^2}{n} \bigg(1 - \frac{r}{n}\bigg)^2.
\end{align*}
We note that
\begin{align*}
\frac{1}{\binom{n}{r}}\sum_{s = 0}^r (s^2 - s) \binom{r}{s}\binom{n - r}{r - s} &=\bbE(X^2 - X) \\
&= \Var(X) + (\bbE[X])^2 - \bbE[X] \\
&\le \frac{r^2}{n} \bigg(1 - \frac{r}{n}\bigg)^2 + \frac{r^2}{n} \bigg(\frac{r^2}{n} - 1\bigg) \\
&= \frac{r^2}{n} \bigg(1 - \frac{2r}{n} + \frac{r^2}{n^2} + \frac{r^2}{n} - 1\bigg) \\
&= \frac{r^3}{n^2} \bigg(\frac{r}{n} + r - 2 \bigg) \\
&\le \frac{r^3(r - 1)}{n^2}.
\end{align*}
Therefore
\[
    \Delta_{n, r}^2 \le \frac{\binom{n}{r}}{n N} \cdot \frac{r^3(r - 1)}{n^2} = \frac{n (n-1)}{nr (r - 1)} \cdot \frac{r^3(r - 1)}{n^2} \le \frac{r^2}{n}.
\]
This proves (i). For (ii), note that
\begin{align*}
    r \Xi_{n, r}^2 &= \frac{\binom{n}{r}}{nN}\sum_{s = 0}^r (r - 1)^2 s \frac{\binom{r}{s}\binom{n - r}{r - s}}{\binom{n}{r}} + \Delta_{n, r}^2 \\
    &= \frac{\binom{n}{r}}{nN} (r - 1)^2 \bbE[X] +\Delta_{n, r}^2 \\
    &= \frac{(n - 1)}{r(r - 1)} (r - 1)^2 \frac{r^2}{n} + \Delta_{n, r}^2 \\
    &\le (n - 1) \frac{r^2}{n} + \frac{r^2}{n} \\
    &= r^2.
\end{align*}
Finally, for (iii), we have
\begin{align*}
    \Gamma_{n, r}^2 &= \frac{\binom{n}{r}}{nN}\sum_{s = 0}^r s \frac{\binom{r}{s}\binom{n - r}{r - s}}{\binom{n}{r}} + \Delta_{n, r}^2 \\
    &= \frac{\binom{n}{r}}{nN} \bbE[X] +\Delta_{n, r}^2 \\
    &= \frac{(n - 1)}{r(r - 1)} \frac{r^2}{n} + \Delta_{n, r}^2 \\
    &\le \frac{r}{r - 1} + \frac{r^2}{n}.
\end{align*}
This completes the proof.
\end{proof}
\begin{remark}
    It is clear from the proof of Lemma~\ref{lem:lipschitz_const_estimate} that in fact, $\Delta_{n, r} = \Theta\big(\frac{r}{\sqrt{n}}\big)$, $\Gamma_{n, r} = \Theta(\sqrt{r})$, and $ \Xi_{n, r} = \Theta\big(\sqrt{\frac{r}{(r - 1)} + \frac{r^2}{n}}\big)$.
\end{remark}
Now we will prove Corollary~\ref{cor:conc_esd}.
Henceforth we will use the notation $\langle f, \mu \rangle := \int f \, d\mu$ for brevity. The proof of the following result is standard (see \cite{guionnet2000concentration}).
\begin{lemma}\label{lem:lip_lsi}
Suppose that all the entries of $\bY = (Y_{\ell})_{1 \le \ell \le M}$ satisfy $\LSI_{\kappa}$ for some $\kappa > 0$. There are universal constants $C_j, \hat{C}_j, \tilde{C}_j > 0$, $j = 1, 2$, such that for any $1$-Lipschitz function $f$, we have for any $t > 0$, 
\begin{enumerate}
     \item [(i)] $\bbP(|\langle f, \mu_{n^{-1/2}H_n(\bY)} \rangle - \langle f, \mubar_{n^{-1/2}H_n(\bY)} \rangle| > t) \le C_1 \exp\big(-\frac{C_2 n^2 t^2}{\kappa r^2}\big)$;
    \item [(ii)] $\bbP(|\langle f, \mu_{(nr)^{-1/2} L_{H_n(\bY)}} \rangle - \langle f, \mubar_{(nr)^{-1/2} L_{H_n(\bY)}} \rangle| > t) \le \hat{C}_1 \exp\big(-\frac{\hat{C}_2 n t^2}{\kappa r}\big)$;
    \item [(ii)] $\bbP(|\langle f, \mu_{n^{-1/2}\tilde{L}_{H_n(\bY)}} \rangle - \langle f, \mubar_{n^{-1/2}\tilde{L}_{H_n(\bY)}} \rangle| > t) \le \tilde{C}_1 \exp\big(-\frac{\tilde{C}_2 \min\{n, n^2 / r^2\} t^2}{\kappa}\big)$.
\end{enumerate}
\end{lemma}
\begin{proof}
For any $1$-Lipschitz function $f$,
\[
    |\langle f, \mu_{n^{-1/2}H_n(\bx)} \rangle - \langle f, \mu_{n^{-1/2}H_n(\bx')} \rangle| \le \frac{1}{n}\sum_{i = 1}^n |f(\lambda_i) - f(\lambda_i')|  \le \frac{1}{n}\sum_{i = 1}^n |\lambda_i - \lambda_i'|.
\]
Using the Cauchy-Schwartz and Hoffman-Wielandt inequalities, the last quantity can be bounded by
\begin{align*}
    \bigg(\frac{1}{n}\sum_{i = 1}^n (\lambda_i - \lambda_i')^2\bigg)^{1/2}&\le \bigg(\frac{1}{n} \|n^{-1/2} H_n(\bx) - n^{-1/2} H_n(\bx')\|_F^2 \bigg)^{1/2} \\
    &\le \frac{\Delta_{n, r}}{\sqrt{n}} \|\bx - \bx'\| \\
    &\le \frac{r}{n} \|\bx - \bx'\|.
\end{align*}
Thus the map $\bx \mapsto \langle f, \mu_{n^{-1/2}H_n(\bx)} \rangle$ is $\frac{r}{n}$-Lipschitz.

If all entries $(Y_{\ell})_{1 \le \ell \le M}$ satisfy $\LSI_{\kappa}$, then by the tensorization property of LSI, their joint law also satisfies $\LSI_{\kappa}$ (see, e.g., \cite{anderson2010introduction}). Then by concentration of Lipschitz functions for measures satisfying LSI, we have
\begin{align*}
    \bbP(|\langle f, \mu_{n^{-1/2}H_n(\bY)} \rangle - \langle f, \mubar_{n^{-1/2}H_n(\bY)} \rangle| > t) &= \bbP(|\langle f, \mu_{n^{-1/2}H_n(\bY)} \rangle - \bbE \langle f, \mu_{n^{-1/2}H_n(\bY)} \rangle| > t) \\
    &\le C_1 \exp\bigg(-\frac{C_2 n^2 t^2}{\kappa r^2}\bigg).
\end{align*}
This proves (i). 

The proofs of (ii) and (iii) are similar. For (iii), we just note that the Lipschitz constant of the map $\bx \mapsto \langle f, \mu_{n^{-1/2}\tilde{L}_{H_n(\bx)}} \rangle$ is
\[
    \frac{\Xi_{n, r}}{\sqrt{n}} = \sqrt{\frac{r}{n(r - 1)} + \frac{r^2}{n^2}} = \Theta \bigg(\max\bigg\{\frac{1}{\sqrt{n}}, \frac{r}{n}\bigg\}\bigg).
\]
This completes the proof.
\end{proof}
Similarly, using Talagrand's inequality, one can prove the following.
\begin{lemma}\label{lem:lip_bounded_support}
Suppose the entries of $\bY = (Y_{\ell})_{1 \le \ell \le M}$ are uniformly bounded by $K > 0$. There exist absolute constants $C_j, \hat{C}_j, \tilde{C}_j > 0$, $j = 3, 4, 5, 6$, such that with $\delta_1(n) = \frac{C_5 K}{n}$, $\hat{\delta}_1(n) = \frac{\hat{C}_5 K}{n}$, $\tilde{\delta}_1(n) = \frac{\tilde{C}_5 K}{n}$ and $Q = C_6 K$, $\hat{Q} = \hat{C}_6 K$, $\tilde{Q} = \tilde{C}_6 K$ we have for any $1$-Lipschitz function $f$ and any $t, \hat{t}, \tilde{t} > 0$ satisfying $t > (C_3 (Q + \sqrt{t}) \delta_1(n))^{2/5}$, $\hat{t} > (\hat{C}_3 (\hat{Q} + \sqrt{\hat{t}})\hat{\delta}_1(n))^{2/5}$ and $\tilde{t} > (\tilde{C}_3 (\tilde{Q} + \sqrt{\tilde{t}})\tilde{\delta}_1(n))^{2/5}$, that
\begin{enumerate}
    \item [(i)] $\bbP(|\langle f, \mu_{n^{-1/2}H_n(\bY)} \rangle - \langle f, \mubar_{n^{-1/2}H_n(\bY)} \rangle| > t) \le \frac{C_3 (Q + \sqrt{t})}{t^{3/2}}\exp\big(- \frac{C_4 n^2}{r^2} \cdot \big(\frac{t^{5/2}}{C_3 (Q + \sqrt{t})} - \delta_1(n)\big)^2\big)$;
    \item [(ii)] $\bbP(|\langle f, \mu_{(nr)^{-1/2} L_{H_n(\bY)}} \rangle - \langle f, \mubar_{(nr)^{-1/2} L_{H_n(\bY)}} \rangle| > \hat{t}) \le \frac{\hat{C}_3 (\hat{Q} + \sqrt{\hat{t}})}{\hat{t}^{3/2}}\exp\big(- \frac{\hat{C}_4 n}{r} \cdot \big(\frac{\hat{t}^{5/2}}{\hat{C}_3 (\hat{Q} + \sqrt{\hat{t}})} - \hat{\delta}_1(n)\big)^2\big)$;
    \item [(iii)] $\bbP(|\langle f, \mu_{n^{-1/2}\tilde{L}_{H_n(\bY)}} \rangle - \langle f, \mubar_{n^{-1/2}\tilde{L}_{H_n(\bY)}} \rangle| > t) \le \frac{\tilde{C}_3 (\tilde{Q} + \sqrt{\tilde{t}})}{\tilde{t}^{3/2}}\exp\big(- \tilde{C}_4 \cdot \min\{n, \frac{n^2}{r^2}\} \cdot \big(\frac{\tilde{t}^{5/2}}{\tilde{C}_3 (\tilde{Q} + \sqrt{\tilde{t}})} - \tilde{\delta}_1(n)\big)^2\big)$.
\end{enumerate}
\end{lemma}
\begin{proof}[Proof of Corollary~\ref{cor:conc_esd}]
Using a standard covering argument (see the proofs of Theorems~1.3 and 1.4 in \cite{guionnet2000concentration}), we can obtain from Lemmas~\ref{lem:lip_lsi} and \ref{lem:lip_bounded_support} the inequalities given in (i)-(vi). We skip the details.
\end{proof}

\subsection{Proofs of Theorems~\ref{thm:extreme_eigen} and~\ref{thm:extreme_eigen_laplacian}}

\begin{proof}[Proof of Theorem~\ref{thm:extreme_eigen}]
    Define $G_n$ and $G'_n$ as in \eqref{eq:GHAM_decomp} and \eqref{eq:GHAM_decomp_w/o_diag}, respectively and recall that $H_n$ and $G'_n$ have the same distribution. Also, $n^{-1/2} Z_n$ is a GOE random matrix.
    Note that by Weyl's inequality,
    \[
        \max_{1 \le i \le n}|\lambda_i(G_n) - \lambda_i(G'_n)| \le \|\diag(G_n)\|_{\op} =  \max_i |G_{n,ii}| = O_P(\sqrt{\log n}).
    \]
     Write $P_n = \alpha_n U \bone\bone^\top + \beta_n (\bone \bV^\top + \bV \bone^\top)$. Then
    \begin{equation}\label{eq:closensess_of_G_and_P}
        \max_{1 \le i \le n}|\lambda_i(G_n) - \lambda_i(P_n)| \leq \theta_n \|Z_n\|_{\op} = O_P(\sqrt n).
    \end{equation}
    Thus, it is enough to prove \eqref{eq:gham_extreme_1} by replacing $H_n$ with $P_n$ and \eqref{eq:gham_extreme_3} by replacing $H_n$ with $G_n$. 
    
    Notice that $\bone$, $\bV$ are almost surely linearly independent. Consider an ordered basis $\mathcal{B} = \{\bone, \bV, \bw_3, \bw_4, \ldots, \bw_n\}$ of $\bbR^n$, where $\{\bw_3, \bw_4, \ldots, \bw_n\}$ is a linearly independent set which is orthogonal to both $\bone$ and $\bV$. In the basis $\mathcal B$, $P_n$ has the following matrix representation:
    \[
        \begin{pmatrix}
            n\alpha_n U + n\beta_n \bar V  & n\alpha_n U\bar V  +\beta_n \|\bV\|^2  & 0  &0  &\cdots &0\\
            n\beta_n  & n\beta_n \bar V  &0  &0  &\cdots  &0\\
            0  &0  &0  &0  &\cdots  &0\\
            \vdots  &\vdots  &\vdots  &\vdots  &  &\vdots\\
            0  &0  &0  &0  &\cdots  &0
        \end{pmatrix}.
    \]
    From this representation, one can easily compute the eigenvalues of $P_n$. Let $s^2 := \frac{1}{n} \sum_{i=1}^n (V_i - \bar{V})^2$. Then
    \begin{align}
        \frac{\lambda_1(P_n)}{n} &= \frac{\alpha_n}{2} U  + \beta_n \bar V + \sqrt{\bigg(\frac{\alpha_n}{2}U + \beta_n \bar V\bigg)^2 + \beta_n^2 s^2}, \label{eq:eigen_lambda_1}\\
        \frac{\lambda_n(P_n)}{n} &= \frac{\alpha_n}{2}U + \beta_n \bar V -\sqrt{\bigg(\frac{\alpha_n}{2}U + \beta_n \bar V\bigg)^2 + \beta_n^2 s^2}, \label{eq:eigen_lambda_n}\\
        \lambda_k(P_n) &= 0 \quad \text{for } k = 2, 3, \ldots, n - 1.
    \end{align}
    Note that $\bar{V} \convas 0$ and $s^2 \convas 1$. Therefore it is natural to compare $\lambda_1(P_n)$ to $\frac{\alpha_n}{2} U + \sqrt{\frac{\alpha_n^2}{4}U^2 + \beta_n^2}$. To that end, note that
    \begin{align*}
        \sqrt n\bigg|&\frac{\lambda_1(P_n)}{n} - \frac{c_n}{2}U - \sqrt{\frac{c_n^2}{4}U^2 + c_n(1-c_n)}\bigg| \\
        &= \sqrt n \bigg|\frac{1}{2}(\alpha_n - c_n)U +\beta_n \bar V + \frac{(\frac{\alpha_n}{2}U +\beta_n \bar V)^2 +\beta_n^2 s^2 -\frac{c_n^2}{4}U^2 - c_n(1-c_n)}{\sqrt{(\frac{\alpha_n}{2}U +\beta_n \bar V)^2 + \beta_n^2 s^2} + \sqrt{\frac{c_n^2}{4}U^2 + c_n(1-c_n)}} \bigg|\\
        &\le \frac{1}{2}\sqrt{n}|\alpha_n - c_n| \cdot |U| + \beta_n \sqrt{n} |\bar V| + \frac{|\alpha_n U + c_n U + 2\beta_n \bar V| \cdot (\sqrt{n}|\alpha_n - c_n| |U| + 2\beta_n \sqrt{n} |\bar{V}|)}{4\sqrt{c_n (1-c_n)}}\\
        &\qquad\qquad\qquad + \frac{ \sqrt n|\beta_n^2 s^2 -c_n(1-c_n)|}{ \sqrt{c_n(1-c_n)}}.
    \end{align*}
    Observe that $ \sqrt n|\alpha_n - c_n| = O(1/\sqrt n)$. Also, since $\sqrt{n} \bar{V} \sim N(0, 1)$,
    \[
        \beta_n \sqrt n \bar V = O_P(\sqrt c_n).
    \]
    Further, since $\sqrt{n}(s^2 - 1) \convd N(0, 2)$,
    \begin{align*}
        \sqrt n |\beta_n^2 s^2 -c_n(1-c_n)| 
        &\leq \sqrt n s^2|\beta_n^2 - c_n(1-c_n)| + c_n (1-c_n) \sqrt n|s^2 -1| \\
        &= O_P(1/\sqrt n) + O_P(c_n), 
    \end{align*}
    Finally, notice that $\frac{1}{\sqrt{c_n(1-c_n)}} = O(1/\sqrt c_n)$. Combining these we get that
    \[
        \sqrt n\bigg(\frac{\lambda_1(P_n)}{n} - \frac{c_n}{2}U - \sqrt{\frac{c_n^2}{4}U^2 + c_n(1-c_n)}\bigg) = O_P(1).
    \]
    From these the desired result follows for $\lambda_1(P_n)$ and hence for $\lambda_1(H_n)$. Similarly, one can tackle $\lambda_n(H_n)$. This completes the proofs of \eqref{eq:gham_extreme_1} and \eqref{eq:gham_extreme_2}.

    Now we prove (ii). Suppose that $c_n \to 0$ and $r \to \infty$. From \eqref{eq:closensess_of_G_and_P} it follows that
    \[
        \sup_{1 \le i \le n}\bigg|\frac{\lambda_i(G_n)}{\sqrt{nr}} - \frac{\lambda_i(P_n)}{\sqrt{nr}}\bigg| = O_P\bigg(\frac{1}{\sqrt{r}}\bigg).
    \]
    Further note that
    \begin{align*}
        \frac{\lambda_1(P_n)}{\sqrt{nr}} &= \frac{1}{\sqrt{c_n}} \cdot \frac{\lambda_1(P_n)}{n} = \frac{1}{\sqrt{c_n}} \bigg[\frac{c_n}{2}U + \sqrt{\frac{c_n^2}{4}U^2 + c_n(1-c_n)} +  O_P\bigg(\frac{1}{\sqrt{n}}\bigg)\bigg] \\
        &= \frac{\sqrt{c_n}}{2}U + \sqrt{\frac{c_n}{4}U^2 + (1-c_n)} + O_P\bigg(\frac{1}{\sqrt{r}}\bigg),
    \end{align*}
    from which it follows that
    \[
        \frac{\lambda_1(P_n)}{\sqrt{nr}} \convp 1,
    \]
    and similarly,
    \[
        \frac{\lambda_n(P_n)}{\sqrt{nr}} \convp -1.
    \]
    
    Now we consider (iii), the regime where $r$ is fixed. In this regime, we have to consider $G_n$ itself. We think of $G_n$ as a low rank deformation of a scaled GOE matrix:
    \[
        G_n = P_n + \theta_n Z_n.
    \]
    Since $r$ is fixed, it is clear that
    \begin{align*}
        n\alpha_n = O(1), \quad \sqrt{n}\beta_n \to \sqrt{r-2} \quad\text{and} \quad \theta_n \to 1.
    \end{align*}
    Now, multiplying both sides of \eqref{eq:eigen_lambda_1} and \eqref{eq:eigen_lambda_n} by $\sqrt n$ and using the facts that $\bar{\bV} \convas 0 $ and $s^2 \convas 1$, we get that
    \begin{align}\label{eq:lambda_1_and_n_convergence}
        \frac{\lambda_1(P_n)}{\sqrt n} \convas \sqrt{r-2} \quad \text{and} \quad \frac{\lambda_n(P_n)}{\sqrt n} \convas -\sqrt{r-2}.
    \end{align}
    Suppose $\bu_1$ and $\bu_n$ are a orthonormal pair of eigenvectors corresponding to $\lambda_1$ and $\lambda_n$, respectively. Then $ P_n$ has the representation
    \[
        P_n = \lambda_1(P_n) \bu_1 \bu_1^\top + \lambda_n(P_n) \bu_n \bu_n^\top.
    \]
    Define $\tilde P_n = \sqrt{n(r-2)}(\bu_1 \bu_1^\top - \bu_n \bu_n^\top)$. Now, by virtue of \eqref{eq:lambda_1_and_n_convergence}, we may conclude that $\frac{1}{\sqrt n}\|P_n - \tilde P_n\|_\op \convas 0$. Therefore by Weyl's inequality, it is enough to consider $\frac{1}{\sqrt{n}}\tilde G_n$ instead of $\frac{1}{\sqrt n} G_n$, where $\tilde G_n$ is defined as 
    \[
        \tilde G_n = \tilde P_n + \theta_n Z_n.
    \]
    Now note that $\tilde{P}_n$ and $Z_n$ are independent and $Z_n$ is orthogonally invariant. Further, the LSD of $\theta_n Z_n$ is the standard semi-circle law. Hence one may apply Theorem~2.1 of \cite{benaych2011eigenvalues} on $\tilde G_n$ to conclude that
    \begin{align*}
        \frac{\lambda_1(\tilde{G_n})}{\sqrt n} &\convas 
        \begin{cases}
            \sqrt{r-2} + \frac{1}{\sqrt {r-2}} & \text{ if } \sqrt{r-2} > 1, \text{ i.e. if } r > 3, \\
            2 & \text{ if } \sqrt{r-2} \leq 1, \text{ i.e. if } r \le 3.
        \end{cases}
    \end{align*}
    Similarly,
    \begin{align*}
        \frac{\lambda_n(\tilde G_n)}{\sqrt n}&\convas 
        \begin{cases}
            - \sqrt{r-2} - \frac{1}{\sqrt {r-2}} & \text{ if } \sqrt{r-2} > 1, \text{ i.e. if } r > 3, \\
            -2 & \text{ if } \sqrt{r-2} \leq 1, \text{ i.e. if } r \le 3.
        \end{cases}
    \end{align*}
    This gives us the desired result.
    
    To prove (iv), notice that for $k \geq 1$, by Lemma~\ref{lem:weyl_inequality}, for any $i\leq k+1, i' \geq k+1$,
    \[
        \lambda_{i'}(\theta_n Z_n) + \lambda_{n+k+1-i'}(P_n) \leq \lambda_{k+1}(G_n) \leq \lambda_i(\theta_n Z_n) + \lambda_{k+2-i} (P_n).
    \]
    Taking $i=k$ and $i' = k+2$ and noticing that $\lambda_{2}(P_n) = \lambda_{n-1}(P_n) =0$, we get
    \[
        \theta_n \lambda_{k+2}(Z_n) \leq \lambda_{k+1} (G_n) \leq \theta_n \lambda_k (Z_n).
    \]
    As a consequence of Tracy and Widom's  seminal result on the fluctuations of extreme eigenvalues of the GOE (see \cite{tracy1996extremeeigen}), $n^{2/3}(\lambda_k(Z_n) -2)$ and $n^{2/3} (\lambda_{k+2}(Z_n) - 2)$ are $O_P(1)$, implying that $n^{2/3} (\lambda_{k+1}(G_n) - 2 \theta_n) = O_P(1)$. Therefore, 
    \[
        n^{2/3} (\lambda_{k+1}(G_n) - 2(1-c)) = O_P(1).
    \]
    This proves \eqref{eq:gham_extreme_3}. The proof of \eqref{eq:gham_extreme_4} is similar to that of \eqref{eq:gham_extreme_3}, thus skipped.
\end{proof}

\begin{proof}[Proof of Theorem~\ref{thm:extreme_eigen_laplacian}]   
    Here, we prove \eqref{eq:gham_extreme_5}, \eqref{eq:gham_extreme_7}, \eqref{eq:gham_extreme_9},  \eqref{eq:gham_extreme_13}, \eqref{eq:gham_extreme_17}. Then, \eqref{eq:gham_extreme_6}, \eqref{eq:gham_extreme_8}, \eqref{eq:gham_extreme_10},  \eqref{eq:gham_extreme_14},  \eqref{eq:gham_extreme_18} can be proved following the same line of argument with minor modifications, so we skip them. 

    First, we prove part (A). Since $L_{G_n} = L_{G'_n}$, it suffices to consider $L_{G_n}$. Since $B \mapsto L_B$ is a linear map,
    \begin{align*}
          L_{G_n} = L_{P_n} + \theta_n L_{Z_n} = D_{P_n} - P_n + \theta_n L_{Z_n}. 
    \end{align*}
    Using Theorem 1.2 of \cite{campbell2024extreme}, we have $\|L_{Z_n}\|_{\op} = O_P(\sqrt{n \log n})$. Note that 
    \[
        D_{P_n} = n \alpha_n U I + n \beta_n \diag( \bV) + n \beta_n \bar V I.
    \]
    Now, 
    \begin{align*}
        \| n \alpha_n U I\|_{\op} = O_P(n), \quad \|n \bar V I \|_{\op} = O_P(\sqrt n), \quad \text{ and } \quad \|P_n\|_{\op} = O_P(n).
    \end{align*}
    Combining these, we have by Weyl's inequality, 
    \[
        |\lambda_k(L_{G_n}) - n \beta_n \lambda_k(\diag(\bV))| \le \|L_{G_n} - n \beta_n \diag (\bV)\|_{\op} = O_P(n).
    \]
     Notice that $\lambda_k(\diag(\bV))$ is nothing but the $k$-th largest order statistic of $n$ i.i.d. standard Gaussians. Thus, using Lemma~\ref{lem:order_stat_gaussian}, we get
     \[
        \log n\bigg(\frac{\lambda_k(\diag(\bV))}{\sqrt {2 \log n}} - 1 +\frac{\log \log n}{4 \log n}\bigg) = O_P(1).
     \]
     Taking into account the error terms, we can conclude that
    \begin{align*}
        \sqrt{\log n}&\bigg( \frac{\lambda_k(L_{G_n})}{n\sqrt {2  \log n}} -\beta_n \bigg)\\ &=\frac{\lambda_k(L_{G_n})-n \beta_n\lambda_k(\diag(\bV))}{\sqrt 2 n} + \beta_n \sqrt{\log n}\bigg(\frac{\lambda_k(\diag(\bV))}{\sqrt{2 \log n}}-1-\frac{\log{\log n}}{4 \log n}\bigg) + \frac{\beta_n \log{\log n}}{4 \sqrt{\log n}} \\
        &=O_P(1).
    \end{align*}
    which yields \eqref{eq:gham_extreme_5}.\\

    Now we prove parts (B) and (C). Notice that $\|\tilde L_{G'_n} - \tilde L_{G'_n}\|_\mathrm{op} = \frac{1}{r} \max_{i} |G_{n,ii}| = O_P\big(\frac{\sqrt{\log n}}{r}\big)$. We thus see that \eqref{eq:gham_extreme_7}, \eqref{eq:gham_extreme_9},  \eqref{eq:gham_extreme_13}, \eqref{eq:gham_extreme_17} hold for $\tilde L_{G_n}$ if and only if they hold for $\tilde L_{G'_n}$. Thus, it is enough to work with $\tilde{L}_{G_n}$.

    Note that $\tilde L_{G_n}$ can be expressed as
    \[
        \tilde L_{G_n} = \frac{1}{r-1}D_{P_n} - P_n +\frac{\theta_n}{r-1} D_{Z_n} -\theta_n Z_n.
    \]
    Using Lemma~\ref{lem:order_stat_gaussian}, we obtain the following estimates:
    \begin{align*}
        \|D_{P_n}\|_{\op} &\leq n \alpha_n |U|  + n \beta_n \|\diag( \bV)\|_\op + n \beta_n \|\bar V\|_\op  = O_P(n\sqrt{ \log n}), \\
        \|D_{Z_n}\|_{\op} &= \sqrt{n}\max_i \bigg|\frac{1}{\sqrt n}\sum_j Z_{n,ij}\bigg| =O_P(\sqrt{n \log n}). 
    \end{align*}
    Therefore
    \begin{align*}
        |\lambda_1(\tilde L_{G_n}) - \lambda_1(P_n)| 
        &\leq \frac{1}{r-1}\|D_{P_n}\|_{\op} +\frac{\theta_n}{r-1}\|D_{Z_n}\|_{\op} +\theta_n \|Z_n\|_{\op}\\
        &= O_P\bigg(\frac{n\sqrt{ \log n}}{r}\bigg) +  O_P\bigg(\frac{\sqrt{n \log n}}{r}\bigg) + O_P(\sqrt n)\\
        &= O_P\bigg(\frac{n\sqrt{ \log n}}{r}\bigg) + O_P(\sqrt n). 
    \end{align*}
    Using the proof of \eqref{eq:gham_extreme_1},
    \begin{align*}
        \bigg|\frac{\lambda_1(\tilde L_{G_n})}{n} - &\frac{c_n}{2}U -\sqrt{\frac{c_n^2}{4}U +c_n(1-c_n)}\bigg| \\ 
        &\leq \frac{1}{n} |\lambda_1(\tilde L_{G_n}) - \lambda_1(P_n)| + \bigg|\frac{\lambda_1(P_n)}{n} -\frac{c_n}{2}U -\sqrt{\frac{c_n^2}{4}U +c_n(1-c_n)}\bigg|\\
        &=O_P\bigg(\frac{\sqrt {\log n}}{r}\bigg) + O_P\bigg(\frac{1}{\sqrt n} \bigg).
    \end{align*}
    This proves \eqref{eq:gham_extreme_9}. In order to prove \eqref{eq:gham_extreme_7}, notice that
    \[
        \bigg|\lambda_1(\tilde L_{G_n}) - \frac{1}{r-1}\lambda_1(D_{P_n})\bigg| \leq \|P_n\|_{\op} +\frac{\theta_n}{r-1}\|D_{Z_n}\|_{\op} +\theta_n \|Z_n\|_{\op} = O_P(n),
    \]
    which, together with the proof of \eqref{eq:gham_extreme_5}, implies that
    \begin{align*}
        \bigg|\frac{(r-1) \lambda_1(\tilde L_{G_n})}{n\sqrt{2 \log n}} &-\sqrt{c_n(1-c_n)}\bigg| \\ 
        &\leq \bigg|\frac{(r-1) \lambda_1(\tilde L_{G_n})- \lambda_1(D_{P_n})}{n\sqrt{2 \log n}} \bigg| + \bigg|\frac{\lambda_1(D_{P_n})}{n\sqrt{2 \log n}} - \beta_n\bigg| + |\beta_n - \sqrt{c_n(1-c_n)}\bigg|\\
        &= \bigg|\frac{-(r-1)P_n + \theta_n D_{Z_n} - (r-1)\theta_n Z_n }{n \sqrt{2 \log n}}\bigg|+O_P\bigg(\frac{1}{\sqrt {\log n}}\bigg) + O_P\bigg(\frac{1}{n}\bigg)\\
        &= O_P\bigg(\frac{r}{\sqrt{\log n}}\bigg) + O_P\bigg(\frac{1}{\sqrt n}\bigg) + O_P\bigg(\frac{r}{\sqrt{n \log n}}\bigg) + O_P\bigg(\frac{1}{\sqrt{\log n}}\bigg) + O_P\bigg(\frac{1}{n}\bigg)\\
        &= O_P\bigg(\frac{r}{\sqrt{\log n}}\bigg).
    \end{align*}
    This completes the proof of (B).
    
    For \eqref{eq:gham_extreme_13} and \eqref{eq:gham_extreme_17}, using Weyl's inequality as in the proof of \eqref{eq:gham_extreme_3}, it is enough to consider the matrix 
    \[
        \tilde L'_{G_n} := \tilde L_{G_n} - P_n = \frac{1}{r-1} D_{P_n} +\frac{\theta_n}{r-1} D_{Z_n} + \theta_n Z_n.
    \]
    Note that
    \begin{align*}
        \|\tilde L'_{G_n} - \frac{1}{r-1} D_{P_n}\|_{\op} &= O_P\bigg(\frac{\sqrt {n \log n}}{r}\bigg) +O_P(\sqrt n), \\
        \|\tilde L'_{G_n} - \theta_n Z_n\|_{\op} &= O_P\bigg(\frac{n \sqrt{\log n}}{r}\bigg).
    \end{align*}
    Therefore if $r \ll \sqrt{n \log n}$, then
    \begin{align*}
        \bigg|\frac{(r-1) \lambda_k(\tilde L'_{G_n})}{n\sqrt{2 \log n}} -\sqrt{c_n(1-c_n)}\bigg| 
        &\leq \bigg|\frac{(r-1) \lambda_k(\tilde L'_{G_n})- \lambda_k(D_{P_n})}{n\sqrt{2 \log n}} \bigg| + \bigg|\frac{\lambda_k(D_{P_n})}{n\sqrt{2 \log n}} - c_n(1-c_n)\bigg|\\
        &= O_P\bigg(\frac{1}{\sqrt {\log n}}\bigg) + O_P\bigg(\frac{r}{\sqrt{n \log n}}\bigg),
    \end{align*}
    which proves \eqref{eq:gham_extreme_13}.
    On the other hand, if $r \gg \sqrt{n \log n}$,
    \begin{align*}
        \bigg|\frac{\lambda_k(\tilde L'_{G_n})}{\sqrt n} -2 (1-c)\bigg| 
        &\leq  \frac{1}{\sqrt n}\|\tilde L'_{G_n} - \theta_n Z_n\|_{\op} + \bigg| \theta_n \lambda_k\bigg(\frac{1}{\sqrt n}Z_n\bigg) - 2(1-c) \bigg|\\
        &= O_P\bigg(\frac{\sqrt{n \log n}}{r}\bigg) + O_P\bigg(\frac{1}{n^{2/3}}\bigg)\\
        &= O_P\bigg(\frac{\sqrt{n \log n}}{r}\bigg)
    \end{align*}
    proving \eqref{eq:gham_extreme_17}. This completes the proof of the theorem.
\end{proof}

\section*{Acknowledgements}
We thank Arijit Chakrabarty for helpful discussions. SSM was partially supported by the INSPIRE research grant DST/INSPIRE/04/2018/002193 from the Dept.~of Science and Technology, Govt.~of India, 
and a Start-Up Grant from Indian Statistical Institute.

\bibliographystyle{alpha}
\bibliography{refs.bib}

\newcommand{\etalchar}[1]{$^{#1}$}
\begin{thebibliography}{dMGCC22}

\bibitem[AGV23]{au2023spectral}
Benson Au and Jorge Garza-Vargas.
\newblock Spectral asymptotics for contracted tensor ensembles.
\newblock {\em Electronic Journal of Probability}, 28:1--32, 2023.

\bibitem[AGZ10]{anderson2010introduction}
Greg~W Anderson, Alice Guionnet, and Ofer Zeitouni.
\newblock {\em An introduction to random matrices}.
\newblock Cambridge university press, 2010.

\bibitem[Ban21]{banerjee2021spectrum}
Anirban Banerjee.
\newblock On the spectrum of hypergraphs.
\newblock {\em Linear algebra and its applications}, 614:82--110, 2021.

\bibitem[BCM17]{banerjee2017spectra}
Anirban Banerjee, Arnab Char, and Bibhash Mondal.
\newblock Spectra of general hypergraphs.
\newblock {\em Linear Algebra and its Applications}, 518:14--30, 2017.

\bibitem[BDJ05]{bryc2005spectral}
Włodzimierz Bryc, Amir Dembo, and Tiefeng Jiang.
\newblock {Spectral measure of large random Hankel, Markov and Toeplitz matrices}.
\newblock {\em The Annals of Probability}, 33(0):0 -- 38, 2005.

\bibitem[BG99]{bobkov1999exponential}
S.G Bobkov and F~Götze.
\newblock Exponential integrability and transportation cost related to logarithmic sobolev inequalities.
\newblock {\em Journal of Functional Analysis}, 163(1):1--28, 1999.

\bibitem[BGN11]{benaych2011eigenvalues}
Florent Benaych-Georges and Raj~Rao Nadakuditi.
\newblock The eigenvalues and eigenvectors of finite, low rank perturbations of large random matrices.
\newblock {\em Advances in Mathematics}, 227(1):494--521, 2011.

\bibitem[BGT10]{bobkov2010concentration}
Sergey~G Bobkov, Friedrich G{\"o}tze, and Alexander~N Tikhomirov.
\newblock On concentration of empirical measures and convergence to the semi-circle law.
\newblock {\em Journal of Theoretical Probability}, 23(3):792--823, 2010.

\bibitem[Bia97]{biane1997free}
Philippe Biane.
\newblock On the free convolution with a semi-circular distribution.
\newblock {\em Indiana University Mathematics Journal}, pages 705--718, 1997.

\bibitem[Bol93]{bollalaplacian1993}
Marianna Bolla.
\newblock Spectra, {E}uclidean representations and clusterings of hypergraphs.
\newblock {\em Discrete Math.}, 117(1-3):19--39, 1993.

\bibitem[Bon24]{bonnin2024universality}
Remi Bonnin.
\newblock Universality of the wigner-gurau limit for random tensors.
\newblock {\em arXiv preprint arXiv:2404.14144}, 2024.

\bibitem[BP23]{Bangenoperators2023}
Anirban Banerjee and Samiron Parui.
\newblock On some general operators of hypergraphs.
\newblock {\em Linear Algebra Appl.}, 667:97--132, 2023.

\bibitem[BTC{\etalchar{+}}10]{bu2010music}
Jiajun Bu, Shulong Tan, Chun Chen, Can Wang, Hao Wu, Lijun Zhang, and Xiaofei He.
\newblock Music recommendation by unified hypergraph: combining social media information and music content.
\newblock In {\em Proceedings of the 18th ACM international conference on Multimedia}, pages 391--400, 2010.

\bibitem[BvH24]{brailovskaya2024universality}
Tatiana Brailovskaya and Ramon van Handel.
\newblock Universality and sharp matrix concentration inequalities, 2024.

\bibitem[CD12]{cooper2012spectra}
Joshua Cooper and Aaron Dutle.
\newblock Spectra of uniform hypergraphs.
\newblock {\em Linear Algebra and its applications}, 436(9):3268--3292, 2012.

\bibitem[CH22]{LapgenerealWigner2022}
Anirban Chatterjee and Rajat~Subhra Hazra.
\newblock Spectral properties for the {L}aplacian of a generalized {W}igner matrix.
\newblock {\em Random Matrices Theory Appl.}, 11(3):Paper No. 2250026, 66, 2022.

\bibitem[Cha99]{chang1999laplacian}
A~Chang.
\newblock On the laplacian of a hypergraph.
\newblock {\em Mathematica applicata}, 12:93--97, 1999.

\bibitem[Cha05]{chatterjee2005simple}
Sourav Chatterjee.
\newblock A simple invariance theorem.
\newblock {\em arXiv preprint math/0508213}, 2005.

\bibitem[Cha06]{chatterjee2006generalization}
Sourav Chatterjee.
\newblock {A generalization of the Lindeberg principle}.
\newblock {\em The Annals of Probability}, 34(6):2061 -- 2076, 2006.

\bibitem[Chi16]{chimarkovmatrix2016}
Zhiyi Chi.
\newblock Random reversible {M}arkov matrices with tunable extremal eigenvalues.
\newblock {\em Ann. Appl. Probab.}, 26(4):2257--2272, 2016.

\bibitem[CHS13]{chakrabarty2013limiting}
Arijit Chakrabarty, Rajat~Subhra Hazra, and Deepayan Sarkar.
\newblock Limiting spectral distribution for wigner matrices with dependent entries.
\newblock {\em arXiv preprint arXiv:1304.3394}, 2013.

\bibitem[CHS16]{chakrabarty2016random}
Arijit Chakrabarty, Rajat~Subhra Hazra, and Deepayan Sarkar.
\newblock From random matrices to long range dependence.
\newblock {\em Random Matrices: Theory and Applications}, 5(02):1650008, 2016.

\bibitem[Chu93]{chung1993laplacian}
Fan Chung.
\newblock The laplacian of a hypergraph.
\newblock {\em Expanding Graphs}, pages 21--36, 1993.

\bibitem[CLO{\etalchar{+}}24]{campbell2024extreme}
Andrew Campbell, Kyle Luh, Sean O'rourke, Santiago Arenas-Velilla, and Victor Perez-Abreu.
\newblock Extreme eigenvalues of laplacian random matrices with gaussian entries.
\newblock {\em arXiv preprint arXiv:2211.17175}, 2024.

\bibitem[Coo20]{cooper2020adjacency}
Joshua Cooper.
\newblock Adjacency spectra of random and complete hypergraphs.
\newblock {\em Linear Algebra and its Applications}, 596:184--202, 2020.

\bibitem[DJ10]{ding2010spectral}
Xue Ding and Tiefeng Jiang.
\newblock Spectral distributions of adjacency and laplacian matrices of random graphs.
\newblock {\em The annals of applied probability}, pages 2086--2117, 2010.

\bibitem[dMGCC22]{goulart2022random}
Jos\'{e}~Henrique de~M.~Goulart, Romain Couillet, and Pierre Comon.
\newblock A random matrix perspective on random tensors.
\newblock {\em Journal of Machine Learning Research}, 23(264):1--36, 2022.

\bibitem[Dud18]{dudley2018real}
Richard~M Dudley.
\newblock {\em Real analysis and probability}.
\newblock Chapman and Hall/CRC, 2018.

\bibitem[DW23]{dumitriu2023exact}
Ioana Dumitriu and Haixiao Wang.
\newblock Exact recovery for the non-uniform hypergraph stochastic block model.
\newblock {\em arXiv preprint arXiv:2304.13139}, 2023.

\bibitem[DWZ21]{dumitriu2021partial}
Ioana Dumitriu, Haixiao Wang, and Yizhe Zhu.
\newblock Partial recovery and weak consistency in the non-uniform hypergraph stochastic block model.
\newblock {\em arXiv preprint arXiv:2112.11671}, 2021.

\bibitem[EKS19]{erdHos2019random}
L{\'a}szl{\'o} Erd{\H{o}}s, Torben Kr{\"u}ger, and Dominik Schr{\"o}der.
\newblock Random matrices with slow correlation decay.
\newblock In {\em Forum of Mathematics, Sigma}, volume~7, page~e8. Cambridge University Press, 2019.

\bibitem[ES14]{erdHos2014phase}
L{\'a}szl{\'o} Erd{\H{o}}s and Dominik Schr{\"o}der.
\newblock Phase transition in the density of states of quantum spin glasses.
\newblock {\em Mathematical Physics, Analysis and Geometry}, 17(3):441--464, 2014.

\bibitem[F{\etalchar{+}}96]{feng1996spectra}
Keqin Feng et~al.
\newblock Spectra of hypergraphs and applications.
\newblock {\em Journal of number theory}, 60(1):1--22, 1996.

\bibitem[Fri91]{friedman1991spectra}
Joel Friedman.
\newblock The spectra of infinite hypertrees.
\newblock {\em SIAM Journal on Computing}, 20(5):951--961, 1991.

\bibitem[FSS15]{flamm2015generalized}
Christoph Flamm, B{\"a}rbel~MR Stadler, and Peter~F Stadler.
\newblock Generalized topologies: hypergraphs, chemical reactions, and biological evolution.
\newblock In {\em Advances in Mathematical Chemistry and Applications}, pages 300--328. Elsevier, 2015.

\bibitem[FTW19]{feng2019spectrum}
Renjie Feng, Gang Tian, and Dongyi Wei.
\newblock Spectrum of {SYK} model.
\newblock {\em Peking Mathematical Journal}, 2:41--70, 2019.

\bibitem[FW95]{friedman1995second}
Joel Friedman and Avi Wigderson.
\newblock On the second eigenvalue of hypergraphs.
\newblock {\em Combinatorica}, 15(1):43--65, 1995.

\bibitem[GD14]{ghoshdastidar2014consistency}
Debarghya Ghoshdastidar and Ambedkar Dukkipati.
\newblock Consistency of spectral partitioning of uniform hypergraphs under planted partition model.
\newblock {\em Advances in Neural Information Processing Systems}, 27, 2014.

\bibitem[GD17]{ghoshdastidar2017consistency}
Debarghya Ghoshdastidar and Ambedkar Dukkipati.
\newblock Consistency of spectral hypergraph partitioning under planted partition model.
\newblock {\em The Annals of Statistics}, 45(01):289--315, 2017.

\bibitem[GNT15]{gotze2015limit}
F~G\"{o}tze, AA~Naumov, and AN~Tikhomirov.
\newblock Limit theorems for two classes of random matrices with dependent entries.
\newblock {\em Theory of Probability \& Its Applications}, 59(1):23--39, 2015.

\bibitem[GNTT18]{gotze2018local}
Friedrich G{\"o}tze, Alexey Naumov, Alexander Tikhomirov, and Dmitry Timushev.
\newblock {On the local semicircular law for Wigner ensembles}.
\newblock {\em Bernoulli}, 24(3):2358 -- 2400, 2018.

\bibitem[Gov05]{govindu2005tensor}
Venu~Madhav Govindu.
\newblock A tensor decomposition for geometric grouping and segmentation.
\newblock In {\em 2005 IEEE Computer Society Conference on Computer Vision and Pattern Recognition (CVPR'05)}, volume~1, pages 1150--1157. IEEE, 2005.

\bibitem[Gro75]{gross1975logarithmic}
Leonard Gross.
\newblock Logarithmic sobolev inequalities.
\newblock {\em American Journal of Mathematics}, 97(4):1061--1083, 1975.

\bibitem[GS02]{gibbs2002choosing}
Alison~L Gibbs and Francis~Edward Su.
\newblock On choosing and bounding probability metrics.
\newblock {\em International statistical review}, 70(3):419--435, 2002.

\bibitem[Gur20]{gurau2020generalization}
Razvan Gurau.
\newblock On the generalization of the wigner semicircle law to real symmetric tensors.
\newblock {\em arXiv: Mathematical Physics}, 2020.

\bibitem[GZ00]{guionnet2000concentration}
Alice Guionnet and Ofer Zeitouni.
\newblock {Concentration of the Spectral Measure for Large Matrices}.
\newblock {\em Electronic Communications in Probability}, 5(none):119 -- 136, 2000.

\bibitem[HQ12]{hu2012algebraic}
Shenglong Hu and Liqun Qi.
\newblock Algebraic connectivity of an even uniform hypergraph.
\newblock {\em Journal of Combinatorial Optimization}, 24(4):564--579, 2012.

\bibitem[HQ15]{hu2015laplacian}
Shenglong Hu and Liqun Qi.
\newblock The laplacian of a uniform hypergraph.
\newblock {\em Journal of Combinatorial Optimization}, 29(2):331--366, 2015.

\bibitem[Jia12]{sparselap2012}
Tiefeng Jiang.
\newblock Empirical distributions of {L}aplacian matrices of large dilute random graphs.
\newblock {\em Random Matrices Theory Appl.}, 1(3):1250004, 20, 2012.

\bibitem[JJ16]{ji2016coauthorship}
Pengsheng Ji and Jiashun Jin.
\newblock Coauthorship and citation networks for statisticians.
\newblock {\em Annals of Applied Statistics}, 10(4):1779--1812, 2016.

\bibitem[Led06]{ledoux2006concentration}
Michel Ledoux.
\newblock Concentration of measure and logarithmic sobolev inequalities.
\newblock In {\em Seminaire de probabilites XXXIII}, pages 120--216. Springer, 2006.

\bibitem[LKC20]{lee2020robust}
Jeonghwan Lee, Daesung Kim, and Hye~Won Chung.
\newblock Robust hypergraph clustering via convex relaxation of truncated mle.
\newblock {\em IEEE Journal on Selected Areas in Information Theory}, 1(3):613–631, November 2020.

\bibitem[LLR83]{leadbetter1983extreme}
M.~R. Leadbetter, Georg Lindgren, and Holger Rootz\'en.
\newblock {\em Extremes and related properties of random sequences and processes}.
\newblock Springer, 1983.

\bibitem[LP12]{lu2012loose}
Linyuan Lu and Xing Peng.
\newblock Loose laplacian spectra of random hypergraphs.
\newblock {\em Random Structures \& Algorithms}, 41(4):521--545, 2012.

\bibitem[LQY13]{li2013z}
Guoyin Li, Liqun Qi, and Gaohang Yu.
\newblock The z-eigenvalues of a symmetric tensor and its application to spectral hypergraph theory.
\newblock {\em Numerical Linear Algebra with Applications}, 20(6):1001--1029, 2013.

\bibitem[McK81]{mckay1981expected}
Brendan~D. McKay.
\newblock The expected eigenvalue distribution of a large regular graph.
\newblock {\em Linear Algebra and its Applications}, 40:203--216, 1981.

\bibitem[MN12]{michoel2012alignment}
Tom Michoel and Bruno Nachtergaele.
\newblock Alignment and integration of complex networks by hypergraph-based spectral clustering.
\newblock {\em Physical Review E—Statistical, Nonlinear, and Soft Matter Physics}, 86(5):056111, 2012.

\bibitem[MV23]{mann2023ai}
Vipul Mann and Venkat Venkatasubramanian.
\newblock Ai-driven hypergraph network of organic chemistry: network statistics and applications in reaction classification.
\newblock {\em Reaction Chemistry \& Engineering}, 8(3):619--635, 2023.

\bibitem[Pas72]{pastur1972}
L.~A. Pastur.
\newblock On the spectrum of random matrices.
\newblock {\em Theoretical and Mathematical Physics}, 10(1):67–74, January 1972.

\bibitem[PS11]{pastur2011eigenvalue}
Leonid~Andreevich Pastur and Mariya Shcherbina.
\newblock {\em Eigenvalue distribution of large random matrices}.
\newblock American Mathematical Soc., 2011.

\bibitem[PV00]{pastur2000law}
L~Pastur and V~Vasilchuk.
\newblock On the law of addition of random matrices.
\newblock {\em Communications in Mathematical Physics}, 214:249--286, 2000.

\bibitem[PZ14]{pearson2014spectral}
Kelly~J Pearson and Tan Zhang.
\newblock On spectral hypergraph theory of the adjacency tensor.
\newblock {\em Graphs and Combinatorics}, 30:1233--1248, 2014.

\bibitem[PZ21]{pal2021community}
Soumik Pal and Yizhe Zhu.
\newblock Community detection in the sparse hypergraph stochastic block model.
\newblock {\em Random Structures \& Algorithms}, 59(3):407--463, 2021.

\bibitem[Rg02]{rodri2002laplacian}
Juan~A Rodri{\'{}}~guez.
\newblock On the laplacian eigenvalues and metric parameters of hypergraphs.
\newblock {\em Linear and Multilinear Algebra}, 50(1):1--14, 2002.

\bibitem[Rod03]{rodriguez2003laplacian}
Juan~Alberto Rodriguez.
\newblock On the laplacian spectrum and walk-regular hypergraphs.
\newblock {\em Linear and Multilinear Algebra}, 51(3):285--297, 2003.

\bibitem[Rod09]{rodriguez2009laplacian}
JA~Rodriguez.
\newblock Laplacian eigenvalues and partition problems in hypergraphs.
\newblock {\em Applied Mathematics Letters}, 22(6):916--921, 2009.

\bibitem[SFM14]{skvortsova2014hypergraph}
MI~Skvortsova, II~Fashutdinova, and NA~Mikhailova.
\newblock Hypergraph models of hydrocarbon molecules and their applications in computer chemistry.
\newblock {\em Fine Chemical Technologies}, 9(5):86--93, 2014.

\bibitem[SGC24]{seddik2024random}
Mohamed El~Amine Seddik, Maxime Guillaud, and Romain Couillet.
\newblock When random tensors meet random matrices.
\newblock {\em The Annals of Applied Probability}, 34(1A):203--248, 2024.

\bibitem[SSP22]{sahalapspectra2022}
S.~S. Saha, K.~Sharma, and S.~K. Panda.
\newblock On the {L}aplacian spectrum of {$k$}-uniform hypergraphs.
\newblock {\em Linear Algebra Appl.}, 655:1--27, 2022.

\bibitem[THK09]{tian2009hypergraph}
Ze~Tian, TaeHyun Hwang, and Rui Kuang.
\newblock A hypergraph-based learning algorithm for classifying gene expression and arraycgh data with prior knowledge.
\newblock {\em Bioinformatics}, 25(21):2831--2838, 2009.

\bibitem[TW96]{tracy1996extremeeigen}
C.~Tracy and H~Widom.
\newblock On orthogonal and sympletic matrix ensembles.
\newblock {\em Communications in Mathematical Physics}, 177:727--754, 1996.

\bibitem[Wig58]{wigner1958distribution}
Eugene~P. Wigner.
\newblock On the distribution of the roots of certain symmetric matrices.
\newblock {\em Ann. of Math. (2)}, 67:325--327, 1958.

\bibitem[XC13a]{xie2013zadj}
Jinshan Xie and An~Chang.
\newblock On the z-eigenvalues of the adjacency tensors for uniform hypergraphs.
\newblock {\em Linear Algebra and its Applications}, 439(8):2195--2204, 2013.

\bibitem[XC13b]{xie2013z}
Jinshan Xie and An~Chang.
\newblock On the z-eigenvalues of the signless laplacian tensor for an even uniform hypergraph.
\newblock {\em Numerical Linear Algebra with Applications}, 20(6):1030--1045, 2013.

\end{thebibliography}

\appendix
\section{Auxiliary results}\label{sec:aux}
Here we collect lemmas and results borrrowed from the literature. First we define some notations.
\[
    \Mat_n(\bbC) := \text{The set of all } n \times n \text{ matrices with complex entries}.
\]
For $A \in \Mat_n(\bbC) $, define the Frobenius norm of $A$ by
\[
    \|A\|_{\mathrm{F}} := \sqrt{\sum_{1 \le i \le j \le n}|A_{ij}|^2}.
\]
For $x \in \bbR^n $, let $\|x\| = \sqrt{\sum_{i = 1}^n x^2_i}$. The Operator norm of $A$ is defined as
\[
\|A\|_{\op} := \sup_{\|x\| = 1} \|Ax\|.
\]
For a random matrix $A$ with eigenvalues $\lambda_1, \ldots, \lambda_n$, let $F_{A}(x) := \frac{1}{n} \sum_{i = 1}^n \ind(\lambda_i \le x)$ be the empirical distribution function associated with the eigenvalues.  Let $\cS_n$ denotes the set of all permutations of the set $\{1, 2, \ldots, n\}$. 
\begin{lemma}[Hoffmann-Wielandt inequality]\label{lem:hoffman-wielandt}
  Let $A,B \in \Mat_n(\bbC) $ are two normal matrices, with eigenvalues $\lambda_1(A),\lambda_2(A), \ldots, \lambda_n(A)$ and $\lambda_1(B),\lambda_2(B), \ldots, \lambda_n(B)$ respectively. Then we have
  \[
    \min_{\sigma \in \cS_n} \sum_{i = 1}^n | \lambda_i(A) - \lambda_{\sigma(i)}(B)|^2 \le \|A - B \|^2_{\mathrm{F}}.
  \]
  An immediate consequence of this is that
  \[
        d_{W_2}(\mu_A, \mu_B)^2 \le \frac{\|A - B\|^2}{n}.
  \]
\end{lemma}
\begin{lemma}[Rank inequality]\label{lem:rank_ineq}
 Let $A, B \in \Mat_n(\bbC)  $ are two Hermitian matrices. Then,
 \[
     \sup_{x \in \bbR} |F_A(x) - F_B(x)| \le \frac{\mathrm{rank}(A - B)}{n}.
 \]
\end{lemma}
\begin{lemma}[Weyl's inequality]\label{lem:weyl_inequality}
    Let $A, B \in \Mat_n(\bbC)$ be two Hermitian matrices with decreasing sequence of eigenvalues $\lambda_1(A), \lambda_2(A), \ldots, \lambda_n(A)$ and $\lambda_1(B), \lambda_2(B), \ldots, \lambda_n(B)$, respectively. Then, for $i \in [n]$,
    \[
        \lambda_{j'}(A) + \lambda_{i-j'+n}(B) \leq \lambda_i(A+B) \leq \lambda_j(A) + \lambda_{i-j+1}(B)
    \]
    for any $j\leq i$ and $j' \geq i$. A consequence of this is that for any $1 \le i \le n$,
    \[
        |\lambda_i(A + B) - \lambda_i(A)| \le \max \{|\lambda_1(B)|, |\lambda_n(B)|\} = \|B\|_{\op}.
    \]
\end{lemma}
The following lemma is an amalgamation of Theorems~1.5.3 and ~2.2.2 from \cite{leadbetter1983extreme}.
\begin{lemma}[Extreme order statistics of Gaussians]\label{lem:order_stat_gaussian}
    If $\xi_1, \xi_2, \ldots, \xi_n$ are i.i.d. standard normal variables, then for $k \in \bbN$,
    \begin{equation*}
        \bbP\bigg[\sqrt{2 \log n}\bigg(\xi_{(k)} - \sqrt{2 \log n} + \frac{\log \log n + \log (4\pi)}{2 \sqrt{2 \log n}}\bigg)\leq t \bigg] \to e^{-e^{-x}}\sum_{j=0}^{k-1} \frac{e^{-jx}}{j!},
    \end{equation*}
where $\xi_{(k)}$ denotes the $k$-th largest order statistic of $\xi_1, \xi_2, \ldots, \xi_n$. In particular, 
\begin{equation}\label{eq:order_stat_gaussian}
    2 \log n\bigg(\frac{\xi_{(k)}}{\sqrt {2 \log n}}-1 +\frac{\log \log n}{4 \log n}\bigg) = O_P(1).
\end{equation}
\end{lemma}
Since the $\xi_i$'s are symmetric, a similar result holds for $\xi_{n+1-k}$, modulo the obvious sign change in the centering parameter.

\section{Miscellaneous results and proofs}\label{sec:more_proofs}
We can adapt Lemma 4.12 of \cite{bryc2005spectral} to our setting. 
\begin{lemma}\label{lem:BDJ_replacement}
Let 
\[
    \hat{W}_n = \frac{\sqrt{n(r - 2)}}{r - 1} \diag(\bV) + \frac{\diag(Z_n\bone)}{r - 1} - Z_n
\]
and
\[
    \tilde{W}_n = \frac{\sqrt{n(r - 2)}}{r - 1} \diag(\bV) + \frac{\sqrt{n + 1} \, \diag(\bg)}{r - 1} + Z_n,
\]
where $n^{-1/2} Z_n$ is a GOE random matrix and $\bV$ and $\bg$ are vectors of i.i.d. standard Gaussian random variables. Further, $\bV$, $\bg$ and $Z_n$ are independent. Then, for fixed $r$ and every $k \in \bbN$,
\[
    \lim_{n \to \infty} n^{-(k + 1)} [ \bbE \Tr(\hat{W}_n^{2k}) - \bbE \Tr(\tilde{W}_n^{2k})] = 0.
\]
\end{lemma}
\begin{proof}
Let $\bV = \bV_n$ and $\bV_{n + 1} = \begin{pmatrix}
    \bV'_n \\
    V_{n + 1}
\end{pmatrix}$ where $\bV'_n$ is an i.i.d copy of $\bV_n$ and $V_{n + 1}$ is a standard Gaussian independent of $\bV'_n$. Let $Z_{n + 1}$ be the matrix defined by 
\[
Z_{n + 1} = \begin{pmatrix}
     Z'_n & \bZ \\
     \bZ^\top & Z^1
\end{pmatrix},
\]
where $ Z'_n$ is an i.i.d copy of $Z_n$, $\bZ$ is an $ n \times 1$ vector with i.i.d. Gaussian entries and $Z^1$ is a mean 0 variance 2 Gaussian random variable. Now consider the matrix 
\[
    W'_n =  \frac{\sqrt{n(r - 2)}}{r - 1} \diag(\bV'_n) + \frac{\diag(Z'_n \bone)}{r - 1} - Z_n.
\]
Then $W'_n$ and $\tilde{W}_n$ has same distribution. So, it is enough to prove that 
\begin{align}\label{eq:limit_W_and_W'}
    \lim_{n \to \infty} n^{-(k + 1)} [ \bbE \Tr(\hat{W}_n^{2k}) - \bbE \Tr((W'_n)^{2k})] = 0.
\end{align}
Now
\[
    \bbE [\Tr(\hat{W}_n^{2k}) - \Tr((W'_n)^{2k})] = \sum_{\pi} [\bbE \hat{W}_{\pi} - \bbE W'_{\pi}]',
\]
where the sum is over all circuits $\pi : \{0, \ldots, 2k \} \to [n]$ with $\pi(0) = \pi(2k)$ and
\[
    \hat{W}_{\pi} = \prod_{i = 1}^{2k} \hat{W}_{\pi(i - 1), \pi(i)}.
\]
A word $w$ of length $k$ is a sequence of numbers of length $k$, e.g., $12234$ is a word of length $5$. Set each word $w$ of length $2k$ to be a circuit assigning $w[0] = w[2k]$. We associate a word with each circuit of length $2k$ via the relation $ w[i] = \pi(i), 1 \le i \le 2k$. Let $\Pi(w)$ denote the collection of circuits $\pi$ such that the distinct letters of $w$ are in a one-to-one correspondence with the distinct values of $\pi$. Let $v(w)$ be the number of distinct letters in the word $w$. We note that $\# \Pi(w) \le n^{v(w)}$. Define $ f_n(w) := \bbE \hat{W}_{w} - \bbE W'_{w}$. We show that, for any word $w$, there exits a constant $C_{w} > 0$ such that for all $n \ge 1$, 
\begin{align}\label{eq: est_f_n}
    |f_n(w)| = |\bbE \hat{W}_{w} - \bbE{W'_{w}}| \le C_{w} \bigg(\frac{n}{r - 1} + 3\bigg)^{k - v(w) + 1/2}.  
\end{align}
Let $q = q(w)$ be the number of indices $1 \le i \le 2k$ for which $w[i] = w[i - 1]$. If $q(w) = 0$ then $\hat{W}_{w}$ is product of only off-diagonal entries of $\hat{W}_n$ and off-diagonal entries of $\hat{W}_n$ and $W'_n$ are same. Thus $f_n(w) = 0$ and $f_n(w) \neq 0$ only if $q(w) \ge 1$. Let $\cG_w$ be the graph with vertex set as the distinct letters of $w$ and edge set $\{ (w[i - 1], w[i] ) : 1 \le i \le 2k \}$.  Let $u = u(w)$ be the number of edges of $\cG_w$ with distinct endpoints in $w$, which appear exactly once along the circuit $w$, e.g $u(12234) = 4$. Then by independence of the entries of $W'_n$ we have that $\bbE W'_{w} = 0$ as soon as $u(w) \ge 1$. For $u(w) > q(w)$, let there are $b$ number of off-diagonal entries in $\hat{W}_{w}$ and $( 2k - b)$ diagonal entries. Then $ b > 2k - b$. Since $ u(w) > 1$, among these $b$ off-diagonal entries there will be at least one entry, say $\hat{W}_{ij}$ such that the edge $\{i,j\}$ is traversed exactly once in $\cG_w$. By Wick's formula $\bbE \hat{W}_{w} = 0$. Thus, to prove \eqref{eq: est_f_n} it is enough to consider $q(w) \ge u(w)$. 

It is easy to check that excluding the $q$ loop-edges (each vertex connecting to itself), there are at most $ k + \lfloor (u - q)/2 \rfloor$ distinct edges in $w$. These distinct edges form a connected path through $v(w)$ vertices, which for $u \ge 1$ must also be a circuit. From the proof of Lemma~4.12 of \cite{bryc2005spectral}, we have
\begin{equation}\label{eq:est_v(w)}
    v(w) \le k + \ind( u(w) = 0) + \lfloor (u(w) - q(w) )/2 \rfloor \le k. 
\end{equation}
Proceeding to bound $|f_n(w)|$, suppose first that $u \ge 1$, in which case $f_n(w) = \bbE \hat{W}_w$. To compute this expectation we employ Wick's formula. Note that if we match $u$ off-diagonal entries with the $u$ diagonal entries with which they are correlated, and remaining $( q - u ) $ diagonal entries are either self matched, or they can match among themselves, then only contribution to the expectation will be non-zero. Other matching configurations will result in 0 due to the independence of the entries, e.g for the word $w = 11223333$, we have $\bbE \hat{W}_w = \bbE [\hat{W}_{11} \hat{W}_{12}] \bbE [\hat{W}_{22} \hat{W}_{23}]  \bbE [\hat{W}_{33} \hat{W}_{33}]  \bbE [\hat{W}_{33} \hat{W}_{31}]$. Further, observe that covariance terms, i.e  $\bbE[ \hat{W}_{ii} \hat{W}_{ij} ] = \frac{1}{r - 1}$, $\bbE[ \hat{W}_{ij} \hat{W}_{jj}] = \frac{1}{( r - 1)^2}$ and $\bbE[\hat{W}^2_{ii}] = \frac{n}{r - 1} + \frac{1}{(r - 1)^2} + 2 - \frac{4}{r - 1} < \frac{n}{r - 1} + 3$, for any $ 1 \le i \le n$. Then 
\begin{align*}
    |f_n(w)| &\le C^3_w\bigg(\frac{1}{r - 1}\bigg)^u \bigg[ C^1_w\bigg( \frac{1}{(r - 1)^2} \bigg)^{\frac{( q - u )}{2}} + C^2_w\bigg(\frac{n + r - 3}{r - 1} + \bigg( \frac{r - 2}{r - 1}\bigg)^2\bigg)^{\frac{( q - u )}{2}} \bigg] \\
    &\le C_w \bigg[ 1 + C'_w \bigg( \frac{n}{r - 1} + 2 \bigg)^{\frac{( q - u )}{2}} \bigg].
\end{align*}

By our bound \eqref{eq:est_v(w)} on $v(w)$, this implies that \eqref{eq: est_f_n} holds.

Consider next words $w$ for which $u(w) = 0$.  There will be some diagonal entries and some off-diagonal entries which can appear twice or more in $f_n ( w)$. Let $ a_1, \ldots, a_q$ be the $q$ vertices for which $\{ a_i , a_i \}$ is an edge of $\cG_w$. Then 
\[
 \hat{W}_w = \prod_{i = 1}^q \hat{W}_{a_i, a_i} \hat{V}_w, 
\]
where $V_w$ is the product of $( 2k - q )$ off-diagonal entries of $\hat{W}_n$ that correspond to the edges of $w$ that are in $\cG_w$. Similarly, we can write
\[
    W'_w = \prod_{i = 1}^q W'_{a_i , a_i} V'_w.
\]
Since off-diagonal entries of $\hat{W}_n$ are $W'_n$ are same, we can replace $\hat{V}_w$ by $V'_w$ in $\hat{W}_w$.
Let $\hat{W}_{ii} = \hat{D}_i + \hat{S}_i$ and $W'_{ii} = D'_i + S'_i$, for $i = 1,\ldots, 2k$, where $\hat{D}_i = \frac{\sqrt{n(r - 2)}}{r - 1} \bV_i + \frac{1}{r - 1} \sum_{j = 1}^{2k} Z_{ij} - Z_{ii}$, $D'_i = \frac{\sqrt{n(r - 2)}}{r - 1} \bV'_i + \frac{1}{r - 1} \sum_{j = 1}^{2k} Z'_{ij} - Z_{ii}$ and $\hat{S}_i = \sum_{j = 2k + 1}^n Z_{ij}$ and $S'_i = \sum_{j = 2k + 1}^n Z'_{ij}$. Note that we may and shall replace each $\hat{S}_i$ by $S'_i$ without altering $\bbE \hat{W}_{w}$, as $Z_{ij}$ and $Z'_{ij}$ have same distribution. 
\begin{align*}
   f_n(w) &= \bbE \bigg[ V'_w \bigg[ \prod_{i = 1}^q ( \hat{D}_{a_i} + S'_{a_i}) - \prod_{i = 1}^q  (D'_{a_i} + S'_{a_i}) \bigg]\bigg]\\
   &= \sum_{i = 1}^q \bbE \bigg[ V'_w ( D_{a_i} - D'_{a_i} ) \prod_{j = 1}^{i - 1} \hat{W}_{a_j, a_j} \prod_{j = i + 1}^q W'_{a_j, a_j} \bigg] \\
   &\le \sum_{i = 1}^q (\bbE[V'^4_{w}])^{1/4} \bbE[( D_{a_i} - D'_{a_i} )^4 ]^{1/4} \sqrt{\bbE\bigg[ \prod_{j = 1}^{i - 1} \hat{W}^2_{a_j, a_j} \prod_{j = i + 1}^q W'^2_{a_j, a_j} \bigg]}.
\end{align*}
Note that $V'_w$ is product of some i.i.d standard Gaussians, so $\bbE[V'^4_{w}]$ is constant. For each $i$, $D_{a_i} - D'_{a_i}$ is a Gaussian random variable with mean 0 variance $\frac{2n(r - 2)}{( r - 1)^2} + \frac{4k}{(r - 1)^2} = O(n)$ , so $\bbE[( D_{a_i} - D'_{a_i} )^4 ]^{1/4} = O( \sqrt{n}) $. We know $\bbE[ W'^2_{ii}] = \frac{n}{r - 1} + \frac{1}{(r - 1)^2} + 2 < \frac{n}{r - 1} + 3$ and $\bbE[\hat{W}^2_{ii}] = \frac{n}{r - 1} + \frac{1}{(r - 1)^2} + 2 - \frac{4}{r - 1} < \frac{n}{r - 1} + 3$ and $\bbE[ \hat{W}_{ii} W_{jj}] = - \frac{2}{r - 1} + 2$. By Wick's formula, we have 
\[
    \bbE\bigg[ \prod_{j = 1}^{i - 1} \hat{W}^2_{a_j, a_j} \prod_{j = i + 1}^q W'^2_{a_j, a_j} \bigg] = O\bigg(\bigg( \frac{n}{r - 1} + 3 \bigg)^q\bigg).
\]
Hence
\[
    f_n(w) \le C_w \sqrt{n} \bigg( \frac{n}{r - 1} + 3 \bigg)^{q/2}.
\]
It is clear that \eqref{eq: est_f_n} implies that \eqref{eq:limit_W_and_W'} holds and hence the proof of the lemma is complete.
\end{proof}

\begin{proof}[Proof of Proposition~\ref{prop:bulk_universality_laplacian_orig}]
Let $g(\bx) = \hat{G}_n(\bx)$. As before, we have to control $\lambda_2(g)$ and $\lambda_3(g)$. We have the identities:
\begin{align}
    \frac{\partial g(\bx) }{\partial x_{\ell}} &= - \frac{1}{n^{3 / 2} r^{1/2}} \Tr\bigg(\frac{\partial L_{H_n}(\bx)}{\partial x_{\ell}} \hat{R}_n^2(\bx)\bigg), \\
    \frac{\partial^2 g(\bx)}{\partial x^2_{\ell}} &= \frac{2}{n^2 r} \Tr\bigg(\frac{\partial L_{H_n}(\bx)}{\partial x_{\ell}} \hat{R}_n(\bx) \frac{\partial L_{H_n}(\bx)}{\partial x_{\ell}} \hat{R}_n^2(\bx)\bigg), \\
    \frac{\partial^3 g(\bx) }{\partial x^3_{\ell}} &= - \frac{6}{n^{5/2} r^{3/2}} \Tr\bigg(\frac{\partial L_{H_n}(\bx)}{\partial x_{\ell}} \hat{R}_n(\bx) \frac{\partial L_{H_n}(\bx)}{\partial x_{\ell}} \hat{R}_n(\bx) \frac{\partial L_{H_n}(\bx)}{\partial x_{\ell}} \hat{R}_n^2(\bx)\bigg).
\end{align}
Note that
\[
    \diag(H_n(\bx) \bone) = \diag\bigg(\frac{1}{\sqrt{N}} \sum_{\ell = 1}^M x_{\ell} Q_{\ell} \bone\bigg) = \frac{1}{\sqrt{N}} \sum_{\ell = 1}^M x_{\ell} \, \diag(Q_{\ell} \bone).
\]
Hence 
\begin{equation}\label{eq:lap_deriv_1}
    \frac{\partial L_{H_n}(\bx)}{\partial x_{\ell}} = \frac{1}{\sqrt{N}} \bigg(\diag(Q_{\ell} \bone) - Q_{\ell}\bigg) = \frac{1}{\sqrt{N}} \begin{pmatrix}
        r I_r - J_r & 0 \\
        0 & 0
    \end{pmatrix},
\end{equation}
where we have assumed (after a relabeling of nodes if necessary) that
\[
    Q_{\ell} = \begin{pmatrix}
        J_r - I_r & 0 \\
        0 & 0
    \end{pmatrix}.
\]
Now
\[
    \bigg|\Tr\bigg(\frac{\partial L_{H_n}(\bx)}{\partial x_{\ell}} \hat{R}_n^2(\bx)\bigg)\bigg| \le \frac{1}{v^2} \sum_{i,j} \bigg| \bigg( \frac{\partial L_{H_n}(\bx)}{\partial x_{\ell}} \bigg)_{ij}\bigg| = \frac{1}{v^2} \frac{2 r (r - 1)}{\sqrt{N}} \le \frac{1}{v^2} \frac{2 r^2}{\sqrt{N}},
\]
and a fortiori,
\[
    \bigg\| \frac{\partial g(\bx) }{\partial x_{\ell}} \bigg\|_{\infty} \le \frac{2}{v^2} \frac{r^{3/2}}{n^{3/2}\sqrt{N}}.
\]
From \eqref{eq:lap_deriv_1} we also deduce that
\[
    \bigg\|\frac{\partial L_{H_n}(\bx)}{\partial x_{\ell}} \bigg\|^2_{F} = \frac{r^2(r - 1)}{N} \le \frac{r^3}{N} \quad \text{and} \quad \bigg\|\frac{\partial L_{H_n}(\bx)}{\partial x_{\ell}} \bigg\|_{\op} = \frac{r}{\sqrt{N}}.
\]
Therefore
\[
    \Tr\bigg(\frac{\partial L_{H_n}(\bx)}{\partial x_{\ell}} \hat{R}_n(\bx) \frac{\partial L_{H_n}(\bx)}{\partial x_{\ell}} \hat{R}_n^2(\bx)\bigg) \le \bigg\|\frac{\partial L_{H_n}(\bx)}{\partial x_{\ell}}  \bigg\|^2_{F} \cdot \|\hat{R}_n(\bx)\|^3_{\op} \le \frac{r^3}{N} \frac{1}{v^3}
\]
and so
\[
    \bigg\| \frac{\partial^2 g(\bx)}{\partial x^2_{\ell}} \bigg\|_{\infty} \le \frac{2r^2}{n^2Nv^3}.
\]
Finally,
\begin{align*}
    \Tr\bigg(\frac{\partial L_{H_n}(\bx)}{\partial x_{\ell}} \hat{R}_n(\bx) \frac{\partial L_{H_n}(\bx)}{\partial x_{\ell}} \hat{R}_n(\bx) \frac{\partial L_{H_n}(\bx)}{\partial x_{\ell}} \hat{R}_n^2(\bx)\bigg) &\le \bigg\|\frac{\partial L_{H_n}(\bx)}{\partial x_{\ell}} \bigg\|_{\op} \cdot \bigg\|\frac{\partial L_{H_n}(\bx)}{\partial x_{\ell}} \bigg\|_{F}^2 \cdot \|\hat{R}_n(\bx)\|_{\op}^4 \\
    & \le \frac{r}{\sqrt{N}} \cdot \frac{r^3}{N} \cdot \frac{1}{v^4}
\end{align*}
and therefore
\[
    \bigg\| \frac{\partial^3 g(\bx) }{\partial x^3_{\ell}} \bigg\|_{\infty} \le \frac{6 r^{5 / 2}}{n^{5/2} N^{3/2} v^4}.
\]
We conclude that
\[
    \lambda_2(g) \le 4 \max( v^{-4}, v^{-3}) \frac{r^3}{n^2 N}
\]
and
\[
    \lambda_3(g) \le 8 \max( v^{-6}, v^{-\frac{9}{2}}, v^{-4}) \frac{r^{9/2}}{n^{5/2} N^{3/2}}.
\]
Hence 
\begin{align*}
    |\bbE(\hat{G}_n&(\mathbf{Y})) - \bbE(\hat{G}_n(\mathbf{Z}))| \\ \nonumber
    &\le 8 \max(v^{-3}, v^{-4}) \frac{r^3}{n^2 N} \sum_{\ell = 1}^M \bigg[ \bbE[Y^2_{\ell}\ind(|Y_{\ell}| > K)] + \bbE[Z^2_{\ell}\ind(|Z_{\ell}| > K)] \bigg] \\
    &\quad + 16 \max(v^{-6}, v^{-\frac{9}{2}}, v^{-4}) \frac{r^{9/2}}{n^{5/2} N^{3/2}} \sum_{\ell = 1}^M \bigg[ \bbE[|Y_{\ell}|^3\ind(|Y_{\ell}| \le K)] + \bbE[|Z_{\ell}|^3\ind(|Z_{\ell}| \le K)] \bigg].
\end{align*}
This completes the proof.
\end{proof}

\begin{proof}[Proof of Proposition~\ref{prop:bulk_universality_laplacian}]
Let $g(\bx) = \tilde{G}_n(\bx)$. As before, we have to control $\lambda_2(g)$ and $\lambda_3(g)$. We have the identities:
\begin{align}
    \frac{\partial g(\bx) }{\partial x_{\ell}} &= - \frac{1}{n^{\frac{3}{2}}} \Tr\bigg(\frac{\partial \tilde{L}_{H_n}(\bx)}{\partial x_{\ell}} \tilde{R}_n^2(\bx)\bigg), \\
    \frac{\partial^2 g(\bx)}{\partial x^2_{\ell}} &= \frac{2}{n^2} \Tr\bigg( \frac{\partial \tilde{L}_{H_n}(\bx)}{\partial x_{\ell}} \tilde{R}_n(\bx) \frac{\partial \tilde{L}_{H_n}(\bx)}{\partial x_{\ell}} \tilde{R}_n^2(\bx)\bigg), \\
    \frac{\partial^3 g(\bx) }{\partial x^3_{\ell}} &= - \frac{6}{n^{\frac{5}{2}}} \Tr\bigg(\frac{\partial \tilde{L}_{H_n}(\bx)}{\partial x_{\ell}} \tilde{R}_n(\bx) \frac{\partial \tilde{L}_{H_n}(\bx)}{\partial x_{\ell}} \tilde{R}_n(\bx) \frac{\partial \tilde{L}_{H_n}(\bx)}{\partial x_{\ell}} \tilde{R}_n^2(\bx)\bigg).
\end{align}
Note that
\[
    \diag(H_n(\bx) \bone) = \diag\bigg(\frac{1}{\sqrt{N}} \sum_{\ell = 1}^M x_{\ell} Q_{\ell} \bone\bigg) = \frac{1}{\sqrt{N}} \sum_{\ell = 1}^M x_{\ell} \, \diag(Q_{\ell} \bone).
\]
Hence 
\begin{equation}\label{eq:lap_deriv_2}
    \frac{\partial \tilde{L}_{H_n}(\bx)}{\partial x_{\ell}} = \frac{1}{\sqrt{N}} \bigg( \frac{\diag(Q_{\ell} \bone)}{r - 1} - Q_{\ell}\bigg) = \frac{1}{\sqrt{N}} \begin{pmatrix}
        2I_r - J_r & 0 \\
        0 & 0
    \end{pmatrix},
\end{equation}
where we have assumed (after a relabeling of nodes if necessary) that
\[
    Q_{\ell} = \begin{pmatrix}
        J_r - I_r & 0 \\
        0 & 0
    \end{pmatrix}.
\]
Now
\[
    \bigg|\Tr\bigg(\frac{\partial \tilde{L}_{H_n}(\bx)}{\partial x_{\ell}} \tilde{R}_n^2(\bx)\bigg)\bigg| \le \frac{1}{v^2} \sum_{i,j} \bigg| \bigg( \frac{\partial \tilde{L}_{H_n}(\bx)}{\partial x_{\ell}} \bigg)_{ij}\bigg| =\frac{1}{v^2} \frac{r^2}{\sqrt{N}},
\]
and a fortiori,
\[
    \bigg\| \frac{\partial g(\bx) }{\partial x_{\ell}} \bigg\|_{\infty} \le \frac{1}{v^2} \frac{r^2}{n^{3/2}\sqrt{N}}.
\]
From \eqref{eq:lap_deriv_2} we also deduce that
\[
    \bigg\|\frac{\partial \tilde{L}_{H_n}(\bx)}{\partial x_{\ell}} \bigg\|^2_{F} = \frac{r^2}{N} \quad \text{and} \quad \bigg\|\frac{\partial \tilde{L}_{H_n}(\bx)}{\partial x_{\ell}} \bigg\|_{\op} = \frac{\max\{2, r - 2\}}{\sqrt{N}} \le \frac{r}{\sqrt{N}}.
\]
Therefore
\[
    \Tr\bigg(\frac{\partial \tilde{L}_{H_n}(\bx)}{\partial x_{\ell}} \tilde{R}_n(\bx) \frac{\partial \tilde{L}_{H_n}(\bx)}{\partial x_{\ell}} \tilde{R}_n^2(\bx)\bigg) \le \bigg\|\frac{\partial \tilde{L}_{H_n}(\bx)}{\partial x_{\ell}}  \bigg\|^2_{F} \cdot \|\tilde{R}_n(\bx)\|^3_{\op} \le \frac{r^2}{N} \frac{1}{v^3}
\]
and so
\[
    \bigg\| \frac{\partial^2 g(\bx)}{\partial x^2_{\ell}} \bigg\|_{\infty} \le \frac{2r^2}{n^2Nv^3}.
\]
Finally,
\begin{align*}
    \Tr\bigg(\frac{\partial \tilde{L}_{H_n}(\bx)}{\partial x_{\ell}} \tilde{R}_n(\bx) \frac{\partial \tilde{L}_{H_n}(\bx)}{\partial x_{\ell}} \tilde{R}_n(\bx) \frac{\partial \tilde{L}_{H_n}(\bx)}{\partial x_{\ell}} \tilde{R}_n^2(\bx)\bigg) &\le \bigg\|\frac{\partial \tilde{L}_{H_n}(\bx)}{\partial x_{\ell}} \bigg\|_{\op} \cdot \bigg\|\frac{\partial \tilde{L}_{H_n}(\bx)}{\partial x_{\ell}} \bigg\|_{F}^2 \cdot \|\tilde{R}_n(\bx)\|_{\op}^4 \\
    & \le \frac{r}{\sqrt{N}} \cdot \frac{r^2}{N} \cdot \frac{1}{v^4}
\end{align*}
and therefore
\[
    \bigg\| \frac{\partial^3 g(\bx) }{\partial x^3_{\ell}} \bigg\|_{\infty} \le \frac{6 r^3}{n^{5/2} N^{3/2} v^4}.
\]
We conclude that
\[
    \lambda_2(g) \le 2\max( v^{-4}, v^{-3}) \frac{r^4}{n^2 N}
\]
and
\[
    \lambda_3(g) \le 6 \max( v^{-6}, v^{-\frac{9}{2}}, v^{-4}) \frac{r^6}{n^{5/2} N^{3/2}}.
\]
Hence 
\begin{align*}
    |\bbE(\tilde{G}_n&(\mathbf{Y})) - \bbE(\tilde{G}_n(\mathbf{Z}))| \\ \nonumber
    &\le 4 \max(v^{-3}, v^{-4}) \frac{r^4}{n^2 N} \sum_{\ell = 1}^M \bigg[ \bbE[Y^2_{\ell}\ind(|Y_{\ell}| > K)] + \bbE[Z^2_{\ell}\ind(|Z_{\ell}| > K)] \bigg] \\
    &\quad + 12 \max(v^{-6}, v^{-\frac{9}{2}}, v^{-4}) \frac{r^6}{n^{5/2} N^{3/2}} \sum_{\ell = 1}^M \bigg[ \bbE[|Y_{\ell}|^3\ind(|Y_{\ell}| \le K)] + \bbE[|Z_{\ell}|^3\ind(|Z_{\ell}| \le K)] \bigg].
\end{align*}
This completes the proof.
\end{proof}
\end{document}